\documentclass[11pt,reqno]{amsart}
\usepackage[cmtip,all]{xy}
\usepackage{graphicx}
\usepackage{amssymb}
\usepackage{mathrsfs}
\usepackage{amsfonts}
\usepackage{cases}
\usepackage{amsmath}
\usepackage{graphicx}
\usepackage{color}
\usepackage{subfig}
\usepackage{enumitem}
\usepackage{hyperref}
\usepackage{appendix}
\usepackage{cite}

\newtheorem{theorem}{Theorem}[section]
\newtheorem{lem}[theorem]{Lemma}
\newtheorem{prop}[theorem]{Proposition}
\newtheorem{corollary}[theorem]{Corollary}

\theoremstyle{definition}
\newtheorem{definition}[theorem]{Definition}

\theoremstyle{remark}
\newtheorem{remark}[theorem]{Remark}

\numberwithin{equation}{section}

\def\bsi{{\mathrm{i}}}
\def\Oh{{\mathcal  O}}
\def\bsi{{\mathrm  i}}
\def \rmd {{\mathrm  d}}
\allowdisplaybreaks[4]

\begin{document}

\title[Pattern formations of coupled PDEs and applications]{Pattern formations of coupled PDE\lowercase{s} with transparent boundary conditions in product-type ends and applications} 

\author{Huaian Diao}
\address{School of Mathematics and Key Laboratory of Symbolic Computation and Knowledge Engineering of Ministry of Education, Jilin University, Changchun, China}
\email{hadiao@gmail.com, diao@jlu.edu.cn}

\author{Hongyu Liu}
\address{Department of Mathematics, City University of Hong Kong, Hong Kong SAR, China}
\email{hongyu.liuip@gmail.com, hongyliu@cityu.edu.hk}

\author{Qingle Meng}
\address{Department of Mathematics, City University of Hong Kong, Hong Kong SAR, China}
\email{mengq12021@foxmail.com, qinmeng@cityu.edu.hk}

\author{Li Wang}
\address{Department of Mathematics, City University of Hong Kong, Hong Kong SAR, China}
\email{liwangmath12@126.com}

\thanks{}
\date{}


\begin{abstract}
This paper studies pattern formations in coupled elliptic PDE systems governed by transparent boundary conditions. Such systems unify diverse areas, including inverse boundary problems (via a single passive/active boundary measurement), spectral geometry of transmission eigenfunctions, and geometric characterization of invisibility phenomena and inverse shape problems in wave scattering. We uncover and rigorously characterize a novel local pattern formation, establishing a sharp quantitative relationship between the difference in the PDEs' lower-order terms and the geometric/regularity parameters within a generic domain's product-type ends—structures characterized by high extrinsic curvature. This foundational result yields new findings with novel physical insights and practical implications across these fields.
\end{abstract}
\maketitle

\section{Mathematical setup and statement of the main results}\label{sub:main results}

We consider the following coupled system of partial differential equations (PDEs) for $u, v \in H^1(\Omega)$:
\begin{equation}\label{eq:sym1}
\left\{
\begin{aligned}
-\nabla \cdot (h_1 \nabla u) &= f(\boldsymbol{x}, u, v) && \text{in } \Omega, \\
-\nabla \cdot (h_2 \nabla v) &= g(\boldsymbol{x}, u, v) && \text{in } \Omega, \\
u = v,\,\, h_1\partial &_{\nu}  u = h_2 \partial_{\nu} v && \text{on } \hat{\Gamma},
\end{aligned}
\right.
\end{equation}
where $h_1$ and $h_2$ are spatially varying coefficients, the nonlinearities $f$ and $g$ depend on the spatial coordinate $\boldsymbol{x}$ and the solution fields $u$ and $v$, and $\Omega \subset \mathbb{R}^n$ ($n=2,3$) is a bounded Lipschitz domain. Here $\hat{\Gamma} \subset \partial\Omega$ denotes a boundary segment, and \( \nu \in \mathbb{S}^{n-1} := \{ \boldsymbol{x} \in \mathbb{R}^n \; | \; |\boldsymbol{x}| = 1 \} \)  represents the outward unit normal to $\partial\Omega$. The boundary conditions on $\hat{\Gamma}$ in \eqref{eq:sym1} are termed transparent boundary conditions (TBCs) due to their characterization of a unique continuation property for solutions. Further discussion of this connection appears in Section~\ref{sec:applications}.
 
It is important to note that the coupling between $u$ and $v$ occurs both within the PDEs and through the TBCs. These coupled systems arise in a variety of problems of considerable theoretical and practical importance, including inverse boundary value problems, the spectral geometry of transmission eigenfunctions, and the geometric characterization of invisibility phenomena in wave scattering. We will discuss these applications in subsequent sections.

This paper focuses on pattern formation in system \eqref{eq:sym1}, specifically investigating the intrinsic relationships among the solutions $u$, $v$, the coefficients $h_1$, $h_2$, the nonlinearities $f$, $g$, and the geometry/topology of the boundary segment $\hat{\Gamma}$. We primarily examine local pattern formation in which $\hat{\Gamma}$ is a proper subset of $\partial \Omega$, considering nontrivial solutions ($u, v \not\equiv 0$) that naturally arise in the aforementioned applications. The geometric configuration for our study comprises domains with product-type ends that are thin or narrow in nature. Next, we introduce the specific geometric setup for our study, which comprises product-type thin or narrow ends of the domain $\Omega$; refer to Fig.~\ref{fig:t} for a schematic illustration of these two types of ends.

Let $\mathcal{N}_{\varepsilon}^T \subset \mathbb{R}^n\,(n=2,3 ) $ represent a thin end constructed by the parallel translation of the cross section $\Omega^T_{\varepsilon}$ along a simple curve $\eta(t) \in C^2$, with $\eta(t): I=(-L, L)\rightarrow \mathbb{R}^n$ being a graph, where \( L \in \mathbb{R}_+ \). Specifically, 
\begin{align}\label{eq:n1}
	\mathcal{N}_\varepsilon^T &:= \Omega^T_\varepsilon \times \eta(t),\ \ \varepsilon \ll 1, \ \ t \in I,
\end{align}
where \( \Omega^T_\varepsilon \) is a bounded, simply-connected Lipschitz domain in \( \mathbb{R}^{n-1} \) for \( n=2,3 \), with \( \mathrm{diam}(\Omega^T_\varepsilon) = \varepsilon \) and \( \partial \Omega^T_\varepsilon \)  being piecewise smooth and parametrizable. Assume there exists a unique \( t_0 \in I \) such that \( \eta(t) \cap \Omega^T_\varepsilon = \{\eta(t_0)\} \), and that \( \eta(t_0) \) is the highest point of \( \Omega^T_\varepsilon \). Furthermore, \( \eta \) is injective, i.e., \( \eta(t_1) \neq \eta(t_2) \) for any distinct \( t_1, t_2 \in I \). 

Similarly, the narrow end is defined as follows:
\begin{align}\label{eq:n2}
	\mathcal{N}_\varepsilon^N &:= \Omega^N_\varepsilon \times \eta(t) \subset \mathbb{R}^3,
\end{align}
 where \( \Omega^N_{\varepsilon} \subset \mathbb{R}^{2} \) is a rectangle with width \( \varepsilon \) and conventional length, while the simple curve \( \eta(t) \) is analogous to the curve defined in \eqref{eq:n1}. Notably, the thin end \( \mathcal{N}^T_\varepsilon \subset \mathbb{R}^3 \) possesses two sufficiently small dimensions, whereas the narrow end \( \mathcal{N}^N_\varepsilon \) has a thickness of \( \varepsilon \), which is significantly smaller than both its length and width.

Throughout this paper, we collectively refer to both thin and narrow ends as \emph{product-type ends}, denoted by 
\begin{equation} \label{eq:N}  
\mathcal{N}_{\varepsilon} := \mathcal{N}_{\varepsilon}^T \quad \text{or} \quad \mathcal{N}_{\varepsilon}^N,
\end{equation} 
defined in \eqref{eq:n1} and \eqref{eq:n2}. The lateral boundary of $\mathcal{N}_{\varepsilon}$ is given by
\begin{align}\label{eq:lateral}
	\hat{\Gamma} := \partial \Omega_\varepsilon \times \eta(t), \quad \text{where} \quad \Omega_{\varepsilon} = \Omega_{\varepsilon}^T \quad \text{or} \quad \Omega_{\varepsilon}^N.
\end{align}
Here, \( \varepsilon \) serves as the geometric parameter characterizing \( \mathcal{N}_\varepsilon \) and plays a crucial role in our analysis. Furthermore, we assume that the following geometric condition holds for product-type ends  $\mathcal{N}_{\varepsilon}$: for any point $\boldsymbol{x} \in \mathcal{N}_{\varepsilon}$, there exists a point $\boldsymbol{x}_0 \in \hat{\Gamma}$ such that $\boldsymbol{x}_0 - \boldsymbol{x}$ and $\boldsymbol{\nu}_{\boldsymbol{x}_0}$ are linearly dependent, where $\boldsymbol{\nu}_{\boldsymbol{x}_0}$ denotes the exterior unit normal vector to $\hat{\Gamma}$ at $\boldsymbol{x}_0$. In other words, $\boldsymbol{x}_0 - \boldsymbol{x}$ is colinear with the exterior unit normal $\boldsymbol{\nu}_{\boldsymbol{x}_0}$.  We refer to this condition as \textbf{Assumption G}. This is a generic assumption, and a detailed discussion is provided in Appendix \ref{sec:apA}. In fact, when cross-section $\Omega_{\varepsilon}\subset\mathcal{N}_{\varepsilon}$ is a Lipschitz domain for \(n=2\) and a convex domain with $C^2$ boundary for \(n=3\), we rigorously prove  \textbf{Assumption G} in Theorem \ref{th:cover}. If $\Omega_\varepsilon$ is a line segment in $\mathbb{R}^1$ or a rectangle in $\mathbb{R}^2$, it is readily seen that the generated domain $\mathcal{N}_{\varepsilon}$ fulfills \textbf{Assumption G}.




\begin{figure}[htbp]
	\subfloat{
	\includegraphics[width=6cm,height=3.5cm]{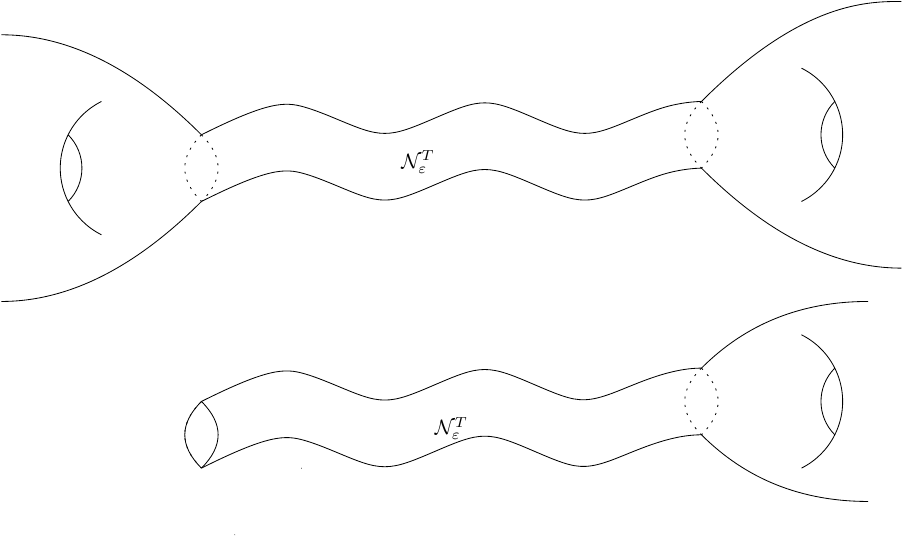}}
	\hfill
	\centering
	\subfloat{
	\includegraphics[width=9cm,height=5cm]{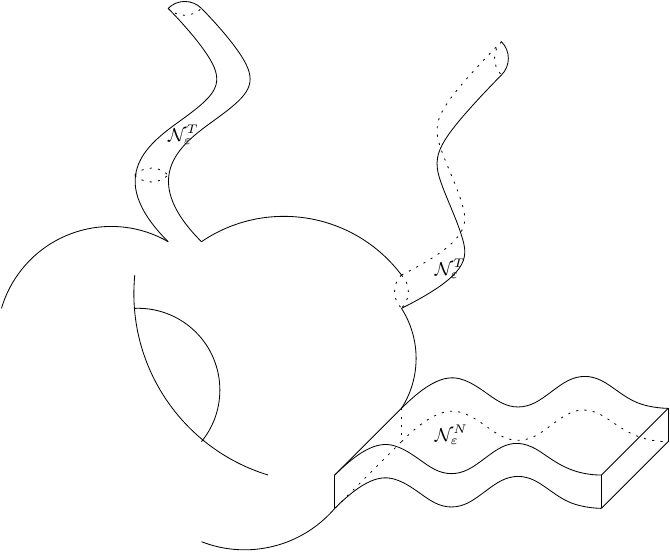}}
	\caption{Schematic illustration of the thin and narrow ends in $\mathbb{R}^3$}
	\label{fig:t}
\end{figure}

To motivate this construction, consider a generic domain \( \Omega \) as in \eqref{eq:sym1}, which contains a boundary point \( p \in \partial \Omega \) with high mean curvature \( K > 0 \) (see Fig.~\ref{fig:extrinsic points}). Locally opening the boundary near \( p \) and attaching a product-type end \( \mathcal{N}_\varepsilon \) extends the domain, with the scaling \( \varepsilon \sim K^{-1} \) relating the geometric parameter to the curvature. Thus, \( \mathcal{N}_\varepsilon \) constitutes a high-curvature end of the extended domain, where each interior point exhibits extrinsic curvature of order \( \varepsilon^{-1} \) when viewed from the regular part of \( \Omega \).

\begin{figure}
 \centering
\includegraphics[width=6.5cm,height=4cm
 ]{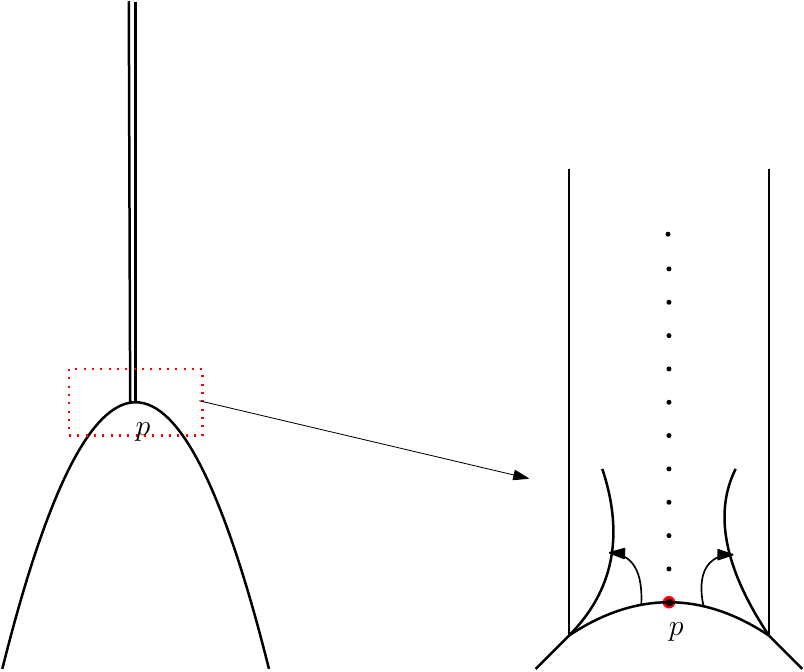}
\caption{
Schematic of extrinsic curvature points in the product-type end}\label{fig:extrinsic points} 
\end{figure}

Now, we can fix the geometric setup of our study by considering that \( \Omega \subset \mathbb{R}^n \) is a Lipschitz domain with product-type ends (see Figure \ref{fig:t}), defined by
\begin{equation}\label{eq:multi}
\mathcal{N}_\varepsilon = \bigcup_{l=1}^M \mathcal{N}_\varepsilon^l \subset \Omega,
\end{equation}
where each \( \mathcal{N}_\varepsilon^l \) corresponds to a thin or narrow end, and \( \mathcal{N}_\varepsilon^i \) and \( \mathcal{N}_\varepsilon^j \) are pairwise disjoint for \( i \neq j \) with \( 1 \leq i, j \leq M \). The subset \( \hat{\Gamma} \subset \partial \Omega \) is defined in \eqref{eq:lateral}. Moreover, the nonlinearities \( f \) and \( g \) in \eqref{eq:sym1}, involving three variables \( (\boldsymbol{x}, z, p) \in \Omega \times \mathbb{C} \times \mathbb{C} \), 
satisfy the following regularity conditions, collectively termed \textbf{Assumption H}:
\begin{itemize}
    \item[(a)] For \( u(\cdot), v(\cdot) \in H^1(\Omega) \) and \( h_1(\cdot), h_2(\cdot) \in L^\infty(\Omega) \), \( f(\boldsymbol{x}, z, p) \) and \( g(\boldsymbol{x}, z, p) \) belong to \( L^2(\Omega) \).
    \item[(b)] Nonlinearities \( f(\boldsymbol{x}, z, p) \) and \( g(\boldsymbol{x}, z, p) \) are \( C^{0,\zeta} \)-continuous, with \( \zeta \in (0,1) \), concerning \( (\boldsymbol{x}, z, p) \in \mathcal{N}_\varepsilon \times \mathbb{C} \times \mathbb{C} \). Specifically, when any two of the variables \( (\boldsymbol{x}, z, p) \) are fixed, \( f(\cdot) \) and \( g(\cdot) \) are \( C^{0,\zeta} \)-continuous with respect to the remaining variable. The \( C^{0,\zeta} \)-norm is defined as follows:
    \begin{align*}
        \hspace{1.4cm} \|h(\boldsymbol{x}, z, p)\|_{C^{0,\zeta}(\mathcal{N}_\varepsilon \times \mathbb{C} \times \mathbb{C})} = \sup \Big\{ 
        &\|h(\boldsymbol{x}, z_0, p_0)\|_{C^{0,\zeta}(\mathcal{N}_\varepsilon)}, \|h(\boldsymbol{x}_0, z, p_0)\|_{C^{0,\zeta}(\mathbb{C})}, \\
        &\|h(\boldsymbol{x}_0, z_0, p)\|_{C^{0,\zeta}(\mathbb{C})} \colon \boldsymbol{x}_0 \in \mathcal{N}_\varepsilon, z_0 \in \mathbb{C}, p_0 \in \mathbb{C} \text{ fixed} \Big\},
    \end{align*}
    for \( h(\cdot) = f(\cdot) \) or \( h(\cdot) = g(\cdot) \). Here, \( \|h(\boldsymbol{x}, z_0, p_0)\big\|_{C^{0,\zeta}(\mathcal{N}_\varepsilon)} \) is defined as:
    $$
    \|h(\boldsymbol{x}, z_0, p_0)\big\|_{C^{0,\zeta}(\mathcal{N}_\varepsilon)} =
    \underset{\boldsymbol{x} \in \mathcal{N}_\varepsilon}{\sup}
    \big|h(\boldsymbol{x}, z_0, p_0)\big| + \underset{\substack{\boldsymbol{x}, \boldsymbol{y} \in \mathcal{N}_\varepsilon \\ \boldsymbol{x} \neq \boldsymbol{y}}}{\sup} \frac{\left|h(\boldsymbol{x}, z_0, p_0) - h(\boldsymbol{y}, z_0, p_0)\right|}{\left|\boldsymbol{x} - \boldsymbol{y}\right|^\zeta}.
    $$
    The H\"older norms \( \|h(\boldsymbol{x}_0, \cdot, p_0)\|_{C^{0,\zeta}(\mathbb{C})} \) and \( \|h(\boldsymbol{x}_0, z_0, \cdot)\|_{C^{0,\zeta}(\mathbb{C})} \) are defined analogously.
    \item[(c)] For the nonlinearities \( f(\boldsymbol{x}, z, p) \) and \( g(\boldsymbol{x}, z, p) \) that satisfy \((\mathrm{a})\) and \((\mathrm{b})\), we assume the coupled PDE system \eqref{eq:sym1} admits a non-trivial solution pair \((u, v)\).
\end{itemize}

The coupled PDE system \eqref{eq:sym1} arises in the study of uniqueness for Calder\'on's inverse boundary value problems and inverse quantum or acoustic scattering. In the context of invisible quantum or acoustic scattering, the system \eqref{eq:sym1} takes special forms with specific choices of the functions \( h_1 \), \( h_2 \), \( f \), and \( g \). For further details, see Subsections \ref{sub:calder} and \ref{sub:visi}.

\textbf{Assumption H} for the system \eqref{eq:sym1} is generic, as it holds across a wide range of physical scenarios, including investigations of uniqueness in Calder\'on's inverse boundary value problems and inverse quantum or acoustic scattering, as well as studies of invisibility in quantum and acoustic scattering under general conditions. Additional discussion of this point is provided in the subsequent sections.

We now state the main result, which characterizes local pattern formations in the coupled PDE system \eqref{eq:sym1} and establishes a sharp quantitative relationship between differences in the PDEs' lower-order terms and the geometric/regularity parameters of product-type ends. The novel findings in Theorem \ref{thm:1} provide significant theoretical and practical insights across multiple disciplines.

\begin{theorem}\label{thm:1}
Let \( (u,v) \) be a solution to the coupled PDE system \eqref{eq:sym1}. Suppose there exist positive constants \( C_1 \), \( C_2 \), \( \alpha_1 \in (0,1) \), \( \alpha_2 \in (0,1) \), \( \zeta \in (0,1) \), and \( \alpha_3 \in [2,\infty) \), all independent of \( \varepsilon \) with \( \alpha_1 \leq \alpha_2 \), such that the a priori conditions are as follows:
\begin{itemize}
\item the solution \( (u,v) \) satisfies
    \begin{equation}\label{eq:ap1}
 u, v \in H^1_{loc}(\mathcal{N}_\varepsilon) \cap C^{1 ,\alpha_1}(\overline{\mathcal{N}}_\varepsilon) \ \ \mbox{with}\ \ \|u\|_{C^{1,\alpha_1}(\overline{\mathcal{N}}_\varepsilon)},\,\, \|v\|_{C^{1,\alpha_1}(\overline{\mathcal{N}}_\varepsilon}) \leq C_1,
\end{equation}
where \( \alpha_1 \in \left(\frac{1}{2},1\right) \) in \( \mathbb{R}^2 \) and \( \alpha_1 \in \left(\frac{2}{3},1\right) \) in \( \mathbb{R}^3 \), and \( \mathcal{N}_{\varepsilon} \subset \Omega; \)
\item the coefficients \( h_1 \) and \( h_2 \) in the system \eqref{eq:sym1} satisfy
    \begin{equation}\label{eq:h}
	\hspace{0.7cm} h_1,h_2 \in C^{2,\alpha_2}(\overline{\mathcal{N}}_{\varepsilon}) \ \ \mbox{with} \ \ 0 < M_2 \leq h_1,h_2 \leq M_1 \ \ \mbox{and} \ \ \|h_1-h_2\|_{C^{2,\alpha_1}(\partial \mathcal{N}_\varepsilon)} \leq \varepsilon^{\alpha_3};
\end{equation}
\item the right-hand sides \( f \) and \( g \) of the system \eqref{eq:sym1} satisfy 
    \begin{equation}\label{eq:fg1}
 f,g \in C^{0,\zeta}(\overline{\mathcal{N}}_\varepsilon \times \mathbb{C} \times \mathbb{C}) \ \ \mbox{with} \ \ \|f\|_{C^{0,\zeta}(\overline{\mathcal{N}}_\varepsilon \times \mathbb{C} \times \mathbb{C})}, \,\, \|g\|_{C^{0,\zeta}(\overline{\mathcal{N}}_{\varepsilon} \times \mathbb{C} \times \mathbb{C})} \leq C_2.
\end{equation}
\end{itemize}
Then, for any \( \boldsymbol{x} \in \mathcal{N}_\varepsilon \), there exists a positive constant \( \varepsilon_0 \) such that if \( \varepsilon < \varepsilon_0 \),
\begin{equation}\label{eq:fg}
\Big| f(\boldsymbol{x},u(\boldsymbol{x}),v(\boldsymbol{x})) - g(\boldsymbol{x},u(\boldsymbol{x}),v(\boldsymbol{x})) \Big| \leq C( \varepsilon_0,M_1,M_2,C_1,C_2) \varepsilon^{\tau},
\end{equation}
where \( \tau \in (0,1) \) and \( C( \varepsilon_0, M_1, M_2, C_1, C_2) \) is a generic positive constant that depends only on \( \varepsilon_0, M_1, M_2, C_1, C_2 \).
\end{theorem}
\begin{remark}The TBCs (i.e., \( u = v \) and \(h_1 \partial_{\nu} u = h_2\partial_{\nu} v \)) on \( \hat{\Gamma} \) in \eqref{eq:sym1}, along with the geometric construction of product-type ends, only guarantee pointwise approximations between the \( 0 \)th-order and \( 1 \)st-order derivatives of \( u \) and \( v \) within the product-type ends. However, these conditions do not yield pointwise approximations for the second-order and higher derivatives of \( u \) and \( v \) in $\mathcal{N}_{\varepsilon}$. Consequently, the result in Theorem~\ref{thm:1} is both significant and non-trivial.

Furthermore, it is important to emphasize that \( \tau \) in \eqref{eq:fg} is determined by the regularity parameters \( \alpha_1 \) and \( \zeta \); see \eqref{eq:tau2} for 2D and \eqref{eq:tau3} for 3D for details. Based on the observations in equations \eqref{eq:tau2} and \eqref{eq:tau3}, \( \tau \) increases with these regularity parameters, leading to progressively better approximations of \( f \) to \( g \) in \( \mathcal{N}_{\varepsilon} \). Moreover, when the coefficients \( h_1 \) and \( h_2 \) in \eqref{eq:sym1} are both equal to 1, the estimate \eqref{eq:fg} remains valid for a bounded simply-connected Lipschitz domain \( \Omega^T_\varepsilon \), without requiring piecewise smoothness or parameterization of the boundary \( \partial \Omega^T_\varepsilon \). Indeed, \eqref{eq:fg1} follows directly from the assumptions (a)-(c) in {\bf Assumption H} related to the system \eqref{eq:sym1}. The configuration lists \eqref{eq:ap1} and \eqref{eq:h} are reasonable and can be satisfied under generic conditions, particularly in the context of the invisibility of the medium scatterer. Please refer to Remark \ref{rem:CD} for further details.

\end{remark}

\section{Applications and connection with existing results}\label{sec:applications}
\subsection{Calder\'on's inverse boundary value problems}\label{sub:calder}
The Calder\'on's problem \cite{Calderon1980}, which seeks to recover the electrical conductivity of a medium through voltage and current measurements taken at its boundary, has inspired extensive research on inverse boundary value problems (IBVPs). Subsequent work has extended to both linear and nonlinear regimes, including elliptic \cite{IY2013, Isakov1994, sun2010, Uhlmann2009,IUY2010,KSU2007,SU1987}, parabolic  \cite{Isakov1993,LLLS1}, hyperbolic equations  \cite{CLOP2021, HUW2020,  sun2005,LLL1}, as well as Calder\'on-type inverse problems for quasilinear equations defined on Riemannian manifolds \cite{ LLLS2021,LLS2020} and related references.

Building on the local pattern formations established earlier, we derive novel, unique identifiability results for inverse boundary problems associated with the nonlinear PDE:
\begin{equation}\label{eq:ibvp}
\begin{cases}
    -\nabla \cdot (h \nabla u) = f(\boldsymbol{x},u) & \text{in } \Omega, \\
    u|_{\partial \Omega} = \psi & \text{on } \partial \Omega,
\end{cases}
\end{equation}
where \(\Omega \subset \mathbb{R}^n\) (\(n = 2, 3\)) is a Lipschitz domain with thin or narrow ends, \(\psi \in H^{1/2}(\partial \Omega)\), \(h \in L^\infty(\Omega)\), and \(f\) belong to an a priori admissible class \(\mathcal{M}\). Here, \(\psi\) represents boundary quantities (temperature, concentration, or population in different physical contexts), while \(h\, \partial_{\nu} u |_{\partial \Omega}\) denotes the equilibrium flux across the boundary, with \(u\) solving \eqref{eq:ibvp}. For appropriate \(\psi\in H^{1/2}(\partial\Omega)\), we assume existence of solutions \(u \in H^1(\Omega)\) (see e.g., \cite{DFLW2022, Hormander1983, LLLS2021}). It is essential to note that we only assume the existence of a solution to \eqref{eq:ibvp} and do not assume its uniqueness; therefore, multiple solutions may exist. 

We study the inverse problem of identifying the nonlinearity $f$ in \eqref{eq:ibvp} from partial Cauchy data on $\hat{\Gamma}$ (defined in \eqref{eq:lateral}). To formulate this, we introduce the Dirichlet-to-Neumann (DtN) map $\Lambda_f : \psi \mapsto h\,\partial_{\nu}u|_{\partial \Omega}$, which associates Dirichlet boundary data with the corresponding Neumann data. The inverse problem is then expressed as:
\begin{equation}\label{eq:S}
\big(\psi, \Lambda_f(\psi)\big)\big|_{\hat{\Gamma}} \longrightarrow f \quad \text{for a fixed } \psi \in H^{1/2}(\partial\Omega).
\end{equation}

While most existing studies of Calder\'on-type inverse problems utilize boundary measurements $(\psi,\Lambda_{f}(\psi))|_{\hat{\Gamma}}$ with $\psi$ varying over a function space (representing infinitely many measurements), our work fixes $\psi$ to yield a single measurement. We distinguish two scenarios:
\begin{itemize}
    \item \textit{Single active measurement}: When $\psi$ is actively prescribed and $(\psi,\Lambda_{f}(\psi))|_{\hat{\Gamma}}$ is recorded;
    \item \textit{Single passive measurement}: When a physical process inside $\Omega$ generates $u$ naturally and $(u, h\partial_\nu u)|_{\hat{\Gamma}}$ is passively observed. 
\end{itemize}
Noting that $\hat{\Gamma} \Subset \partial\Omega$, our inverse problem \eqref{eq:S} thus employs a \textit{single partial boundary measurement} (active or passive), which constitutes longstanding challenges in the field.
 
Under generic conditions, we can establish the approximate unique identifiability of \eqref{eq:S} within an a priori admissible class of \( f \) using a single measurement on \( \hat{\Gamma} \). To that end, we first introduce
\begin{definition}({\bf  Admissible class $\mathcal{M}$}) 
We define $f(\boldsymbol{x},z):\Omega\times \mathbb{C} \longrightarrow \mathbb{C}$, and $h\in L^\infty(\Omega)$ to be in the class $\mathcal{M}$ if it satisfies the following admissibility conditions:
  \begin{itemize}
  \item[(a)]
  for $u\in H^1(\Omega)$, the function $f(\boldsymbol{x},u)\in L^2(\Omega)$ is $C^1$-continuous with respect to $u$ for each fixed $\boldsymbol{x}\in\Omega$, and $\partial_uf(\boldsymbol{x},u)\in L^\infty(\Omega)$;
  \item[(b)] in $\mathcal{N}_\varepsilon$, $u\in C^{1,\alpha_1}(\overline{\mathcal{N}}_\varepsilon)$ for $\alpha_1\in(0,1)$, $h\in C^{2,\alpha_2}(\overline{\mathcal{N}}_{\varepsilon})$ for $\alpha_2\in(0,1)$, and $h\in (M_2, M_1)$ for positive constants $M_2$ and $M_1$. Additionally, $f(\boldsymbol{x},z)$ is $C^{0,\zeta}$-continuous with respect to $(\boldsymbol{x},z)\in \mathcal{N}_\varepsilon\times \mathbb{C}$, and $\zeta\in (0,1)$;
  \item[(c)] for the right-hand term $f(\boldsymbol{x},u)$ satisfying conditions (a) and (b), we assume the existence of a solution to \eqref{eq:ibvp} for some appropriate \( \psi \in H^{1/2}(\partial \Omega) \). 
  \end{itemize}
\end{definition}

The admissible assumptions on the nonlinear term $f$ are incorporated within the constraints of the nonlinear coupled right-hand sides $f(\cdot)$ and $g(\cdot)$ in system \eqref{eq:sym1}. We now establish the main theoretical result for the inverse problem \eqref{eq:ibvp}, demonstrating that under generic conditions, $f(\boldsymbol{x}, z)\in\mathcal{M}$ can be stably reconstructed up to an approximation from a single boundary measurement pair $\big(\psi, \Lambda_{f}(\psi)\big)|_{\hat{\Gamma}}$.
\begin{theorem}\label{th:ibvp}
Let \(\Omega \subset \mathbb{R}^n\,(n=2,3)\) be a bounded Lipschitz domain with thin or narrow ends $\mathcal{N}_{\varepsilon}$. Assume that \( f_j \) and $h$ belong to the admissible class \( \mathcal{M} \) and \( \Lambda_{f_j} \) is the DtN map associated with \eqref{eq:ibvp} for \( j=1,2 \). For a fixed admissible \( \psi \in H^{1/2}(\hat{\Gamma}) \) with \( \hat{\Gamma} \subset \partial \Omega \), if
\begin{equation*}
    \left(\psi,\Lambda_{f_1}(\psi)\right)\big|_{\hat{\Gamma}}=\left(\psi,\Lambda_{f_2}(\psi)\right)\big|_{\hat{\Gamma}},
\end{equation*}
 then it holds that
\begin{equation}\label{eq:err}
\big|f_1(\boldsymbol{x},u_1(\boldsymbol{x}))-f_2(\boldsymbol{x},u_2(\boldsymbol{x})) \big|\leq C(\varepsilon_0,M_1,M_2,C_1,C_2)  \varepsilon^{\tau}\quad \mbox{for any }\,\,\boldsymbol{x}\in \mathcal{N}_{\varepsilon},
\end{equation}
where \( u_j\,(j=1,2) \) are solutions to \eqref{eq:ibvp} corresponding to \( f_j \) and satisfying \eqref{eq:ap1}. Here, \( \tau \in (0,1) \) and  \( \varepsilon_0,M_1,M_2, C_1, C_2 \) are specified in Theorem \ref{thm:1} and are independent of \( \varepsilon \).
\end{theorem}

Theorem~\ref{th:ibvp} implies that a single passive/active boundary measurement on $\hat{\Gamma}$ can uniquely determine (with $\varepsilon^\tau$-accuracy) the physical process, namely the PDE--$-\nabla\cdot(h\nabla u)-f(\boldsymbol{x}, u)=0$ inside $\mathcal{N}_\varepsilon$. To be clearer, we can present the unique identifiability result for the inverse problem \eqref{eq:S} from a novel operator-theoretic perspective. We define the associated operator $\mathscr{F}$ for the nonlinearity $f(\boldsymbol{x},u)$ as a mapping between Sobolev spaces: \begin{align*}
\mathscr{F} \colon H^1(\Omega) &\to L^2(\Omega), \notag \\
u(\boldsymbol{x}) &\mapsto \mathscr{F}(u)(\boldsymbol{x}) := f(\boldsymbol{x}, u(\boldsymbol{x})), \quad \boldsymbol{x} \in \Omega.
\end{align*}
The inverse problem \eqref{eq:S} can be reformulated as recovering the operator $ \mathscr{F} $ from the boundary measurements $ \big(\psi,\Lambda_{f}(\psi)\big) $ obtained from samples in $ H^1(\Omega) $ on the partial boundary $ \hat{\Gamma} $. That is,
\begin{equation}\label{eq:S2}
\big(\psi,\Lambda_{f}(\psi)\big)\Big|_{\hat{\Gamma}} \longrightarrow \mbox{operator }\, \mathscr{F}.
\end{equation}
In this framework, $ u(\boldsymbol{x}) $ is regarded as a sample from $ H^1(\Omega) $, and $ \mathscr{F}(u)(\boldsymbol{x})= f(\boldsymbol{x}, u(\boldsymbol{x})) $ corresponds to a realization of the operator $ \mathscr{F} $ in $ L^2(\Omega) $.
 
The research perspective described in \eqref{eq:S2} is highly general and offers substantial practical value, representing a transformation in methodological philosophy. This operator-theoretic perspective naturally encompasses the results in Theorem \ref{th:ibvp}. In fact, we can first show that
\begin{theorem}\label{th:operator}
Let \( \Omega = \mathcal{N}_\varepsilon \) be defined as in \eqref{eq:multi}, under the same mathematical framework for the nonlinearities \( f_j \,(j=1,2) \) as in Theorem \ref{th:ibvp}. Let \( \Lambda_{f_j} \) be the DtN maps and \( \mathscr{F}_j \) the associated operators corresponding to \( f_j\). For any fixed admissible \( \psi \in H^{1/2}(\hat{\Gamma}) \) with \( \hat{\Gamma} \subset \partial \Omega \), if it holds that
\begin{equation*}
    \left(\psi,\Lambda_{f_1}(\psi)\right)\big|_{\hat{\Gamma}} = \left(\psi,\Lambda_{f_2}(\psi)\right)\big|_{\hat{\Gamma}},
\end{equation*}
then the following inequalities hold for any \( \boldsymbol{x} \in \Omega \):
\begin{equation*}
|u_1(\boldsymbol{x}) - u_2(\boldsymbol{x})| \leq \widetilde{C}_1(\varepsilon_0,M_1,M_2,C_1,C_2)\varepsilon^{\tau},
\end{equation*}
and
\begin{equation*}
\big|\mathscr{F}_1(u_j)(\boldsymbol{x}) - \mathscr{F}_2(u_j)(\boldsymbol{x}) \big| \leq \widetilde{C}_2(\varepsilon_0,M_1,M_2,C_1,C_2)\varepsilon^{\tau \zeta},\quad j=1,2,
\end{equation*}
where \( u_j\,(j=1,2) \), are solutions to \eqref{eq:ibvp} corresponding to \( f_j \) and satisfying \eqref{eq:ap1}. Here, $\tau \in(0,1)$, and \( \zeta \in(0,1)\). The positive constants \( \widetilde{C}_1(\varepsilon_0,M_1,M_2,C_1,C_2) \) and \( \widetilde{C}_2(\varepsilon_0,M_1,M_2,C_1,C_2) \)  depend only on $\varepsilon_0,M_1,M_2,C_1,C_2$.
\end{theorem}

Theorem~\ref{th:operator} demonstrates that any realization of \( \mathscr{F} \) can be approximately reconstructed from a single boundary measurement on $\hat{\Gamma}$. Specifically, the discrepancy between realizations of \(\mathscr{F}_1\) and \(\mathscr{F}_2\) at product-type ends is bounded by \( \varepsilon^{\tau\zeta} \), where \(\varepsilon\) and \(\tau\zeta\) represent geometric and regularity parameters, respectively. Full recovery of the operator \(\mathscr{F}\) itself subsequently requires generating enough samples to obtain multiple realizations for the reconstruction process.

\begin{remark}
Compared to previous approaches to inverse boundary problems involving nonlinear PDEs, the new methodology is a fundamental shift. Physically, the mapping relationship between the two function spaces $H^1(\Omega)$ and $L^2(\Omega)$ corresponds to identifying the physical process within \( \Omega\) through boundary observations. The key innovation here lies in concentrating on the input-output behavior of the operator \( \mathscr{F} \) rather than the explicit functional form of the realizations of \( \mathscr{F} \). Naturally, to reconstruct the explicit expression of the realization \( f(\boldsymbol{x}, u) \), multiple measurements across different samples and appropriate prior information about the realization are required, such as the requirement that \( f(\boldsymbol{x}, u) \) is a polynomial in terms of \( u \).
\end{remark}

\begin{remark}
For the IBVP in Theorem~\ref{th:operator}, we consider a special geometric configuration where \(\mathcal{N}_\varepsilon = \Omega_s \times \Omega_r\) with \(\Omega_s := \Omega_\varepsilon\) and \(\Omega_r := \eta\). In \(\mathbb{R}^3\), \(\Omega_\varepsilon\) constitutes a two-dimensional manifold while \(\eta\) is one-dimensional. Mathematically, this decomposes \(\mathcal{N}_\varepsilon\) into a product of sub-manifolds: \(\Omega_s\) (the hidden dimension) and \(\Omega_r\) (the regular dimension). Our unique identifiability result demonstrates that physical processes occurring within the hidden dimension \(\Omega_s\) in \(\mathcal{N}_\varepsilon\) can be stably recovered from measurement data on \(\partial \Omega_s \times \Omega_r\), which represents the projection onto the regular dimension. This stability originates from the high extrinsic curvature of the hidden dimension, as established in Section~\ref{sub:main results}. Herein, we validate this framework for the semilinear PDE \eqref{eq:ibvp} in \(\mathbb{R}^n\) (\(n = 2, 3\)), with generalizations to broader mathematical and physical contexts deferred to forthcoming work.
\end{remark}

\subsection{Transmission eigenvalues problems} 
Transmission eigenvalue problems are closely linked to inverse scattering theory via unique identifiability, reconstruction algorithms, and invisibility cloaking, the historical developments and contemporary advances of which can be reviewed in \cite{Cako1, CK2, L1}.

In this subsection, we primarily investigate the spectral geometry of transmission eigenfunctions. If there exists a non-trivial pair of solutions $(u,v)$  to the system \eqref{eq:sym1}  such that 
\begin{equation}\label{eq:spe}
h_1=h_2=1,\quad f(\boldsymbol{x},u,v)=\left(\lambda+V\right) u,\quad \mbox{and} \quad g(\boldsymbol{x},u,v)=\lambda v
\end{equation}
for $\lambda :=k^2\in \mathbb{R}_+$ and $V\in L^\infty (\Omega)$,
then the system \eqref{eq:sym1} is classified as the interior transmission eigenvalue problem. In this context,  $u,\,v$ are referred to as the transmission eigenfunctions. Thus, the general coupled system \eqref{eq:sym1} reduces to
\begin{equation}\label{eq:s1}
	\begin{cases}
		-\Delta u-Vu = \lambda u\quad & \mbox{in}\ \ \Omega,\\
		\hspace*{1.0cm} -\Delta v = \lambda v\quad & \mbox{in}\ \ \Omega,\\
		 u=v,\,\, \partial_\nu u=\partial_\nu v & \mbox{on}\ \ \hat{\Gamma}.
	\end{cases}
\end{equation}

 It is important to note that \eqref{eq:s1} comprises the Schr\"{o}dinger equation and the Helmholtz equation. The findings regarding the local pattern formations of the interior transmission problem \eqref{eq:s1} are elucidated in the subsequent theorems, revealing that the transmission eigenfunctions $u$ and $v$ exhibit near-vanishing behavior in region $\overline{\mathcal{N}}_\varepsilon$.
\begin{corollary}\label{Cor:1}
Let \((u,v)\) be a solution to the coupled system \eqref{eq:s1}. Assume there exist positive constants \(C_1\) and a regularity parameter \(\alpha_1 \in (0,1)\), both independent of \(\varepsilon\), such that the solution \((u,v)\) satisfies
\begin{equation}\label{eq:uv2}
u, v \in H^1_{loc}(\mathcal{N}_\varepsilon) \cap C^{1,\alpha_1}(\overline{\mathcal{N}}_\varepsilon) \quad \text{with} \quad \|u\|_{C^{1,\alpha_1}(\overline{\mathcal{N}}_\varepsilon)}, \, \|v\|_{C^{1,\alpha_1}(\overline{\mathcal{N}}_\varepsilon)} \leq C_1,
\end{equation}
where \(\alpha_1 \in \left(\frac{1}{2},1\right)\) in \(\mathbb{R}^2\) and \(\alpha_1 \in \left(\frac{2}{3},1\right)\) in \(\mathbb{R}^3\), and \(\mathcal{N}_{\varepsilon} \subset \Omega\). Moreover, the potential \(V\) satisfies the following conditions:
\begin{equation}\label{eq:ap2}
\|V\|_{C^{0, \alpha}(\overline{\mathcal{N}}_\varepsilon)} \leq V_0 
\end{equation} 
and
\begin{equation}\label{eq:ap3}
|V(\boldsymbol{x})| \geq \epsilon_0, \quad x \in \overline{\mathcal{N}}_\varepsilon,
\end{equation}
where \(\alpha \in (0,1)\), and \(V_0\) and \(\epsilon_0\) are positive constants that do not depend on \(\varepsilon\). Then, there exists a positive constant \(\varepsilon_0\) such that when \(\varepsilon < \varepsilon_0\),

\begin{equation}\label{eq:k1}
|u(\boldsymbol{x})|, |v(\boldsymbol{x})| \leq C(\varepsilon_0, \epsilon_0, k, V_0, C_1) \varepsilon^{\tau_1}, 
\end{equation}
where \(\tau_1 \in (0,1)\) and \(C(\varepsilon_0, \epsilon_0, k, V_0, C_1)\) is a generic positive constant that depends solely on \(\varepsilon_0, \epsilon_0, k, V_0, C_1\).
\end{corollary}
Corollary \ref{Cor:1} applies to all wave numbers $k\in\mathbb{R}_+$. Moreover, the results remain valid when $V$ is substituted with $k^2(q(x)-1)$, where $q\in C^{0,\alpha}(\overline{\mathcal{N}}_\varepsilon)$ has a discontinuity on the lateral boundary $\hat{\Gamma}$ of $\mathcal{N}_\varepsilon$. In this context, the Schr\"{o}dinger equation in \eqref{eq:s1} transforms into the Helmholtz equation represented by
	$$\Delta u+k^2q u=0,$$ and the coupled system described by \eqref{eq:s1} becomes
	\begin{equation}\label{eq:h2}
		\left\{\begin{aligned}
			\Delta u+k^2 q u & =0 & & \text { in } \Omega, \\
			\Delta v+k^2 v & =0 & & \text { in } \Omega, \\
			 u=v,\,\, \partial_\nu& u=\partial_\nu v & & \text { on } \hat{\Gamma}. 
		\end{aligned}\right.
	\end{equation}
The following corollary offers a more comprehensive description.

\begin{corollary}\label{Cor:2}
Let \( (u,v) \) be a solution to the system \eqref{eq:h2}. Suppose the solution \( (u,v) \) satisfies the priori condition \eqref{eq:uv2} and that the coefficient \( q \) satisfies 
\begin{equation}\label{eq:q}
q \in C^{0,\alpha}(\overline{\mathcal{N}_{\varepsilon}}) \quad \text{and} \quad |q - 1| \geq \epsilon_0 \ \ \text{on} \ \ \overline{\Omega},
\end{equation}
where $\alpha\in (0,1)$ and $\epsilon_0$ is a positive number. 
Then, it follows that
\begin{equation}\label{eq:kk2}
|u(\boldsymbol{x})|, |v(\boldsymbol{x})| \leq C(\varepsilon_0,\epsilon_0, k, V_0,C_1) \varepsilon^{\tau_2} \quad \text{for any} \quad \boldsymbol{x} \in \overline{\mathcal{N}}_\varepsilon,
\end{equation}
where \( \tau_2 \in (0,1) \) and \( C(\varepsilon_0,\epsilon_0, k, V_0,C_1) \) is a generic positive constant that depends solely on \( \varepsilon_0, \epsilon_0, k, V_0, C_1 \).
\end{corollary}

\begin{remark}\label{rem:CD}
The coupled systems \eqref{eq:s1} and \eqref{eq:h2} are special cases of the general system \eqref{eq:sym1} when the functions \( h_i \) (\( i=1,2 \)), \( f \), and \( g \) are chosen as specified in \eqref{eq:spe}. Consequently, Corollaries \ref{Cor:1} and \ref{Cor:2} follow directly from Theorem \ref{thm:1}. The assumptions \eqref{eq:ap1} to \eqref{eq:fg1} in Theorem \ref{thm:1} are reformulated as \eqref{eq:uv2}, \eqref{eq:ap2}, \eqref{eq:ap3}, and \eqref{eq:q} for these coupled systems. 

Furthermore, \eqref{eq:s1} and \eqref{eq:h2} arise naturally in inverse scattering problems for quantum and acoustic systems, as described by \eqref{eym:Sch}, particularly when studying invisible medium scatterers (see Subsection \ref{sub:visi} for details). In the scattering system \eqref{eym:Sch}, the total wave field is given by \( w = w^i + w^s \), where \( w^i \) denotes the incident wave and \( w^s \) the scattered wave in the medium scatterer \( \Omega \). When \( \Omega \) has product-type ends and is invisible, the systems \eqref{eq:s1} or \eqref{eq:h2} emerge in two distinct scenarios: quantum scattering or acoustic scattering. It follows that \( u = w \) and \( v = w^i \) form a solution to either \eqref{eq:s1} or \eqref{eq:h2}. We emphasize that the assumptions \eqref{eq:ap2}, \eqref{eq:ap3}, and \eqref{eq:q} are generic conditions related to the physical parameters \( V \) and \( q \) in \eqref{eq:s1} and \eqref{eq:h2}, which represent potential and refractive index, respectively, in the quantum and acoustic scattering system \eqref{eym:Sch}. These assumptions hold a priori when the invisible medium scatterer \( \Omega \) satisfies the corresponding physical configurations. In Appendix \ref{sec:apB}, for illustration,  we rigorously demonstrate that assumption \eqref{eq:uv2} holds under \eqref{eq:q} for acoustic medium scattering in \eqref{eym:Sch}. Thus, the assumptions \eqref{eq:ap1} to \eqref{eq:fg1} in Theorem \ref{thm:1} are satisfied for the coupled systems \eqref{eq:s1} and \eqref{eq:h2} arising from invisible quantum or acoustic scattering systems \eqref{eym:Sch}, confirming their generic and broadly applicable nature.

\end{remark}

\begin{remark}
    It is important to emphasize that the physical configurations \(V\) and \(q\) differ from the corresponding physical configuration of the homogeneous background; See \eqref{eq:ap3} and \eqref{eq:q} for details. These differences are termed singularities in the physical configurations, which are crucial for establishing the near-vanishing properties described in \eqref{eq:k1} and \eqref{eq:kk2} in Corollaries \ref{Cor:1} and \ref{Cor:2}, respectively. Specifically, \eqref{eq:k1} and \eqref{eq:kk2} describe the near-vanishing of the corresponding wave fields in quantum scattering or acoustic medium scattering at the thin or narrow ends. Generally, existing results in the literature on invisibility in inverse scattering have been established under the condition that the medium scatterer exhibits geometric singularities, such as corner points or high-curvature points on the boundary, while also demonstrating singularities in the physical configurations. However, Corollaries \ref{Cor:1} and \ref{Cor:2} indicate that the near-vanishing of the eigenfunctions occurs not only at the boundaries of the domains but also within their interiors. These findings suggest that the geometric singularity of the domain is not essential; rather, the singularity of the physical configuration in the coupled system is critical, resulting in properties that are either vanishing or nearly vanishing.
\end{remark}

\subsection{Invisibility and inverse shape problems in wave scattering}\label{sub:visi}
In this subsection, we focus on the quantum and wave scattering theory. Let $\Omega\subset\mathbb{R}^n\,(n=2,3)$ be a bounded Lipschitz domain with connected complement $\mathbb{R}^n \backslash \bar{\Omega}$. In the physical context, the scattering of an incident field denoted by $w^i$ by the inhomogeneous medium $\Omega$ can indeed describe both quantum scattering and acoustic scattering, as given by
\begin{equation}\label{eym:Sch}
\begin{cases}
\hat{H} w-k^2w=0 &\mbox{in}\quad \mathbb{R}^n,\\
w =w^i+w^s &\mbox{in}\quad \mathbb{R}^n,\\
\lim\limits_{r\rightarrow \infty}r^{\frac{n-1}{2}}(\partial_r-\mathrm{i}k)w^s=0,\ & r=|\boldsymbol{x}|,
\end{cases}
\end{equation}
where \(k\) is interpreted as the wave number,  \(w\) and \(w^s\) are referred to as the total field and the scattered field, respectively. In quantum scattering, the operator
\begin{equation*} 
	\hat{H}:=-\Delta-V
	\end{equation*} is known as Schrh\"{o}dinger operator, where the potential $V$ is a complex function that is extended to zero outside the scatterer $(\Omega; V)$ and $k^2$ denotes the energy level (cf.\cite{E2011, Grif, Z2022}). In contrast, in acoustic medium scattering,  the operator is given by
\begin{equation*}
	 \hat{H}:=-\Delta-k^2(q-1),
	\end{equation*}
 where $q$ denotes a refractive index  with $\mathrm{supp}(q-1)\subset \Omega$. In fact, $V$ and $q$ (both in $L^{\infty}(\Omega)$) are referred to as physical configurations in different scenarios, each carrying distinct physical significance. The last limit is known as the Sommerfeld radiation condition, which holds uniformly in the angular variable $\hat{\boldsymbol{x}}:=\boldsymbol{x}/r\in\mathbb{S}^{n-1}$. The scattered field $w^s$ has the following asymptotic expansion as $r\rightarrow +\infty$:
\begin{equation*}
w^s(\boldsymbol{x})=\frac{e^{\mathrm{i}k r}}{r^{(n-1)/2}} w_\infty({\hat{\boldsymbol{x}}})+\mathcal{O}(r^{-(n+1)/2}),
\end{equation*}
where $w_\infty({\hat{\boldsymbol{x}}})$ is known as the far-field pattern. The well-posedness of the direct problem of system \eqref{eym:Sch} can be found in \cite{CK2017,E2011,Mclean2000}. Associated with \eqref{eym:Sch}, we consider the following geometrical inverse problem
\begin{equation}\label{eq:far1}
\mathcal{F}\big((\Omega;V),\,w^i\big)= w_\infty,
\end{equation}
where the operator $\mathcal{F}$ is nonlinear. A notable case for \eqref{eq:far1} occurs when $w_\infty=0$. In this scenario, the scatterer $(\Omega;V)$ does not produce any scattering information detectable from outside, making it effectively invisible or transparent with respect to the probing wave $w^i$. Notably, if $w_\infty\equiv0$, it follows from Rellich’s lemma \cite{CK2017} that $w=w^i$ in $\mathbb{R}^n\backslash\overline{\Omega}$. Thus, we obtain the interior transmission eigenvalue problem given by \eqref{eq:s1} or \eqref{eq:h2} for $u=w|_\Omega$ and $v=w^i|_{\Omega}$. 

In this subsection, we focus on the visibility and unique identifiability of the scatterer $\Omega$ with thin or narrow ends. Firstly, we demonstrate that a scatterer with thin or narrow ends will scatter any incident wave. Secondly, we establish the unique shape identifiability of the inhomogeneous scatterer through a single far-field measurement. This measurement involves collecting far-field data across all observation directions for a fixed incident frequency, rather than relying on multiple far-field measurements. The challenge of determining shape from a single far-field measurement, known as Schiffer's problem, has a rich history in inverse scattering problems (see \cite{CK2017, CK2} and the references therein). To date, uniqueness results derived from a single far-field measurement are valid only when certain a priori information about the scatterer's shape and geometric parameters is available. For further developments on unique identifiability in scattering problems related to conductive acoustic, elastic, and electromagnetic media using a single far-field measurement, refer to \cite{BL2019, BLX2021,  DCL2021, DLS2021}. Specifically, in Theorem \ref{thm:visible}, we shall show that a scatterer with one or more thin or narrow ends is indeed visible. Furthermore, we demonstrate that when the scatterer $\Omega$ belongs to a particular admissible class $\mathcal{A}$, Theorem \ref{thm:iden} states a local uniqueness in determining the scatterer from a single measurement. 

\begin{theorem}\label{thm:visible}
    Consider the scattering problem described by \eqref{eym:Sch}. Let \(\Omega\) be the scatterer characterized by the physical configurations $V$ or $q$, which fulfill the assumptions  \eqref{eq:ap2},  \eqref{eq:ap3} and \eqref{eq:q}, respectively. If \(\Omega\) has one or more thin or narrow ends and the total wave satisfies the local regularity condition \eqref{eq:uv2} in the product-type ends, then \(\Omega\) will always scatter any incident wave.
\end{theorem}

Moving forward, we present the definitions of an admissible set $\mathcal{A}$ for the scatterer $\Omega$.

\begin{definition}\label{def:ad}({\bf  Admissible set $\mathcal{A}$})
	Let \(\Omega \subset \mathbb{R}^n\,(n=2,3)\), be a bounded, simply-connected Lipschitz domain that contains one or more thin or narrow ends $\mathcal{N}_{\varepsilon}$ specified in \eqref{eq:multi}. Assume there are \(M\) thin or narrow ends \(\mathcal{N}_{\varepsilon}^{l}\) corresponding to \(M\) geometric parameters \(\varepsilon_{l}\). Define \(\varepsilon_{\max}\) as the uniform upper bound of these \(M\) geometric parameters \(\varepsilon_{l}\), specifically, \(\varepsilon_{\max} = \max\{\varepsilon_l\,|\, l=1,\dots,M\}\). Consider the scattering problem \eqref{eym:Sch} when \(V\) given by \eqref{eq:ap3}, or when \(V\) is replaced by \(V(x) = k^2(q(x) - 1)\) as specified in \eqref{eq:q}. Then, the scatterer \(\Omega\) is said to be admissible if, for any \(\boldsymbol{x} \in \mathbb{R}^n\), the total wave field  $w \in H^1(\mathcal{N}_\varepsilon) \cap C^{1,\alpha_1}(\overline{\mathcal{N}}_\varepsilon) $, where \(\alpha_1 \in (1/2,1)\) in \(\mathbb{R}^2\) and \(\alpha_1 \in (2/3,1)\) in \(\mathbb{R}^3\), and satisfies 
\begin{equation}\label{eq:max}
    |w(\boldsymbol{x})| > \Oh(\varepsilon_{\max}^{\tilde{\tau}} )\quad \text{with} \quad \varepsilon_{\max} \ll 1.
\end{equation}
Here, 
$\tilde{\tau}:=\min \{\tau_1,\tau_2\}$, where $\tau_1$ and $\tau_2$ are regularity exponents defined in \eqref{eq:k1} and \eqref{eq:kk2}, respectively.
\end{definition}

\begin{remark}
   The admissibility assumption given by \eqref{eq:max} can be satisfied in various physical scenarios. Consider a scenario where \(k \cdot \operatorname{diam}(\Omega) \ll 1\), with \(k\) representing the wave number and \(\operatorname{diam}(\Omega)\) denoting the size of $\Omega$. In this context, the impact of the scatterer on the incident wave can be considered negligible, implying that the scattered wave \(w^s\) is sufficiently small in comparison to the incident wave \(w^i\). Therefore, if the incident wave \(w^i\) cannot be reduced to an extremely small value (e.g., when \(w^i\) is a plane wave), the total wave \(w = w^i + w^s\) should retain a non-negligible magnitude across the entire domain. Furthermore, the local regularity conditions of the total wave field in product-type end $\mathcal{N}_{\varepsilon}$ can generally be satisfied in physical scenarios. We demonstrate that these conditions hold in Appendix \ref{sec:apB}.

\end{remark}
\begin{theorem} \label{thm:iden}
    Consider the scattering problem \eqref{eym:Sch} with the potential \( V \) satisfying assumptions \eqref{eq:ap2} and \eqref{eq:ap3}, or with \( V \) replaced by \( V(x) = k^2 (q(x) - 1) \), where the refractive index \( q(x) \) satisfies assumption \eqref{eq:q}. These conditions ensure the well-posedness of the scattering problem for both quantum and acoustic systems, as discussed in Subsection \ref{sub:visi}. Let \(\Omega ,\,\widetilde{\Omega} \subset \mathbb{R}^n\, (n=2,3)\) be two admissible scatterers. For a fixed incident wave \(w^i = e^{\mathrm{i}k\boldsymbol{x} \cdot \boldsymbol{d}}\), assume that \(w_{\infty}\) and \(\widetilde{w}_{\infty}\) are the corresponding far-field patterns of \(\Omega\) and \(\widetilde{\Omega}\), respectively. If
    \begin{equation}\label{eq:far}
        w_{\infty}(\hat{\boldsymbol{x}}) = \widetilde{w}_{\infty}(\hat{\boldsymbol{x}}), \quad \hat{\boldsymbol{x}} \in \mathbb{S}^{n-1},
    \end{equation}
    then \(\Omega\Delta\widetilde{\Omega}:= (\Omega\setminus\widetilde{\Omega}) \cup (\widetilde{\Omega} \setminus\Omega)\) must not possess one or more thin or narrow ends $\mathcal{N}_{\varepsilon} \subset\Omega\Delta\widetilde{\Omega}$ with the property that for every point $\boldsymbol {x}'$ on the lateral boundary of $\mathcal{N}_{\varepsilon}$, there exists an unbounded path $\Upsilon \subset \mathbb{R}^n\setminus(\Omega\cup \widetilde{\Omega})$ connecting $\boldsymbol {x}'$ to infinity.
\end{theorem}

\subsection{Technical developments and discussion}
It is known that transmission eigenfunctions are closely linked to the presence or absence of non-scattering phenomena, particularly the potential for invisibility. In recent years, it has been observed that when the medium exhibits geometric singularities and the characterized physical configuration $q$ has singularities at the boundary of the domain, the scatterer can be visible to any incident wave for acoustic medium scattering. Significant interest has emerged in quantitatively characterizing the spectral geometry of transmission eigenfunctions, particularly concerning two distinct types of singularities: geometric singularities of the underlying domain and singularities in the physical configuration. Related studies have attracted significant attention and made substantial advances, as highlighted in \cite{BL2017, BL2, CV2023, SS2021} and the references cited therein. Notably, \cite{BL2017} first discovered that acoustic transmission eigenfunctions generically vanish near corner points, where singularities arise in both the physical configuration and the geometry of the underlying domain. Indeed, the geometrical singularity is restricted to right corners, implying that a medium scatterer containing right corners cannot be invisible. Later, \cite{BL2} finds that interior transmission eigenfunctions exhibit near-vanishing behavior at specific high-curvature points on the boundary, which may be very smooth or even analytic. Recently, there have been results on the scattering and non-scattering of inhomogeneities, as discussed in \cite{CV2023, SS2021}. The study \cite{CV2023} reveals that a medium scatterer cannot be transparent because of geometric and physical singularities. The geometric singularity occurs at a non-analytic point on the boundary, while the physical singularity relates to a real-analytic configuration. Furthermore, \cite{SS2021} indicates that the domain must be either a regular domain or a quadrature domain when the incident wave does not vanish. Collectively, these results emphasize the phenomenon of local vanishing around specifically distinctive geometrical points on the boundary.

While earlier research on the occurrence or non-occurrence of invisibility of the medium scatterer, as discussed in \cite{ B2018, BL2019, BL, BLX2021, BP2013, BPL2014, CakoX2021, PS2017}, shares similarities with the current perspective, particularly regarding the vanishing or near-vanishing of wave fields at geometrically singular points on the boundary, there are subtle and technical distinctions between the two viewpoints. Analyzing the singularities of transmission eigenfunctions can yield deeper insights. Specifically, the essential nature of the singularity in eigenfunctions arises from the singularity of the physical configuration rather than from the singularity of the geometric region itself. In this paper, we introduce a more general coupled PDE system \eqref{eq:sym1} and investigate the local pattern formations of these equations. Our main results extend the investigation of near-vanishing properties for transmission eigenfunctions in \eqref{eq:s1} from boundary points to interior points in specific domains, where the points within the domain can be regarded as extrinsic points of the high-curvature regions, as discussed in Section \ref{sub:main results}. Under prior knowledge of $u,\,v,\,f(\cdot)$ and  $ g(\cdot)$, and assuming the domain is sufficiently small in specific dimensions (e.g., thin or narrow geometries), we can quantitatively characterize the local pattern formations of the nonlinear coupled PDEs both on the boundary and within the region. This study encounters several challenges, notably the nonlinear coupling between the right-hand sides $f(\boldsymbol{x},u,v)$ and $g(\boldsymbol{x},u,v)$, as well as the boundary conditions, and the geometric complexity of the domain. It is essential to emphasize that the near-vanishing of the difference of these nonlinear right-hand sides in \eqref{eq:sym1} (see Theorem \ref{thm:1}) arises not only at the boundaries of the domains but also within their interiors, a novel aspect that has not been previously explored. Existing results regarding the vanishing or near-vanishing of transmission functions \cite{BL2017, BL2, CV2023}, conductive transmission eigenfunctions \cite{DCL2021}, elastic transmission eigenfunctions \cite{BL2019, DLS2021}. Electromagnetic transmission eigenfunctions \cite{BLX2021} predominantly focus on boundary points, where singularities in the geometry of the domain may arise. More significantly, our analysis reveals that geometric singularities of the underlying domain play a secondary role compared to the singularities in the physical configurations of the coupled system, which emerge as the dominant factor governing the observed phenomena. Notably, when the coupled right-hand sides $f(\boldsymbol{x},u,v)=\left(\lambda+V\right) u$ and $ g(\boldsymbol{x},u,v)=\lambda v$  and a non-trivial solution $(u,\,v)$ exists, the system \eqref{eq:sym1} reduces to the interior transmission eigenfunctions problem. We demonstrate that these eigenfunctions $u$ and $v$ exhibit near-vanishing behavior in sufficiently narrow subregions, as shown in Corollaries \ref{Cor:1} and  \ref{Cor:2} for different physical configurations. It is worth emphasizing that the phenomenon of eigenfunctions near-vanishing at points inside the region has not been explored in previous studies. Specifically, the results we obtained are based not only on a more complex model and novel methods but also have profound significance and important applications in areas such as inverse boundary value problems, scattering, and invisibility. Moreover, the study of local patterns within the interior of the region is particularly novel.

 To achieve these new results, we develop innovative technical strategies. Specifically, we employ microlocal tools to quantitatively characterize the near-vanishing of the difference of the coupled functions within specific regions, such as thin or narrow ends. Additionally, we utilize Complex Geometric Optics (CGO) solutions of the partial differential operator (PDO) in our analysis, which facilitates various subtle technical estimates and asymptotic evaluations. Unlike most existing research that focuses on distinctive geometric points at the boundary, our work examines the coupling of $u$ and $v$ both in the interior and along partial boundaries of the domain, without relying on transmission conditions. The complexity of this coupled system introduces significant challenges in both analysis and estimation. By meticulously analyzing the asymptotic behavior of the CGO solutions, we can achieve our desired results.

The following sections of this paper are organized as follows: In Section \ref{sec:pre}, we present a coordinate representation of product-type ends $\mathcal{N}_{\varepsilon}$ and derive preliminary results for further analysis. In Sections \ref{sec:2} and \ref{sec:4}, we investigate the local patterns of the general nonlinear coupled PDEs described in \eqref{eq:sym1} for both 2D and 3D cases. Section \ref{sec:6} primarily presents proofs of the theorems and corollaries outlined in Section \ref{sec:applications}. Finally, Appendix \ref{sec:apA} demonstrates that the geometric condition \textbf{Assumption G} can be satisfied within product-type ends $\mathcal{N}_{\varepsilon}$, and Appendix \ref{sec:apB} establishes validity of regularity conditions on $u$ and $v$ in Corollary \ref{Cor:2} for the transmission eigenvalue problem \eqref{eq:h2}.

\section{preliminaries}\label{sec:pre}
\subsection{Coordinate representation of thin or narrow ends}\label{sub:parameterized}
In this subsection, we provide coordinate representations of thin or narrow ends $\mathcal{N}_\varepsilon$ that are relevant to our research. For simplicity, we focus on a specific class of simple curves $\eta(t)$ located in the coordinate plane to study the local pattern formations of coupled PDEs in \eqref{eq:sym1}; however, we assert that our results also apply to more general classes of curves. 

Let $\Omega_{\varepsilon} \subset \mathbb{R}^{n-1}$ ($n=2,3$) denote a bounded, simply-connected Lipschitz domain whose boundary $\partial\Omega_{\varepsilon}$
 is piecewise smooth and parameterizable. We establish the following geometric assumptions:
\begin{enumerate}
    \item [(i)] there exists a unique \(t_0 \in I\) such that the curve \(\eta \colon I \to \mathbb{R}^n\) intersects the cross-section \(\Omega_{\varepsilon}\) at only one point. Namely,
          \[
              \eta(t) \cap \Omega_{\varepsilon} = \{\eta(t_0)\},
          \]
          where \(\eta(t_0)\) represents the highest point of \(\Omega_{\varepsilon}\) along the \(x_n\)-axis. Furthermore, the projection of \(\Omega_{\varepsilon}\) onto the \(x_n\)-axis possesses a length of \(\varepsilon\);
    \item[(ii)] the curve \(\eta\) is injective, which implies that \(\eta(t_1) \neq \eta(t_2)\) for any distinct \(t_1, t_2 \in I\);
    \item[(iii)] for thin ends, the diameter of \(\Omega_{\varepsilon}\) is defined as \(\mathrm{diam}(\Omega_{\varepsilon}) = \varepsilon\). For narrow ends, \(\Omega_{\varepsilon}\) represents a rectangle with width \(\varepsilon\) and conventional length, where \(\varepsilon\) is taken to be sufficiently small in both cases.
\end{enumerate}
Therefore, the coordinate representation of \(\mathcal{N}_{\varepsilon}\) in \eqref{eq:N} is
\begin{align*}
\mathcal{N}_{\varepsilon}= \Big\{ (x_1,\dots,x_n) \in \mathbb{R}^n \;\Big|\; (x_1,\dots,x_{n-2},x_n) \in \Omega_\varepsilon, \ x_{n,\min}= x_n(t) - \varepsilon,\ x_{n,\max}=x_n(t) \Big\}.
\end{align*}
Here, we set $x_{n-1} =t$ and define $x_n(t) := \gamma(t)$ (representing a graph over $t$).  The mapping $\eta(t) = \big({\bf 0},t,\gamma(t)\big)$ characterizes the boundary geometry of $\mathcal{N}_{\varepsilon}$. Furthermore, we assume that 
\begin{equation}\label{eq:gamm}
		\max|\gamma(t)|=\Oh(\varepsilon).
	\end{equation}

It is crucial to emphasize that for a general curve \(\eta(t)\) not lying in the coordinate plane, we need to project the infinitesimal element of the curve orthogonally onto the normal line of \(\Omega_\varepsilon\). The subsequent processing steps are analogous to those for curves or surfaces in the coordinate plane; however, the calculations are notably complex and technical. Therefore, we will focus exclusively on the aforementioned curve.

\subsection{Preliminary results   }
In this subsection, we present some preliminary results that shall be frequently used in our subsequent analysis. We begin by introducing the CGO solutions.

\begin{lem}
 
	Let \begin{equation}\label{eq:cgo}
		u_0(\boldsymbol{x}) =  e^{\rho \cdot \boldsymbol{x}}, \ 	\rho = s (\boldsymbol{d} + \bsi\boldsymbol{d}^{\perp}) , \ \boldsymbol{x} \in  \mathbb{R}^n, n=2,3,
	\end{equation}
where 
$ \boldsymbol{d},\,\boldsymbol{d}^{\perp} \in \mathbb{S}^{n-1}$ satisfy $\boldsymbol{d} \cdot \boldsymbol{d}^{\perp}=0$, and $s\in \mathbb{R}_+$.
Then 
	\begin{equation*}
		\Delta u_0 (\boldsymbol{x})  = 0 .
	\end{equation*}
\end{lem}

Next, we will review a specific type of Green's formula for $H^1$ functions, which will be necessary to establish a crucial integral identity in our subsequent analysis.
\begin{lem}\cite{Costabel} \label{lem:Green}
	Let $\Omega \subset \mathbb{R}^n$ be a bounded Lipschitz domain. For any $f,g \in H^1_{\Delta}:=\{f\in H^1(\Omega)|\Delta f \in L^2(\Omega)\}$, then the following Green formula holds 
	\begin{equation}\label{eq:green}
		\int_{\Omega} (g \Delta f - f \Delta g) \rmd \boldsymbol{x} = \int_{\partial \Omega} (g \partial_{\nu} f - f \partial_{\nu}g)  \rmd \sigma,
	\end{equation}
	where $\partial_{\nu}f: = \nabla f \cdot \nu |_{\partial \Omega}\in H^{-\frac{1}{2}}(\partial\Omega)$,
    and $\nu$ is the exterior normal derivative of $\partial \Omega$.
\end{lem}

The subsequent lemma establishes fundamental estimates within $\mathcal{N}_{\varepsilon}$ that are essential to analyze its geometric properties.
\begin{lem}
Let $(u,v )$ be a solution to the coupled system \eqref{eq:sym1} with configurations \eqref{eq:ap1}--\eqref{eq:fg}. Define
	\begin{equation}\label{eq:uv}
		\tilde{u}:=h_1^{1/2}u,\ \ \tilde{v}:=h_2^{1/2}v.
	\end{equation}
Then,  we obtain the following system 
\begin{equation}\label{eq:wu1}
			\begin{cases}
				 -\Delta \tilde{w} - \big( \frac{1}{4} h_1^{-2}|\nabla h_1|^2- \frac{1}{4} h_2^{-2}|\nabla h_2|^2 + \frac{1}{2} h_2^{-1}\Delta h_2 - \frac{1}{2} h_1^{-1}\Delta h_1\big) \tilde{v} \\ \hspace{0.1cm} - \big(\frac{1}{4} h_1^{-2}|\nabla h_1|^2-\frac{1}{2}h_1^{-1}\Delta h_1\big)\tilde{w} = h_1^{-1/2}f(\boldsymbol{x},h_1^{-1/2}\tilde{u},h_2^{-1/2}\tilde{v})- h_2^{-1/2}g(\boldsymbol{x},h_1^{-1/2}\tilde{u},h_2^{-1/2}\tilde{v})  & \mbox{in}\ \mathcal{N}_\varepsilon,\medskip\\
			\hspace{0.38cm}	\tilde{w}=\frac{h_1- h_2 }{h_1^{1/2}h_2^{1/2}+h_2}\tilde{v}  & \mbox{on}\ \hat{\Gamma},\medskip\\
			\partial_{\nu}\tilde{w}= \frac{h_2- h_1 }{h_1^{1/2}(h_1^{1/2}+h_2^{1/2})}\partial_{\nu}\tilde{v}- \frac{1}{2}h_1^{-1/2} h_2^{-1/2} (\nabla h_2-\nabla h_1)\cdot \nu \tilde{v} & \mbox{on}\ \hat{\Gamma},
			\end{cases}
		\end{equation}
where $\tilde{w}:=\tilde{u}-\tilde{v}$, $\mathcal{N}_{\varepsilon} \subset \Omega$ and $\hat{\Gamma}$ are defined in \eqref{eq:N} and \eqref{eq:lateral}, respectively. Furthermore, the following estimates hold:
	\begin{equation}\label{eq:w1}
		\|\tilde{w}\|_{L^{\infty}(\mathcal{N}_\varepsilon)} \leq C(M_1,C_1)(\varepsilon^{\alpha_3} + \varepsilon^{1+\alpha_1}),\,\,  \|\partial_{\nu} \tilde{w}\|_{L^{\infty}(\mathcal{N}_\varepsilon)} \leq  C(M_1,C_1)\varepsilon^{\alpha_1},
	\end{equation}
where $\alpha_1$ and $\alpha_3$ are as specified in Theorem \ref{thm:1}, and $C(M_1, C_1)$ is a positive constant independent of $\varepsilon$.

Particularly, for the case where \(h_1=h_2=1\) and \(f(\boldsymbol{x})=k^2q u(\boldsymbol{x})\) and \(g(\boldsymbol{x})=k^2v(\boldsymbol{x})\) in system \eqref{eq:sym1} (as specified in Corollary \ref{Cor:2}), the system \eqref{eq:h2} transforms into the following system
\begin{equation}\label{eq:wu0} 
\begin{cases} 
\Delta w +k^2 q w = k^2 (1-q)v\quad & \mbox{in}\ \mathcal{N}_\varepsilon,\medskip\\ 
w=0, \ \partial_{\nu} w = 0\quad & \mbox{on}\ \hat{\Gamma}, 
\end{cases} 
\end{equation} 
where \( w:=u-v \). 
Then, it follows that  
\begin{equation}\label{eq:w} 
\|w\|_{L^{\infty}(\mathcal{N}_\varepsilon)} \leq C(C_1)\varepsilon^{1+\alpha_1}, \ \|\nabla w\|_{L^{\infty}(\mathcal{N}_\varepsilon)} \leq C(C_1)\varepsilon^{\alpha_1}, 
\end{equation} 
where \( C_1 \) is denoted in \eqref{eq:ap1}.
\end{lem}

\begin{proof}
We first prove \eqref{eq:w1} holds. By substituting \eqref{eq:uv}, and $\tilde{w}$ into \eqref{eq:sym1}, we can derive \eqref{eq:wu1}.  Since $u, v \in C^{1,\alpha_1}(\overline{\mathcal{N}}_\varepsilon)$ are defined as in \eqref{eq:ap1}, and $h_1,h_2 \in C^{2,\alpha_2}(\overline{\mathcal{N}}_\varepsilon)$ are defined in \eqref{eq:h}, with the additional assumption that $\alpha_1 \leq \alpha_2$, it follows that $\tilde{u}, \tilde{v} \in C^{1,\alpha_1}(\overline{\mathcal{N}}_{\varepsilon})$. Then, for any $\boldsymbol{x} \in \mathcal{N}_\varepsilon$, the following expansions hold:
		\begin{align}\label{eq:w exp t}
		\tilde{w}(\boldsymbol{x})=  \tilde{w} (\boldsymbol{x}_0) + \nabla \tilde{w}(\boldsymbol{x}_0) \cdot (\boldsymbol{x} - \boldsymbol{x}_0) + \tilde{\delta} \tilde{w},\quad  |\tilde{\delta}  w| \le C_1 |\boldsymbol{x}- \boldsymbol{x}_0|^{1+\alpha_1},
	\end{align}
	and 
	\begin{align}\label{eq:w deri t}
		\nabla	\tilde{w}(\boldsymbol{x})= \nabla \tilde{w} (\boldsymbol{x}_0) + \tilde{\delta} \nabla w ,\quad  |\tilde{\delta} \nabla \tilde{w}| \le C_1 |\boldsymbol{x}- \boldsymbol{x}_0|^{\alpha_1},
	\end{align}
where $\boldsymbol{x}_0 \in \overline{\mathcal{N}}_\varepsilon$ and $C_1$ is defined in \eqref{eq:ap1}. Specifically, it is reasonable to choose $\boldsymbol{x}_0 \in \hat{\Gamma}$ in \eqref{eq:w exp t} and \eqref{eq:w deri t} by applying $\varepsilon \ll 1$. Hence, by combining the boundary conditions presented in \eqref{eq:wu1} and utilizing \eqref{eq:w exp t} and \eqref{eq:w deri t}, we obtain
	\begin{align}\label{eq:w exp ab}
			|\tilde{w}(\boldsymbol{x})|
			&\leq \left| \frac{h_1- h_2 }{h_1^{1/2}h_2^{1/2}+h_2} \tilde{v} (\boldsymbol{x}_0)\right| + |\nabla \tilde{w}(\boldsymbol{x}_0) \cdot \nu ||\boldsymbol{x} - \boldsymbol{x}_0| + C_1 |\boldsymbol{x}- \boldsymbol{x}_0|^{1+\alpha_1},\notag\\
			|\nabla	\tilde{w}(\boldsymbol{x})\cdot \nu|
			&\leq \left|\frac{h_2- h_1 }{h_1^{1/2}(h_1^{1/2}+h_2^{1/2})}\partial_{\nu}\tilde{v} (\boldsymbol{x}_0)- \frac{1}{2}h_1^{-1/2} h_2^{-1/2} (\nabla h_2-\nabla h_1)\cdot \nu \,\tilde{v}(\boldsymbol{x}_0) \right|+  C_1 |\boldsymbol{x}- \boldsymbol{x}_0|^{\alpha_1}.
		\end{align}
By observing \eqref{eq:w exp t}, \eqref{eq:w exp ab} and Theorem \ref{th:cover}, we know that there must exist $\boldsymbol{x}_0 \in \hat{\Gamma}$ such that
\begin{equation*}
\nabla \tilde{w}(\boldsymbol{x}_0) \cdot (\boldsymbol{x} - \boldsymbol{x}_0) = \nabla \tilde{w}(\boldsymbol{x}_0) \cdot \nu |\boldsymbol{x} - \boldsymbol{x}_0|, \quad \forall \ \boldsymbol{x} \in \mathcal{N}_{\varepsilon}.
\end{equation*}
Here, $\nu$ is the exterior unit normal vector to $\hat{\Gamma}$. Next, by combining the geometric construction $\mathcal{N}_{\varepsilon}$ (where $\mathcal{N}_{\varepsilon} := \mathcal{N}_{\varepsilon}^T$ for the thin end case and $\mathcal{N}_{\varepsilon} := \mathcal{N}_{\varepsilon}^N$ for the narrow end case) and applying \textbf{Assumption G}, we know that for any point $\boldsymbol{x} \in \mathcal{N}_{\varepsilon}$, there exists a point $\boldsymbol{x}_0 \in \hat{\Gamma}$ such that the direction of $\boldsymbol{x}_0 - \boldsymbol{x}$ coincides with the direction of $\boldsymbol{\nu}_{\boldsymbol{x}_0}$. Specifically, 
\[
\frac{\boldsymbol{x}_0 - \boldsymbol{x}}{\|\boldsymbol{x}_0 - \boldsymbol{x}\|} = \boldsymbol{\nu}_{\boldsymbol{x}_0},
\]
where $\boldsymbol{\nu}_{\boldsymbol{x}_0}$ is the exterior unit normal vector to $\hat{\Gamma}$ at $\boldsymbol{x}_0$. Furthermore, we get
 \begin{equation}\notag  
 \forall \ \boldsymbol{x} \in \mathcal{N}_{\varepsilon}, \ \exists \ 0 \leq r \leq \varepsilon, \ \mbox{such} \mbox{ that} \ B(\boldsymbol{x},r) \ \mbox{is tangent to } \hat{\Gamma} \ \mbox{at the point} \ \boldsymbol{x}_0.  
 \end{equation}  
 This implies that 
 \begin{equation}\label{eq:xx0} 
 |\boldsymbol{x} - \boldsymbol{x}_0| \le \varepsilon. 
 \end{equation} 
Thus, by incorporating \eqref{eq:w exp ab}--\eqref{eq:xx0}, and the boundary estimates for $h_1$ and $h_2$ as presented in \eqref{eq:h}, we can derive \eqref{eq:w1}.

We now prove that \eqref{eq:w} holds. Similar to \eqref{eq:w exp t} and \eqref{eq:w deri t}, for any $ \boldsymbol{x}\in \mathcal{N}_\varepsilon$, the following expansions hold:
\begin{align}\label{eq:w exp} 
w(\boldsymbol{x})=  w (\boldsymbol{x}_0) + \nabla w(\boldsymbol{x}_0) \cdot (\boldsymbol{x} - \boldsymbol{x}_0) + \tilde{\delta} w,\quad  |\tilde{\delta}  w| \le C_1 |\boldsymbol{x}- \boldsymbol{x}_0|^{1+\alpha_1}, 
\end{align} 
and  
\begin{align}\label{eq:w deri} 
\nabla	w(\boldsymbol{x})= \nabla w (\boldsymbol{x}_0) + \tilde{\delta} \nabla w ,\quad  |\tilde{\delta} \nabla w| \le C_1 |\boldsymbol{x}- \boldsymbol{x}_0|^{\alpha_1}, 
\end{align} 
where $\boldsymbol{x}_0 \in \hat{\Gamma}$ and $C_1$ is described in \eqref{eq:ap1}. From the boundary conditions in \eqref{eq:wu0}, we deduce
\begin{equation}\label{eq:deri}
    \nabla w(\boldsymbol{x}_0) = \boldsymbol{0}, \quad \boldsymbol{x}_0\in \hat{\Gamma}.
\end{equation}
Therefore, combining \eqref{eq:xx0} with \eqref{eq:w exp}--\eqref{eq:deri}, we obtain \eqref{eq:w}.
\end{proof}

\section{Proof of Theorem \ref{thm:1} in 2D}\label{sec:2}
This section presents a detailed proof of Theorem \ref{thm:1}. It is known that the operator $\Delta$ is invariant under rigid motions. We adopt a localization strategy beginning with analysis in a subregion of $\mathcal{N}_{\varepsilon}$. The remaining portions of $\mathcal{N}_{\varepsilon}$ can then be treated through either coordinate transformations or domain translations. Without loss of generality, we employ the technique of shifting $\mathcal{N}_{\varepsilon}$.
Assume that $a>0$ and let 
\begin{align}
	\widetilde{D}_{\varepsilon}& :=\Omega_\varepsilon \times \eta(x_1) \subset \mathcal{N}_\varepsilon, \ x_1 \in ( a, a +  \varepsilon^\ell)\subset I, \ \ \mbox{with} \ \	\ell<\alpha_1,
\label{eq:mini1}
\end{align}
 denote a subregion of $\mathcal{N}_{\varepsilon}$, where $I$ and $\alpha_1$ are defined in \eqref{eq:n1} and \eqref{eq:ap1}, respectively. Analogous to \eqref{eq:lateral}, $\widetilde{\Gamma}_{\varepsilon}^3 \cup \widetilde{\Gamma}_{\varepsilon}^4 := \partial \widetilde{D}_{\varepsilon} \times \eta(x_1)$ denotes the lateral boundary of $\widetilde{D}_{\varepsilon}$. Thus, the boundary of $\widetilde{D}_{\varepsilon}$ is given by 
  $$\partial \widetilde{D}_{\varepsilon} =  \widetilde{\Gamma}_{\varepsilon}^1 \cup  \widetilde{\Gamma}_{\varepsilon}^2\cup \widetilde{\Gamma}_{\varepsilon}^3\cup \widetilde{\Gamma}_{\varepsilon}^4,$$
  where
	\begin{align*}
		\widetilde{\Gamma}_{\varepsilon}^1 &:=\Big\{(x_1,x_2)\mid x_1=a, \ x_2 \in \big( \gamma(a)-\varepsilon,\gamma(a)\big)\Big\},\notag\\
		\widetilde{\Gamma}_{\varepsilon}^2 &:=\left\{(x_1,x_2)\mid x_1=a + \varepsilon^{\ell}, \ x_2 \in \big(\gamma(a + \varepsilon^{\ell})-\varepsilon, \gamma(a + \varepsilon^{\ell})\big)\right\},\notag\\
		\widetilde{\Gamma}_{\varepsilon}^3 &:=\left\{(x_1,x_2)\mid  x_1\in (a,  a+\varepsilon^{\ell}), \ x_2=\gamma(x_1)\right\},\notag\\
		\widetilde{\Gamma}_{\varepsilon}^4 &:=\left\{(x_1,x_2)\mid x_1\in (a,  a+\varepsilon^{\ell}), \ x_2=\gamma(x_1)-\varepsilon\right\},
	\end{align*}
and $\gamma(x_1)$ satisfies \eqref{eq:gamm}. 

Next, we shift the entire region $\mathcal{N}_{\varepsilon}$ to the left by $a-$$\varepsilon^m$ units, as our approach emphasizes localized analysis. Thus, $\widetilde{D}_{\varepsilon}$, as defined in \eqref{eq:mini1}, can be expressed more precisely as
\begin{align}
	D_{\varepsilon}& :=\Omega_\varepsilon \times \eta(x_1), \ x_1 \in ( \varepsilon^m, \varepsilon^m +  \varepsilon^\ell)\subset (-L-a+\varepsilon^m,L-a+ \varepsilon^m), \ \mbox{with} \ 	\ell<\alpha_1\,\,\mbox{and} \,\, m> \frac{1}{2}.
	\label{eq:mini}
\end{align}
Similarly, $\Gamma_{\varepsilon}^3 \cup \Gamma_{\varepsilon}^4 = \partial \Omega_\varepsilon \times \eta(x_1)$ denotes the lateral boundary of $D_{\varepsilon}$. Therefore, the boundary of $D_{\varepsilon}$ is given by
$$\partial D_{\varepsilon} =  \Gamma_{\varepsilon}^1 \cup  \Gamma_{\varepsilon}^2\cup \Gamma_{\varepsilon}^3\cup \Gamma_{\varepsilon}^4,$$
where
\begin{align}
	\Gamma_{\varepsilon}^1 &:=\left\{(x_1,x_2)\mid x_1= \varepsilon^m, \ x_2 \in \left(\gamma(a)- \varepsilon , \gamma(a)\right)\right\},\notag\\
	\Gamma_{\varepsilon}^2 &:=\left\{(x_1,x_2)\mid x_1=\varepsilon^m + \varepsilon^{\ell}, \ x_2\in \big(\gamma( a + \varepsilon^{\ell})- \varepsilon , \gamma( a+\varepsilon^{\ell})\big)\right\}, \notag\\
	\Gamma_{\varepsilon}^3 &:=\left\{(x_1,x_2)\mid  x_1 \in (\varepsilon^m , \varepsilon^m+\varepsilon^{\ell}), \ x_2=\gamma(x_1 +a-\varepsilon^m)\right\},\notag\\
	\Gamma_{\varepsilon}^4 &:=\left\{(x_1,x_2)\mid x_1 \in (\varepsilon^m , \varepsilon^m +\varepsilon^{\ell}), \ x_2=\gamma(x_1+a- \varepsilon^m)-\varepsilon\right\},\label{eq:Gamma12}
\end{align}
and $\gamma$ satisfies \eqref{eq:gamm} as we are only considering the horizontal translation of $\widetilde{D}_{\varepsilon}$; please refer to Figure \ref{fig:1} for a schematic illustration of $\widetilde{D}_{\varepsilon}$ and $D_{\varepsilon}$.

\begin{figure}[htbp]
	\centering
	\includegraphics[scale=0.8]{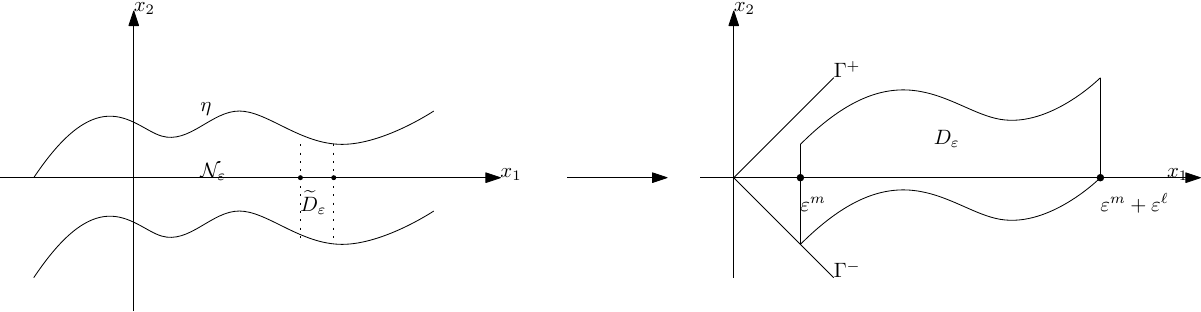}
	\caption{Schematic illustration of the thin end $\mathcal{N}_{\varepsilon}$ and its local region $D_{\varepsilon}$ in 2D}
	\label{fig:1}
\end{figure}
Next, we introduce some notation for the geometric setup. Let $(r, \theta)$ denote the polar coordinates of $\boldsymbol{x}$ in $\mathbb{R}^2$; we have $r = |\boldsymbol{x}|$ and $\theta = \arg (x_1 + \bsi x_2)$. Consider the strictly convex sector
\begin{equation*}
W=\left\{\boldsymbol{x}\in \mathbb{R}^2; \ \theta_1\leq \arg (x_1 + \bsi x_2) \leq \theta_2\right\},
\end{equation*}
whose boundaries are the half-lines $\Gamma^-$ and $\Gamma^+$. Additionally, $-\pi < \theta_1 < \theta_2 < \pi$ satisfies $0 < \theta_2 - \theta_1 < \pi$, with $\theta_1$ and $\theta_2$ representing the polar angles of $\Gamma^-$ and $\Gamma^+$.
Define
\begin{equation}\label{eq:phi}
	\boldsymbol{d}:= (\cos \phi, \sin \phi)^{\top} \in \mathbb{S}^1, \quad  \theta_2 + \frac{\pi}{2}< \phi <\theta_1 + \frac{3\pi}{2},
\end{equation}
in \eqref{eq:cgo} satisfies that \begin{equation} \label{eq:condition}
	 \boldsymbol{d} \cdot \boldsymbol{x} = \boldsymbol{d} \cdot \hat{\boldsymbol{x}} \ |\boldsymbol{x}| \leq - \delta \ |\boldsymbol{x}|<0, \quad \forall \ \boldsymbol{x}\in D_{\varepsilon},
\end{equation}
where $ \delta \in (0, 1]$ is a constant depending on $D_{\varepsilon}$ and $\boldsymbol{d}$.

To begin, we will derive several integral identities and estimates in the subsequent lemmas, which shall be used to prove Theorems \ref{thm:1}.
\begin{lem}\label{lem:integ}
	Recall that the CGO solution $u_0$ is defined in \eqref{eq:cgo}. Let $\tilde{w}=\tilde{u}-\tilde{v},\, w=u-v$, where $u$ and $v$ are defined in \eqref{eq:ap1}, $\tilde{u}$ and $\tilde{v}$ are specified in \eqref{eq:uv}.
	\begin{enumerate}
		\item [(a)] Under the same configuration as Theorem \ref{thm:1} and define $F(\boldsymbol{x}, (h_1^{-1/2} \tilde{u})(\boldsymbol{x}),(h_2^{-1/2} \tilde{v})(\boldsymbol{x}))= F(\boldsymbol{x}, u(\boldsymbol{x}),v(\boldsymbol{x})):=f(\boldsymbol{x}, u(\boldsymbol{x}),v(\boldsymbol{x}))-g(\boldsymbol{x}, u(\boldsymbol{x}),v(\boldsymbol{x}))$, the following integral identity holds
\begin{equation}\label{eq:integral fg}
		 \int_{D_{\varepsilon}} F(\boldsymbol{x}_0, u(\boldsymbol{x}_0), v(\boldsymbol{x}_0)) u_0 \,h_1^{-1/2}\rmd \boldsymbol{x} = - J_F- J_g - J_{\tilde{w}} - J_{\tilde{v}} + J_b,
		\end{equation}
	where
	\begin{align}\label{eq:g}
		J_b &:= \int_{\partial D_{\varepsilon}} \big(\tilde{w} \partial_{\nu} u_0 - u_0 \partial_{\nu} \tilde{w}\big) \rmd \sigma,\notag\\
		 J_F&:= \int_{D_{\varepsilon}} (F_1+F_2+F_3)  u_0( \boldsymbol{x}) h_1^{-1/2}(\boldsymbol{x}) \rmd \boldsymbol{x},\notag\\
		J_g&:= \int_{ D_{\varepsilon}} \frac{h_2- h_1 }{h_1^{1/2}h_2+h_1 h_2^{1/2}} u_0(\boldsymbol{x}) g(\boldsymbol{x}, u(\boldsymbol{x}),v(\boldsymbol{x}))\rmd \boldsymbol{x}, \notag\\
		J_{\tilde{w}}&:= \int_{ D_{\varepsilon}} \left(\frac{1}{4} h_1^{-2} |\nabla h_1|^2 -\frac{1}{2} h_1^{-1} \Delta h_1 \right)(\boldsymbol{x}) u_0(\boldsymbol{x}) \tilde{w}(\boldsymbol{x})\rmd \boldsymbol{x},\notag\\
		J_{\tilde{v}}&:= \int_{ D_{\varepsilon}}	\left( \frac{1}{4} h_1^{-2}|\nabla h_1|^2- \frac{1}{4} h_2^{-2}|\nabla h_2|^2 + \frac{1}{2} h_2^{-1}\Delta h_2 - \frac{1}{2} h_1^{-1}\Delta h_1\right) u_0(\boldsymbol{x}) \tilde{v}(\boldsymbol{x})\rmd \boldsymbol{x},
	\end{align}
	and
	 \begin{align}\label{eq:F}
		 F(\boldsymbol{x}_0, u(\boldsymbol{x}_0),v(\boldsymbol{x}_0))
         := &f(\boldsymbol{x}_0, u(\boldsymbol{x}_0),v(\boldsymbol{x}_0))-g(\boldsymbol{x}_0, u(\boldsymbol{x}_0),v(\boldsymbol{x}_0)),\notag\\
		F_1(\boldsymbol{x}, u(\boldsymbol{x}),v(\boldsymbol{x}))
        :=& \tilde{\delta} f_{\boldsymbol{x}}(\boldsymbol{x}, u(\boldsymbol{x}),v(\boldsymbol{x}))- \tilde{\delta} g_{\boldsymbol{x}}(\boldsymbol{x}, u(\boldsymbol{x}),v(\boldsymbol{x})),\notag \\
		F_2(\boldsymbol{x}_0, u(\boldsymbol{x}),v(\boldsymbol{x})):= &\tilde{\delta} f_u (\boldsymbol{x}_0, u(\boldsymbol{x}),v(\boldsymbol{x}))- \tilde{\delta} g_u(\boldsymbol{x}_0, u(\boldsymbol{x}),v(\boldsymbol{x})),\notag \\
		F_3(\boldsymbol{x}_0, u(\boldsymbol{x}_0),v(\boldsymbol{x})):=& \tilde{\delta} f_u (\boldsymbol{x}_0, u(\boldsymbol{x}_0),v(\boldsymbol{x}))- \tilde{\delta} g_u(\boldsymbol{x}_0, u(\boldsymbol{x}_0),v(\boldsymbol{x})), 
	\end{align}
with the following estimates: \begin{align}\label{eq:F1}
	|\tilde{\delta} f_{\boldsymbol{x}}(\boldsymbol{x}, u(\boldsymbol{x}),v(\boldsymbol{x}))|, |\tilde{\delta} g_{\boldsymbol{x}}(\boldsymbol{x}, u(\boldsymbol{x}),v(\boldsymbol{x}))| & \leq C(C_2) |\boldsymbol{x}-\boldsymbol{x}_0|^{\zeta}, \notag\\
	|\tilde{\delta} f_u (\boldsymbol{x}_0, u(\boldsymbol{x}),v(\boldsymbol{x})) |, |\tilde{\delta} g_u (\boldsymbol{x}_0, u(\boldsymbol{x}),v(\boldsymbol{x})) | & \leq C(C_1^{\zeta},C_2) |\boldsymbol{x}-\boldsymbol{x}_0|^{\zeta \alpha_1}, \notag\\
	|\tilde{\delta} f_v (\boldsymbol{x}_0, u(\boldsymbol{x}_0),v(\boldsymbol{x})) |, |\tilde{\delta} g_v (\boldsymbol{x}_0, u(\boldsymbol{x}_0),v(\boldsymbol{x})) |& \leq C(C_1^{\zeta},C_2 ) |\boldsymbol{x}-\boldsymbol{x}_0|^{\zeta \alpha_1}.
\end{align}
Here, $C_1,C_2,\zeta$ and $\alpha_1$ are specified in \eqref{eq:ap1} and \eqref{eq:fg1}, respectively. The positive constants $C(C_2)$ and $C( C_1^{\zeta},C_2)$ depend only on $C_1, C_2,\zeta$.
		\item[(b)] Under the same conditions as Corollary \ref{Cor:2}, the following integral identity holds
\begin{equation}\label{eq:integral}
\int_{\Gamma_{\varepsilon}^1 \cup \Gamma_{\varepsilon}^2} \big(w \partial_{\nu} u_0 - u_0 \partial_{\nu} w\big) \rmd \sigma = k^2 \int_{D_{\varepsilon}} (q - 1)v u_0 \rmd \boldsymbol{x} + k^2 \int_{D_{\varepsilon}} q w u_0 \rmd \boldsymbol{x}. 
\end{equation}
Furthermore, let us define \(\varphi := q - 1\). It can be directly shown that
\begin{align}\label{eq:qv}
k^2 \int_{D_{\varepsilon}} (q - 1)v u_0 \rmd \boldsymbol{x}  = & k^2 \Big\{\varphi(\boldsymbol{x}_0) v(\boldsymbol{x}_0) \int_{D_{\varepsilon}} e^{\rho \cdot \boldsymbol{x}} \rmd \boldsymbol{x} + v(\boldsymbol{x}_0) \int_{D_{\varepsilon}} (\varphi(\boldsymbol{x}) - \varphi(\boldsymbol{x}_0)) e^{\rho \cdot \boldsymbol{x}} \rmd \boldsymbol{x} \notag \\
& +  \int_{D_{\varepsilon}}(v(\boldsymbol{x}) - v(\boldsymbol{x}_0)) (\varphi(\boldsymbol{x}) - \varphi(\boldsymbol{x}_0)) e^{\rho \cdot \boldsymbol{x}} \rmd \boldsymbol{x} \Big\} \notag \\
& + \varphi(\boldsymbol{x}_0) \int_{D_{\varepsilon}}(v(\boldsymbol{x}) - v(\boldsymbol{x}_0)) e^{\rho \cdot \boldsymbol{x}} \rmd \boldsymbol{x}.
\end{align}
\end{enumerate} 
\end{lem}

\begin{proof}
By applying \eqref{eq:green} from Lemma \ref{lem:Green}, we obtain 

\begin{equation*}
\int_{D_{\varepsilon}} w \Delta u_0 \rmd \boldsymbol{x} - \int_{D_{\varepsilon}} u_0 \Delta w \rmd \boldsymbol{x} = \int_{\Gamma_{\varepsilon}^1 \cup \Gamma_{\varepsilon}^2} \big(w \partial_{\nu} u_0 - u_0 \partial_{\nu} w\big) \rmd \sigma.
\end{equation*}
Noting that $\Delta w + k^2 q w = k^2 (1-q)v$ from \eqref{eq:wu0}, we can naturally derive \eqref{eq:integral} and \eqref{eq:qv}. Similarly, by incorporating \eqref{eq:green} and \eqref{eq:wu1}, we can derive \eqref{eq:integral fg}.
\end{proof}

The following proposition aims to analyze the information related to $u$, while the preceding lemma offers a deeper examination of $v$. This paper primarily focuses on the properties of $v$, but the analysis of $u$ can be conducted in a manner analogous to that of $v$.
\begin{prop}\label{prop:1}
	Within the framework established by Lemma \ref{lem:integ}, the following identity holds
	\begin{equation}\notag
		\int_{\Gamma_{\varepsilon}^1\cup\Gamma_{\varepsilon}^2} \big(w \partial_{\nu} u_0 - u_0 \partial_{\nu} w\big) \rmd \sigma=k^2 \int_{D_{\varepsilon}} (q-1)u u_0 \rmd \boldsymbol{x} + k^2 \int_{D_{\varepsilon}}  w u_0 \rmd \boldsymbol{x}. 
	\end{equation}
\end{prop}

We are now ready to present the proof of Theorem \ref{thm:1} in 2D.
\begin{proof}[\textbf{Proof of Theorem \ref{thm:1} in 2D}] We will prove this theorem in two steps.

\medskip  \noindent {\bf Step I:} Prove the result established in $D_\varepsilon$. Define \begin{equation*}
		J:=F (\boldsymbol{x}_0,u(\boldsymbol{x}_0),v(\boldsymbol{x}_0)) \int_{D_{\varepsilon}} u_0( \boldsymbol{x}) h_1^{-1/2}(\boldsymbol{x}) \rmd \boldsymbol{x}.
	\end{equation*}  
 For any point $\boldsymbol{x} =(x_1,x_2) \in D_{\varepsilon}$, there exists a point $ \ (\varepsilon^m, x_2^\prime) \in \Gamma_{\varepsilon}^1$ such that
 \[	x_2 = x_2^\prime + \gamma(x_1+a-\varepsilon^m)- \gamma(a).
 \]
By applying \eqref{eq:h}, \eqref{eq:cgo} and the first mean value theorem, we obtain
\begin{align}\label{eq:I5 F}
 &\left|F (\boldsymbol{x}_0,u(\boldsymbol{x}_0),v(\boldsymbol{x}_0)) \int_{D_{\varepsilon}} e^{\rho \cdot \boldsymbol{x}} h_1^{-1/2}(\boldsymbol{x}) \rmd \boldsymbol{x}\right|\notag\\
	 \geq  & M_2 \big| F (\boldsymbol{x}_0,u(\boldsymbol{x}_0),v(\boldsymbol{x}_0))  \int_{D_{\varepsilon}} e^{s (d_1 - \bsi d_2) x_1 } e^{s (d_2 + \bsi d_1) x_2 }\rmd x_1 \rmd x_2\big|\notag\\
	 = & M_2 \left|F (\boldsymbol{x}_0,u(\boldsymbol{x}_0),v(\boldsymbol{x}_0))  \int_{\varepsilon^m}^{\varepsilon^m + \varepsilon^\ell} \int^{\gamma(a)}_{\gamma(a)-\varepsilon}  e^{s (d_1 - \bsi d_2) x_1 } e^{s  (d_2 + \bsi d_1) ( x_2^\prime + \gamma(x_1+a-\varepsilon^m)- \gamma(a)) }\rmd x_2^{\prime} \rmd x_1\right|\notag\\
	 =& M_2\varepsilon e^{-s  d_2 \gamma(a) }\left| F (\boldsymbol{x}_0,u(\boldsymbol{x}_0),v(\boldsymbol{x}_0))   \int_{\varepsilon^m}^{\varepsilon^m + \varepsilon^\ell} 
	e^{s (d_1 - \bsi d_2) x_1 }  e^{s  (d_2 + \bsi d_1) \gamma(x_1+a-\varepsilon^m) } \rmd x_1\right| \notag \\
	&\times  \left|
\int^{\frac{\gamma(a)}{\varepsilon}}_{\frac{\gamma(a)}{\varepsilon}-1} e^{s\varepsilon (d_2 + \bsi d_1) \widetilde{x}_2}\rmd \widetilde{x}_2\right| \notag\\
		 \ge & M_2 \varepsilon e^{-s  d_2 \gamma(a) } \left| F (\boldsymbol{x}_0,u(\boldsymbol{x}_0),v(\boldsymbol{x}_0))   \int_{\varepsilon^m}^{\varepsilon^m + \varepsilon^\ell} 
		e^{s (d_1 - \bsi d_2) x_1 }  e^{s  (d_2 + \bsi d_1) \gamma(x_1+a-\varepsilon^m) } \rmd x_1\right|  \notag \\
		& \times \left|   \mathfrak{R} 
		\int^{\frac{\gamma(a)}{\varepsilon}}_{\frac{\gamma(a)}{\varepsilon}-1} e^{s \varepsilon (d_2 + \bsi d_1) \widetilde{x}_2}\rmd \widetilde{x}_2\right|\notag\\
	  = & M_2 \varepsilon  e^{-s  d_2 \gamma(a) } e^{s \varepsilon d_2 \widetilde{x}_{2,\xi}} \left| \cos (s\varepsilon d_1 \widetilde{x}_{2,\xi})F (\boldsymbol{x}_0,u(\boldsymbol{x}_0),v(\boldsymbol{x}_0)) \right. \notag \\
	  &\left. \times  \int_{\varepsilon^m}^{\varepsilon^m + \varepsilon^\ell} 
	 e^{s (d_1 - \bsi d_2) x_1 }  e^{s  (d_2 + \bsi d_1) \gamma(x_1+a-\varepsilon^m) } \rmd x_1 \right|
	 \notag\\
	  = & M_2 \varepsilon  e^{-s  d_2 \gamma(a) } e^{s \varepsilon d_2 \widetilde{x}_{2,\xi}} \left| \cos (s\varepsilon d_1 \widetilde{x}_{2,\xi})F (\boldsymbol{x}_0,u(\boldsymbol{x}_0),v(\boldsymbol{x}_0))  \right|\left|\varepsilon^{\ell} + \int_{\varepsilon^m}^{\varepsilon^m + \varepsilon^\ell} A \rmd x_1 \right| \notag\\
	\ge &M_2 \varepsilon  e^{-s  d_2 \gamma(a) } e^{s \varepsilon d_2 \widetilde{x}_{2,\xi}} \left| \cos (s\varepsilon d_1 \widetilde{x}_{2,\xi})F (\boldsymbol{x}_0,u(\boldsymbol{x}_0),v(\boldsymbol{x}_0))   \right|\bigg(\varepsilon^{\ell} - \Big| \int_{\varepsilon^m}^{\varepsilon^m + \varepsilon^\ell} A \rmd x_1 \Big|\bigg),
\end{align}
where $\Re(\cdot)$ denotes the real part of $(\cdot)$, $\widetilde{x}_{2,\,\xi} \in \left(\frac{\gamma(a)}{\varepsilon}-1,\frac{\gamma(a)}{\varepsilon}\right)$, and
\begin{align}\label{eq:A}
	A= \sum\limits_{n=1}^{\infty} \frac{s^n \Big[(d_1 - \bsi d_2) x_1 + (d_2+\bsi d_1)\gamma(x_1+a-\varepsilon^m) \Big]^n} {n!}.
\end{align}
For any point $\boldsymbol{x} = (x_1, x_2) \in D_{\varepsilon}$ and point $\boldsymbol{x}_0 = (x_{01}, x_{02}) \in \overline{D}_{\varepsilon}$, applying the geometric setup described in \eqref{eq:mini}, we find
\begin{equation}\label{eq:x0}
	|\boldsymbol{x} -\boldsymbol{x}_0|^{\zeta}=
	\left(|x_1 - x_{01}|^2+ |x_2-x_{02}|^2\right)^{\frac{\zeta}{2}} \leq \varepsilon^{\ell \zeta} \big[1+ \Oh(\varepsilon^{2-2\ell})\big]^{\frac{\zeta}{2}}, 
\end{equation}
based on $|x_2 - x_{02}| \leq \varepsilon$ and $\zeta < 1$. Furthermore, by combining \eqref{eq:h}, \eqref{eq:condition}, \eqref{eq:integral fg}, \eqref{eq:F}, \eqref{eq:F1}, and \eqref{eq:x0}, we obtain
\begin{align}\label{eq:IF}
	|J_F|& = \left|\int_{D_{\varepsilon}} (F_1+F_2+F_3) u_0 (\boldsymbol{x}) h_1^{-1/2}(\boldsymbol{x})\rmd \boldsymbol{x}\right|\notag\\
	& \leq  M_1 \int_{D_{\varepsilon}} \left|\tilde{\delta} f_{\boldsymbol{x}} - \tilde{\delta} g_{\boldsymbol{x}} + \tilde{\delta} f_u - \tilde{\delta} g_u + \tilde{\delta} f_v - \tilde{\delta} g_v \right| e^{-s \delta |\boldsymbol{x}|}\rmd \boldsymbol{x}\notag\\ 
	&\leq C(M_1,C_1,C_2) \int_{D_{\varepsilon}} \left(|\boldsymbol{x}- \boldsymbol{x}_0 |^{\zeta}+ |\boldsymbol{x}- \boldsymbol{x}_0 |^{\zeta \alpha_1}  \right) e^{-s \delta |\boldsymbol{x}|}\rmd \boldsymbol{x}\notag\\ 
	&\leq  C(M_1,C_1,C_2) \bigg[\varepsilon^{\zeta \ell}  \big(1+ \Oh(\varepsilon^{2-2\ell})\big)^{\frac{\zeta}{2}} +\varepsilon^{\zeta \alpha_1 \ell}  \big(1+ \Oh(\varepsilon^{2-2\ell})\big)^{\frac{\zeta \alpha_1}{2}} \bigg] \int_{D_{\varepsilon}}  e^{- s \delta |\boldsymbol{x}|} \rmd \boldsymbol{x} \notag\\
	&\leq  C(M_1,C_1,C_2) \bigg[\varepsilon^{\zeta \ell}  \big(1+ \Oh(\varepsilon^{2-2\ell})\big)^{\frac{\zeta}{2}} +\varepsilon^{\zeta \alpha_1 \ell}  \big(1+ \Oh(\varepsilon^{2-2\ell})\big)^{\frac{\zeta \alpha_1}{2}} \bigg] \int_{\theta_1}^{\theta_2} \int_{\varepsilon^m /2}^{\infty} r e^{- s \delta r} \rmd r \rmd \theta \notag\\
	&\leq  C(M_1,C_1,C_2) \bigg[\varepsilon^{\zeta \ell}  \big(1+ \Oh(\varepsilon^{2-2\ell})\big)^{\frac{\zeta}{2}} +\varepsilon^{\zeta \alpha_1 \ell}  \big(1+ \Oh(\varepsilon^{2-2\ell})\big)^{\frac{\zeta \alpha_1}{2}} \bigg]  \int_{\varepsilon^m s \delta/2}^{\infty} \frac{t e^{-t}}{s^2 \delta^2}   \rmd t \notag \\
	&=  C(M_1,C_1,C_2) \bigg[\frac{\varepsilon^{\zeta \ell}}{s^2} \big(1+ \Oh(\varepsilon^{2-2\ell})\big)^{\frac{\zeta}{2}} + \frac{\varepsilon^{\zeta \alpha_1 \ell}}{s^2}  \big(1+ \Oh(\varepsilon^{2-2\ell})\big)^{\frac{\zeta \alpha_1}{2}} \bigg]  \Gamma(2,\varepsilon^m s \delta/2)  \notag \\
	&\leq  C(M_1,C_1,C_2) \bigg[\frac{\varepsilon^{\zeta \ell}}{s^2}  \big(1+ \Oh(\varepsilon^{2-2\ell})\big)^{\frac{\zeta}{2}} + \frac{\varepsilon^{\zeta \alpha_1 \ell}}{s^2}  \big(1+ \Oh(\varepsilon^{2-2\ell})\big)^{\frac{\zeta \alpha_1}{2}} \bigg]
	e^{- \varepsilon^m s \delta / 4}.
\end{align}
It is important to note that in the last inequality of \eqref{eq:IF}, we utilize the incomplete gamma functions $\gamma_0$ and $\Gamma_0: \mathbb{R}_+ \times \mathbb{R}_+ \to \mathbb{R}$ defined as follows:
\begin{equation}\label{eq:gamma0}
	\gamma_0(s,x)=\int_{0}^{x} e^{-t} t^{s-1} \rmd t \quad\mbox{and}\quad \Gamma_0(s,x)=\int_{x}^{\infty}  e^{-t} t^{s-1} \rmd t
\end{equation}
which satisfy $\gamma_0(s, x) \leq \Gamma_0(s) \leq (s-1)!$ and $\Gamma_0(s, x) \leq 2^s \Gamma(s) e^{-x/2}$. Here, $\Gamma(s)$ represents the ordinary complete gamma function.

Subsequently, by combining \eqref{eq:h}, \eqref{eq:fg1}, \eqref{eq:condition}, and \eqref{eq:g}, we obtain
\begin{align}\label{eq:Jg}
	|J_g|&\leq \int_{ D_{\varepsilon}} \bigg|\frac{h_2- h_1 }{h_1^{1/2}h_2+h_1 h_2^{1/2}} (\boldsymbol{x}) u_0(\boldsymbol{x}) g(\boldsymbol{x}, u(\boldsymbol{x}),v(\boldsymbol{x}))\bigg|\rmd \boldsymbol{x}\notag\\
	&\leq C(M_1,C_2) \varepsilon^{\alpha_3} \int_{D_{\varepsilon}}  e^{- s \delta |\boldsymbol{x}|} \rmd \boldsymbol{x} \notag\\
	&\leq C(M_1,C_2) \frac{\varepsilon^{\alpha_3}}{s^2} e^{- \varepsilon^m s \delta / 4}.
\end{align}
Similarly, through appropriate simplification, we can obtain the following estimate of $J_{\tilde{v}}$ as denoted in \eqref{eq:g}, i.e.,
\begin{align}\label{eq:Jv}
	|J_{\tilde{v}}|
	&\leq C(M_1,C_1) (\varepsilon^{\alpha_3}+ \varepsilon^{\alpha_2}) \int_{D_{\varepsilon}}  e^{- s \delta |\boldsymbol{x}|} \rmd \boldsymbol{x} \leq C(M_1,C_1) \frac{\varepsilon^{\alpha_3}+ \varepsilon^{\alpha_2}}{s^2}\, e^{- \varepsilon^m s \delta / 4}.
\end{align}
After that, by combining \eqref{eq:h}, \eqref{eq:w1}, \eqref{eq:condition}, and \eqref{eq:g}, we deduce that
\begin{align}\label{eq:Jwt}
|J_{\tilde{w}}| & \leq \int_{ D_{\varepsilon}} \left|\bigg[\frac{1}{4} h_1^{-2} |\nabla h_1|^2 -\frac{1}{2} h_1^{-1} \Delta h_1 \bigg] u_0(\boldsymbol{x}) \tilde{w}(\boldsymbol{x})\right|\rmd \boldsymbol{x}\notag\\
& \leq C(M_1,C_1) (\varepsilon^{\alpha_3} + \varepsilon^{1+\alpha_1}) \int_{D_{\varepsilon}}  e^{- s \delta |\boldsymbol{x}|} \rmd \boldsymbol{x}\notag\\
& \leq C(M_1,C_1) \frac{\varepsilon^{\alpha_3} + \varepsilon^{1+\alpha_1}}{s^2} e^{- \varepsilon^m s \delta / 4}.
\end{align}
Next, given that $\Gamma_{\varepsilon}^3$ and $\Gamma_{\varepsilon}^4$ represent the actual boundaries, and by using \eqref{eq:ap1}, \eqref{eq:h}, and the boundary conditions \eqref{eq:wu1}, it is straightforward to obtain the estimates as follows
	\begin{align}\label{eq:lateral boundary}
		\|\partial_{\nu}\tilde{w}\|_{L^{\infty}(\Gamma_{\varepsilon}^3\cup \Gamma_{\varepsilon}^4)}&= \bigg\|\frac{(h_2- h_1) \partial_{\nu}\tilde{v} }{h_1^{1/2}(h_1^{1/2}+h_2^{1/2})}- \frac{1}{2}h_1^{-1/2} h_2^{-1/2} (\nabla h_2-\nabla h_1)\cdot \nu \,\tilde{v}\bigg\|_{L^{\infty}(\Gamma_{\varepsilon}^3\cup \Gamma_{\varepsilon}^4)}\leq C(M_1,C_1) \varepsilon^{\alpha_3},\notag\\
		\|\tilde{w}\|_{L^{\infty}(\Gamma_{\varepsilon}^3\cup \Gamma_{\varepsilon}^4)} &=\bigg\|\frac{h_1- h_2 }{h_1^{1/2}h_2^{1/2}+h_2}\tilde{v}\bigg\|_{L^{\infty}(\Gamma_{\varepsilon}^3\cup \Gamma_{\varepsilon}^4)}\leq C(M_1,C_1) \varepsilon^{\alpha_3}.
	\end{align}
 Furthermore, based on the definition of $\Gamma_{\varepsilon}^1$ in \eqref{eq:Gamma12} and by applying the first mean value theorem, we deduce that
\begin{align}\label{eq:Gamma 1}
	\|u_0\|_{L^1(\Gamma_{\varepsilon}^1)} = \int_{\gamma(a) - \varepsilon}^{\gamma(a)} e^{s d_1 \varepsilon^m} e^{s d_2 x_2}\rmd x_2 = \varepsilon e^{s d_1 \varepsilon^m} \int_{\frac{\gamma(a)}{\varepsilon} -1}^{\frac{\gamma(a)}{\varepsilon}} e^{s \varepsilon d_2 \widetilde{x}_{2}}\rmd \widetilde{x}_2 =  \varepsilon e^{sd_1 \varepsilon^m} e^{s \varepsilon d_2 \widetilde{x}_{2,\xi_1}} ,
\end{align}
where $\widetilde{x}_{2,\,\xi_1} \in \left(\frac{\gamma(a)}{\varepsilon}-1,\frac{\gamma(a)}{\varepsilon}\right)$ is a constant independent of $\varepsilon$, and $\boldsymbol{d} = (d_1, d_2)^\top$ is defined in \eqref{eq:phi}. Meanwhile, we derive that
\begin{align}\label{eq:Gamma 3}
	\|u_0\|_{L^1(\Gamma_{\varepsilon}^3)} &= \int_{\varepsilon^m}^{\varepsilon^m + \varepsilon^{\ell}} e^{s d_1 x_1} e^{s d_2 \gamma(x_1+a-\varepsilon^m)} \sqrt{1+ \gamma^{\prime}(x_1+a-\varepsilon^m)^2}\rmd x_1 \notag\\ & \leq C e^{s d_2 \gamma(x_{1,\xi_1}+a-\varepsilon^m)} e^{sd_1 \varepsilon^m} \frac{|e^{sd_1 \varepsilon^{\ell}}-1|}{s},
\end{align}
where $\widetilde{x}_{1,\,\xi_1} \in \left(\varepsilon^m,\varepsilon^m+\varepsilon^{\ell}\right)$. Hence, by combining \eqref{eq:w1}, \eqref{eq:lateral boundary}--\eqref{eq:Gamma 3}, we find
\begin{align}\label{eq:b e}
	J_b  & \leq \|\tilde{w}\|_{L^{\infty}(\partial D_{\varepsilon})} \|\nabla u_0\|_{L^{1}(\partial D_{\varepsilon})} + \|\nabla \tilde{w}\|_{L^{\infty}(\partial D_{\varepsilon})} \| u_0\|_{L^{1}(\partial D_{\varepsilon})}  \notag\\[2mm]
	& \leq C(M_1,C_1)\big[(\varepsilon^{\alpha_3} +\varepsilon^{1+\alpha_1}) s + \varepsilon^{\alpha_1} \big] \| u_0\|_{L^{1}(\Gamma_{\varepsilon}^{1} \cup \Gamma_{\varepsilon}^{2})} + C(M_1,C_1)\big(\varepsilon^{\alpha_3} s + \varepsilon^{\alpha_3} \big) \| u_0\|_{L^{1}(\Gamma_{\varepsilon}^{3} \cup \Gamma_{\varepsilon}^{4})} \notag\\[2mm]
	& \leq C(M_1,C_1) \big[(\varepsilon^{1+\alpha_3} +\varepsilon^{2+\alpha_1}) s + \varepsilon^{1+\alpha_1} \big]  e^{ s d_1 \varepsilon^m} \big(e^{s \varepsilon d_2 \widetilde{x}_{2,\xi_1}} + e^{ s d_1 \varepsilon^{\ell} } e^{s \varepsilon d_2 \widetilde{x}_{2,\xi_2}}\big)\notag\\& \quad + C(M_1,C_1) \big(\varepsilon^{\alpha_3} s  + \varepsilon^{\alpha_3} \big) \frac{|e^{sd_1 \varepsilon^{\ell}}-1|}{s}  e^{ s d_1 \varepsilon^m} \big(e^{s  d_2 \gamma(x_{1,\xi_1}+a-\varepsilon^m)} +  e^{s d_2 \left[\gamma(x_{1,\xi_2}+a-\varepsilon^m)-\varepsilon\right]}\big),
\end{align}
where $\widetilde{x}_{2,\xi_2} \in \left(\frac{\gamma(a+\varepsilon^{\ell})}{\varepsilon}-1,\frac{\gamma(a+\varepsilon^{\ell})}{\varepsilon}\right)$ is a constant independent of $\varepsilon$, and $\widetilde{x}_{1,\xi_2} \in \left(\varepsilon^m, \varepsilon^m+\varepsilon^{\ell}\right)$.
Additionally, by virtue of Lemma \ref{lem:integ}, the integral equation \eqref{eq:integral fg} can be rewritten as
\begin{equation}\label{eq:int 1F}
	|J|=\big|-J_F-J_g -J_{\tilde{v}} - J_{\tilde{w}} + J_b\big|.
\end{equation}
Take $s = \varepsilon^{\beta}$ with $\beta \in (-1, -\frac{1}{2})$ satisfying the conditions
\begin{equation}\label{eq:beta1F}
1 + (1 - \zeta \alpha_1)\ell < -2\beta,\ \ \beta + \ell > 0,\ \ \mbox{and} \ \ \beta + m > 0,
\end{equation}
where $\alpha_1$ and $ \zeta$ are defined in \eqref{eq:ap1} and \eqref{eq:fg1}, respectively. The parameters $\ell$ and $m$ related to $D_{\varepsilon}$ are defined in \eqref{eq:mini}. By substituting \eqref{eq:beta1F} into \eqref{eq:A}, recalling \eqref{eq:gamm}, \eqref{eq:phi}, and applying properties of a geometric sequence, we find that the estimates hold for $\varepsilon \ll 1$ and $n \to \infty$
\begin{align}
	|e^{sd_1 \varepsilon^{\ell}}-1|&\leq \sum\limits_{n=1}^{\infty}  (s \varepsilon^\ell)^n \sim s \varepsilon^\ell ,\ \	|A| \leq   \sum\limits_{n=1}^{\infty}    s^n (\varepsilon^m +\varepsilon^\ell + \Oh(\varepsilon))^n \sim s (\varepsilon^m + \varepsilon^\ell + \Oh(\varepsilon)), \label{eq:1A} 
\end{align}
where the symbol $``\sim"$ represents the equivalence relation. Namely, there exist two positive constants $N_1$ and $N_2$ satisfying $N_1\,s \varepsilon^\ell\leq\sum\limits_{n=1}^{\infty}  (s \varepsilon^\ell)^n\leq N_2\,s \varepsilon^\ell$. Then, by substituting \eqref{eq:gamm}, \eqref{eq:I5 F}, \eqref{eq:IF}, \eqref{eq:Jg}--\eqref{eq:Jwt}, \eqref{eq:b e}, \eqref{eq:beta1F}, and \eqref{eq:1A} into \eqref{eq:int 1F}, we find that
	\begin{align}\label{eq:int 3F}
	&| F (\boldsymbol{x}_0,u(\boldsymbol{x}_0),v(\boldsymbol{x}_0))| \varepsilon^{\ell+1} \notag\\
	\leq &  \frac{ e^{s d_2 \gamma(a)} e^{-s \varepsilon d_2 \widetilde{x}_{2,\xi}}  }{M_2 \big| \cos (s\varepsilon d_1 \widetilde{x}_{2,\xi})  \big| } \left| -J_F-J_g -J_{\tilde{v}} - J_{\tilde{w}} + J_b \right| + \varepsilon \left|F (\boldsymbol{x}_0,u(\boldsymbol{x}_0),v(\boldsymbol{x}_0))\right| \int_{\varepsilon^m}^{\varepsilon^m+\varepsilon^\ell}\left|A\right| \rmd x_1 \notag\\
	\leq &  \frac{ e^{s d_2 \gamma(a)} e^{-s \varepsilon d_2 \widetilde{x}_{2,\xi}}  }{M_2 \Big| \cos (s\varepsilon d_1 \widetilde{x}_{2,\xi})  \big| } \bigg\{ 
	 C(M_1,C_1,C_2) \bigg[\frac{\varepsilon^{\zeta \ell}}{s^2}  \big(1+ \Oh(\varepsilon^{2-2\ell})\big)^{\frac{\zeta}{2}} + \frac{\varepsilon^{\zeta \alpha_1 \ell}}{s^2}  \big(1+ \Oh(\varepsilon^{2-2\ell})\big)^{\frac{\zeta \alpha_1}{2}} \bigg]
	e^{- \varepsilon^m s \delta / 4}
	\notag \\
	& + C(M_1,C_2) \frac{\varepsilon^{\alpha_3}}{s^2} e^{- \varepsilon^m s \delta / 4}
	 + C(M_1,C_1) \frac{\varepsilon^{\alpha_3}+ \varepsilon^{\alpha_2}}{s^2} e^{- \varepsilon^m s \delta / 4} + C(M_1,C_1) \frac{\varepsilon^{\alpha_3} + \varepsilon^{1+\alpha_1}}{s^2} e^{- \varepsilon^m s \delta / 4}\notag\\
	&+  C(M_1,C_1) \Big[(\varepsilon^{1+\alpha_3} +\varepsilon^{2+\alpha_1}) s + \varepsilon^{1+\alpha_1} \Big]  e^{ s d_1 \varepsilon^m} \Big(e^{s \varepsilon d_2 \widetilde{x}_{2,\xi_1}} + e^{ s d_1 \varepsilon^{\ell} } e^{s \varepsilon d_2 \widetilde{x}_{2,\xi_2}}\Big)\notag\\
	 &+ C(M_1,C_1) \Big(\varepsilon^{\alpha_3} s  + \varepsilon^{\alpha_3} \Big) \frac{|e^{sd_1 \varepsilon^{\ell}}-1|}{s}  e^{ s d_1 \varepsilon^m} \Big(e^{s  d_2 \gamma(x_{1,\xi_1}+a-\varepsilon^m)} +  e^{s d_2 \left[\gamma(x_{1,\xi_2}+a-\varepsilon^m)-\varepsilon\right]}\Big)\bigg\}\notag\\
	 & + | F (\boldsymbol{x}_0,u(\boldsymbol{x}_0),v(\boldsymbol{x}_0))| \varepsilon^{\ell+1} s \big[\varepsilon^m + \varepsilon^{\ell}+ \Oh(\varepsilon)\big]\notag\\
	\leq & C(\varepsilon_0,M_1,M_2,C_1,C_2)  e^{  \varepsilon^{\beta} d_2 \gamma(a)} e^{-\varepsilon^{\beta+1} d_2 \widetilde{x}_{2,\xi}} \bigg\{\varepsilon^{\zeta \alpha_1 \ell - 2\beta}  \big[1+ \Oh(\varepsilon^{2-2\ell})\big]^{\frac{\zeta \alpha_1}{2}} e^{- \varepsilon^{\beta+m} \delta / 4} \notag\\
	& + \varepsilon^{\alpha_2-2 \beta} 
	 	e^{- \varepsilon^{\beta+m} \delta / 4} + \varepsilon^{1+\alpha_1}  
	 	e^{ \varepsilon^{\beta+m} d_1 } \Big(e^{\varepsilon^{\beta+1} d_2 \widetilde{x}_{2,\xi_1}} + e^{   \varepsilon^{\beta + \ell} d_1} e^{ \varepsilon^{\beta+1} d_2 \widetilde{x}_{2,\xi_2}}\Big)
	 	\bigg\} \notag\\
	 	&+ \varepsilon^{\ell+1} \varepsilon^{\beta} \big[\varepsilon^m + \varepsilon^{\ell}+ \Oh(\varepsilon)\big]. 
\end{align} 
By combining \eqref{eq:mini}, \eqref{eq:beta1F}, and the assumption $\alpha_1 \leq \alpha_2$ in Theorem \ref{thm:1}, we deduce that
\begin{equation}\label{eq:cF}
	(\zeta \alpha_1-1)\ell - 2 \beta -1>0, \ \alpha_2-2 \beta -\ell -1 >0,
\end{equation}
and
\begin{equation}\label{eq:lF}
	\alpha_1 -\ell\in(0,1),\ \beta + \ell \in(0,1), \ \beta +1 \in (0,1).
\end{equation}
Therefore, the expression in \eqref{eq:int 3F} can be transformed into the following form:
\begin{align}\notag
	|F (\boldsymbol{x}_0,u(\boldsymbol{x}_0),v(\boldsymbol{x}_0))| \leq C(\varepsilon_0,M_1,M_2,C_1,C_2) \Big(&  \varepsilon^{(\zeta \alpha_1 -1)\ell-2\beta -1} + \varepsilon^{\alpha_1-\ell}\Big) + \varepsilon^{\beta}\left(\varepsilon^m + \varepsilon^\ell + \Oh(\varepsilon)\right).
\end{align}
Moreover, taking $\beta=-\frac{2}{3},\,m=\ell=\frac{7}{9}$, we define \begin{equation}\label{eq:tau2} \tau=\min\Big\{\frac{7}{9}\zeta \alpha_1 - \frac{4}{9},\,\,\alpha_1-\frac{7}{9},\,\,\frac{1}{9},\,\,\frac{1}{3}\Big\}.\end{equation}
From \eqref{eq:cF} and \eqref{eq:lF}, it follows that $\tau \in (0, 1)$. For example, when $\zeta=\frac{6}{7}$ and $\alpha_1=\frac{5}{6}$, conditions \eqref{eq:mini} and \eqref{eq:beta1F} are satisfied, resulting in $\tau=\frac{1}{18}$.

\medskip  \noindent {\bf Step II:} Prove the result established in the remaining portion of $D_\varepsilon$.

We note that the geometric configuration after translation remains identical to our initial setup. We will analyze the remaining portion of $\mathcal{N_{\varepsilon}}$ by employing translations of the region $\mathcal{N_{\varepsilon}}$. To implement a method similar to that used for processing $\widetilde{D}_{\varepsilon}$, we simply need to shift the region $\mathcal{N}_{\varepsilon}$ by a certain number of units horizontally to the left or right. It is observed that by employing a localized analysis, we can apply this method repeatedly.

The proof is complete.
\end{proof}

\section{Proof of Theorem \ref{thm:1} in 3D}\label{sec:4}

In a manner similar to the 2D case, we provide a detailed proof of Theorem \ref{thm:1} in the 3D context. For any subregion of $\mathcal{N}_\varepsilon$, define
\begin{align*}
\widetilde{D}_{\varepsilon}:=\Omega_\varepsilon \times \eta(x_2)  \subset \mathcal{N}_\varepsilon, \ x_2 \in (a, a+\varepsilon^{\ell}) \subset I, \ \ \mbox{with} \ \ \ell<1, 
\end{align*}
and the lateral boundary of $\widetilde{D}_{\varepsilon}$, i.e.,
$$
 \widetilde{\Gamma}_{\varepsilon} := \partial \Omega_\varepsilon \times \eta(x_2), 
 $$
where $I$ is defined in \eqref{eq:n1}, the cross-section $\Omega_\varepsilon$ is a bounded, simply-connected Lipschitz domain in $\mathbb{R}^2$, $\partial \Omega_\varepsilon$ is piecewise smooth, and $a > 0$. 

This section is organized as follows: Subsection \ref{T3D} proves the thin end case, while Subsection \ref{T3N} addresses the narrow end configuration.

\subsection{Thin ends} \label{T3D} Recall that the diameter of $\Omega_{\varepsilon}$ is $\varepsilon$, as mentioned in Section \ref{sub:main results}. Thus, the boundary of $\widetilde{D}_{\varepsilon}\subset \mathcal{N}_{\varepsilon}$ is given by
 $$\partial \widetilde{D}_{\varepsilon} =  \widetilde{\Gamma}_{\varepsilon} \cup \widetilde{\Omega}_{\varepsilon} \cup \widetilde{\Omega}^{\prime}_{\varepsilon},$$
 where
\begin{align}
	\widetilde{\Omega}_{\varepsilon} &=\big\{(x_1,x_2,x_3) \Big|  x_{1,\min}=-\frac{\varepsilon}{2},\, x_{1,\max}=\frac{\varepsilon}{2},\ x_2=a, \ x_{3,\min}=\gamma(a)-\varepsilon,\, x_{3,\max}=\gamma(a)\big\},\notag\\
   \widetilde{\Omega}^{\prime}_{\varepsilon}   &=\big\{(x_1,x_2,x_3)\Big| x_{1,\min}=-\frac{\varepsilon}{2},\, x_{1,\max}=\frac{\varepsilon}{2},\ x_2=a+\varepsilon^{\ell}, \notag \\ & \hspace{7.1cm}x_{3,\min}=\gamma(a+\varepsilon^{\ell})-\varepsilon,\  x_{3,\max}=\gamma(a+\varepsilon^{\ell})\big\}.\notag
\end{align} 

\begin{figure}[htbp]
	\centering
	\includegraphics[scale=0.8]{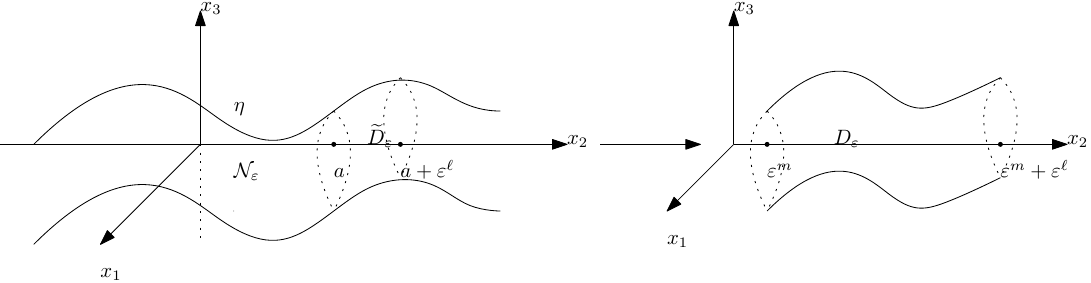}
	\caption{Schematic illustration of the nozzle in 3D}
	\label{fig:2}
\end{figure}

Then we will further shift the entire region $\mathcal{N_{\varepsilon}}$ to the left by $a-\varepsilon^m$ units, as our approach is focused on localized analysis. Thus, the above region $\widetilde{D}_{\varepsilon}$ can be expressed as

 \begin{align}
 	D_{\varepsilon}& :=\Omega_\varepsilon \times \eta(x_2), \ x_2 \in ( \varepsilon^m, \,\varepsilon^m +  \varepsilon^\ell)\subset (-L-a+\varepsilon^m,\,\,L-a+ \varepsilon^m),
 	\label{eq:tube}
 \end{align}
 and 
  \begin{equation}\label{eq:lat}
  \partial D_{\varepsilon} =   \Gamma_{\varepsilon} \cup \Omega_{\varepsilon} \cup \Omega_{\varepsilon}^{\prime}\ \ \mbox{with} \ \ \Gamma_{\varepsilon} = \partial \Omega_\varepsilon \times \eta(x_2),
\end{equation}
  where
\begin{align}
 	\Omega_{\varepsilon} &:=\{(x_1,x_2,x_3)\mid  x_{1,\min}=-\frac{\varepsilon}{2},\, x_{1,\max}=\frac{\varepsilon}{2},\,x_2= \varepsilon^m,\,x_{3,\min}=\gamma(a)-\varepsilon,\,x_{3,\max}=\gamma(a)\},\notag\\
 		\Omega_{\varepsilon}^{\prime} &:=\{(x_1,x_2,x_3)\mid   x_{1,\min}=-\frac{\varepsilon}{2}, \,x_{1,\max}=\frac{\varepsilon}{2},\,x_2=\varepsilon^m + \varepsilon^{\ell}, \notag\\ &\hspace{7.25cm}x_{3,\min}=\gamma(a+\varepsilon^{\ell})-\varepsilon,x_{3,\max}=\gamma(a+\varepsilon^{\ell})\}.\label{eq:Gamma123}
 \end{align}
 Here, $\gamma$ satisfies \eqref{eq:gamm} as we are only considering the horizontal translation of $\widetilde{D}_{\varepsilon}$. And the parameters in \eqref{eq:tube} satisfy
 \begin{align}\label{eq:lm3}
  \frac{2}{3}<m \ \ \mbox{and}\ \ \frac{2}{3}<\ell <\alpha_1.
\end{align} 
We refer to Figure \ref{fig:2} for a schematic illustration of $\widetilde{D}_{\varepsilon}$ and $D_{\varepsilon}$.

Additionally, choose $\boldsymbol{d}$ and $\boldsymbol{d}^{\top}$ in \eqref{eq:cgo} as follows:
\begin{equation}\label{eq:3d}
\boldsymbol{d}:=(d_1,d_2,d_3)^{\top} \in \mathbb{S}^2 \ \mbox{and} \ \boldsymbol{d}^{\top}:=(-d_2,d_1, 0)^{\top}
\end{equation}
such that
\begin{equation} \label{eq:condition3}
\boldsymbol{d} \cdot \boldsymbol{x} = \boldsymbol{d} \cdot \hat{\boldsymbol{x}} |\boldsymbol{x}| \leq -\delta |\boldsymbol{x}| < 0, \ \ \forall \ \boldsymbol{x} \in D_{\varepsilon},
\end{equation}
where $0 < \delta \leq 1$ is a constant depending on $D_{\varepsilon}$ and $\boldsymbol{d}$. Throughout this section, we consistently assume that the unit vector $\boldsymbol{d}$ in the form of the CGO solution, as given by \eqref{eq:cgo}, satisfies \eqref{eq:condition3}.

\begin{proof}[\textbf{ Proof of the Theorem \ref{thm:1} in 3D with thin ends}] We prove this theorem in two steps.

\medskip  \noindent {\bf Step I:} Prove the result established in $D_\varepsilon$. Define
\begin{equation}\label{eq:I56 3F}
\widetilde{J} := F(\boldsymbol{x}_0, u(\boldsymbol{x}_0), v(\boldsymbol{x}_0)) \int_{D_{\varepsilon}} u_0(\boldsymbol{x})\, h_1^{-1/2}(\boldsymbol{x}) \rmd \boldsymbol{x}.
\end{equation}
Then, for any point $\boldsymbol{x} = (x_1, x_2, x_3) \in D_{\varepsilon}$, there exists $\boldsymbol{x}^{\prime} = (x_1^{\prime}, \varepsilon^m, x_3^{\prime}) \in \Omega_{\varepsilon}$ such that the following relationship holds between them
	\[\left\{
	\begin{array}{ll}
		x_1 = x_1^{\prime},\\
		x_3 = x_3^\prime + \gamma(x_2+a-\varepsilon^m)- \gamma(a).
	\end{array}\right.
	\]
	 Then, applying \eqref{eq:h}, \eqref{eq:cgo}, and the first mean value theorem, one has
\begin{align}\label{eq:I5FF 3}
		& \left| F (\boldsymbol{x}_0,u(\boldsymbol{x}_0),v(\boldsymbol{x}_0)) \int_{D_{\varepsilon}} e^{\rho \cdot \boldsymbol{x}}   h_1^{-1/2}(\boldsymbol{x})\rmd \boldsymbol{x}\right|\notag\\
		 \ge  & M_2 \left| F (\boldsymbol{x}_0,u(\boldsymbol{x}_0),v(\boldsymbol{x}_0)) \int_{D_{\varepsilon}} e^{s (d_1 - \bsi d_2) x_1 } e^{s (d_2 + \bsi d_1) x_2 } e^{s d_3 x_3}  \rmd \boldsymbol{x} \right|\notag\\
		 =  & M_2 \left|F (\boldsymbol{x}_0,u(\boldsymbol{x}_0),v(\boldsymbol{x}_0)) \int_{\varepsilon^m}^{\varepsilon^m + \varepsilon^\ell}\int_{\Omega_{\varepsilon}} e^{s (d_1 - \bsi d_2) x_1^{\prime} } e^{s  (d_2 + \bsi d_1) x_2 } e^{s d_3 ( x_3^\prime + \gamma(x_2 +a-\varepsilon^m)- \gamma(a))}\rmd \sigma^{\prime} \rmd x_2\right|\notag\\
		 =  & \left|\int_{\varepsilon^m}^{\varepsilon^m + \varepsilon^\ell}
		e^{s (d_2 + \bsi d_1) x_2 }  e^{s d_3  \gamma(x_2+a-\varepsilon^m) } \rmd x_2
		\int_{\Omega^0}  e^{s \varepsilon (d_1 - \bsi d_2) \widetilde{x}_1} e^{s \varepsilon d_3 \widetilde{x}_3} \rmd \widetilde{\sigma}\right| \notag\\
        &\times M_2 \,\varepsilon^2  e^{-s \gamma(a) d_3 } | F (\boldsymbol{x}_0,u(\boldsymbol{x}_0),v(\boldsymbol{x}_0)) |\notag \\
		 \ge &  M_2 \,\varepsilon^2  e^{-s \gamma(a) d_3 }  \left| F (\boldsymbol{x}_0,u(\boldsymbol{x}_0),v(\boldsymbol{x}_0))  \int_{\varepsilon^m}^{\varepsilon^m + \varepsilon^\ell}
		e^{s (d_2 + \bsi d_1) x_2 }  e^{s d_3 \gamma(x_2+a-\varepsilon^m) } \rmd x_2\right| \notag \\
&\times\left| \mathfrak{R} \int_{\Omega^0} e^{s \varepsilon (d_1 - \bsi d_2) \widetilde{x}_1} e^{s \varepsilon d_3 \widetilde{x}_3} \rmd \widetilde{\sigma} \right|
		\notag\\
		 =  & \left| \cos(s\varepsilon d_2 \tilde{x}_{1,\xi}) F (\boldsymbol{x}_0,u(\boldsymbol{x}_0),v(\boldsymbol{x}_0)) \int_{\varepsilon^m}^{\varepsilon^m + \varepsilon^\ell}
		e^{s (d_2 + \bsi d_1) x_2 }  e^{s d_3 \gamma(x_2+a-\varepsilon^m) } \rmd x_2\right|
		\notag\\
&\times  M_2 \,\varepsilon^2 e^{-s \gamma(a) d_3 } e^{s \varepsilon d_1  \tilde{x}_{1,\xi}} e^{s \varepsilon d_3 \tilde{x}_{3,\xi}} \big|\Omega^0\big| \notag\\
		  =  & M_2 \,\varepsilon^2 e^{-s \gamma(a) d_3 } e^{s \varepsilon d_1  \tilde{x}_{1,\xi}} e^{s \varepsilon d_3 \tilde{x}_{3,\xi}} \big|\Omega^0\big|\big| \cos(s\varepsilon d_2 \tilde{x}_{1,\xi}) F (\boldsymbol{x}_0,u(\boldsymbol{x}_0),v(\boldsymbol{x}_0))\big| \Big|  \varepsilon^{\ell}+\int_{\varepsilon^m}^{\varepsilon^m + \varepsilon^\ell} B
		\rmd x_2 \Big|\notag\\
		\ge & M_2 \,\varepsilon^2 e^{-s \gamma(a) d_3 } e^{s \varepsilon d_1  \tilde{x}_{1,\xi}} e^{s \varepsilon d_3 \tilde{x}_{3,\xi}} \big|\Omega^0\big| \big| \cos(s\varepsilon d_2 \tilde{x}_{1,\xi}) F (\boldsymbol{x}_0,u(\boldsymbol{x}_0),v(\boldsymbol{x}_0)) \big|  \bigg(\varepsilon^{\ell}- \Big|\int_{\varepsilon^m}^{\varepsilon^m + \varepsilon^\ell} B
		\rmd x_2 \Big| \bigg),
	\end{align}
    where $\Omega_{\varepsilon}:= \left\{(\varepsilon \tilde{x}_1,\varepsilon^m,\varepsilon \tilde{x}_3) \mid (\tilde{x}_1,\tilde{x}_3) \in \Omega^0 \right\}$ with  $\Omega^0$ is a bounded Lipschitz domain independent of $\varepsilon$. Here $\left(\tilde{x}_{1,\,\xi},\tilde{x}_{3,\,\xi}\right)\in \Omega^0$ and
	\begin{align}\label{eq:B}
		B= \sum\limits_{n=1}^{\infty} \frac{s^n\Big[ (d_2 + \bsi d_1) x_2 + d_3 \gamma(x_2+ a-\varepsilon^m)\Big]^n} {n!}.
	\end{align}
Noting that \eqref{eq:B} is similar to \eqref{eq:A}, we can employ a similar approach moving forward. Additionally, for any point $\boldsymbol{x}_0 = (x_{01}, x_{02}, x_{03}) \in \overline{D}_{\varepsilon}$, applying the aforementioned geometric setup in \eqref{eq:tube} and \eqref{eq:lm3}, we find
\begin{equation}\label{eq:x03}
	|\boldsymbol{x} -\boldsymbol{x}_0|^{\zeta} \leq \varepsilon^{\ell \zeta} \big[1+ \Oh(\varepsilon^{2-2\ell})\big]^{\frac{\zeta}{2}}.
\end{equation}
Furthermore, by combining \eqref{eq:h}, \eqref{eq:condition}, \eqref{eq:F}, \eqref{eq:F1}, and \eqref{eq:x03},  we obtain
\begin{align}\label{eq:IF 3}
	|J_F|& = \left|\int_{D_{\varepsilon}} (F_1+F_2+F_3) u_0(\boldsymbol{x}) h_1^{-1/2} (\boldsymbol{x})\rmd \boldsymbol{x}\right|\notag\\
	& \leq M_1 \int_{D_{\varepsilon}} \left|\tilde{\delta} f_{\boldsymbol{x}} - \tilde{\delta} g_{\boldsymbol{x}} + \tilde{\delta} f_u - \tilde{\delta} g_u + \tilde{\delta} f_v - \tilde{\delta} g_v \right| e^{-s \delta |\boldsymbol{x}|}\rmd \boldsymbol{x}\notag\\ 
	&\leq C(M_1,C_1,C_2)\int_{D_{\varepsilon}} \Big(|\boldsymbol{x}- \boldsymbol{x}_0 |^{\zeta} +|\boldsymbol{x}- \boldsymbol{x}_0 |^{\zeta \alpha_1}\Big) e^{-s \delta |\boldsymbol{x}|}\rmd \boldsymbol{x}\notag\\ 
	&\leq C(M_1,C_1,C_2)\bigg[ \frac{\varepsilon^{\zeta \ell}}{s^3}  \big(1+ \Oh(\varepsilon^{2-2\ell})\big)^{\frac{\zeta}{2}} + \frac{\varepsilon^{\zeta \alpha_1 \ell}}{s^3}  \big(1+ \Oh(\varepsilon^{2-2\ell})\big)^{\frac{\zeta \alpha_1}{2}} \bigg]
	e^{- \varepsilon^m s \delta / 4}.
\end{align}
Subsequently, through the combination of \eqref{eq:h}, \eqref{eq:fg1}, \eqref{eq:condition}, and \eqref{eq:g}, we get
\begin{align*}
	|J_g|&\leq \int_{ D_{\varepsilon}} \bigg|\frac{h_2- h_1 }{h_1^{1/2}h_2+h_1 h_2^{1/2}} (\boldsymbol{x}) u_0(\boldsymbol{x}) g(\boldsymbol{x}, u(\boldsymbol{x}),v(\boldsymbol{x}))\bigg|\rmd \boldsymbol{x}\notag\\
	&\leq C(M_1,C_2) \varepsilon^{\alpha_3} \int_{D_{\varepsilon}}  e^{- s \delta |\boldsymbol{x}|} \rmd \boldsymbol{x} \notag\\
	&\leq C(M_1,C_2) \frac{\varepsilon^{\alpha_3}}{s^3} e^{- \varepsilon^m s \delta / 4}.
\end{align*}
Similarly, through appropriate simplification, we can obtain the following estimate of $J_{\tilde{v}}$, as denoted in \eqref{eq:g}, i.e.,    
\begin{align}\label{eq:Jv1}
	|J_{\tilde{v}}|
	&\leq C(M_1,C_1) (\varepsilon^{\alpha_3}+ \varepsilon^{\alpha_2}) \int_{D_{\varepsilon}}  e^{- s \delta |\boldsymbol{x}|} \rmd \boldsymbol{x} \leq C(M_1,C_1) \frac{\varepsilon^{\alpha_3}+ \varepsilon^{\alpha_2}}{s^3} e^{- \varepsilon^m s \delta / 4}.
\end{align}
After that, by combining \eqref{eq:h}, \eqref{eq:w1}, \eqref{eq:condition}, and \eqref{eq:g}, we deduce that
\begin{align}\label{eq:Jw}
	|J_{\tilde{w}}| & \leq \int_{ D_{\varepsilon}} \left|\Big[\frac{1}{4} h_1^{-2} |\nabla h_1|^2 -\frac{1}{2} h_1^{-1} \Delta h_1 \Big] u_0(\boldsymbol{x}) \tilde{w}(\boldsymbol{x})\right|\rmd \boldsymbol{x}\notag\\
	& \leq C(M_1,C_1) (\varepsilon^{\alpha_3} + \varepsilon^{1+\alpha_1}) \int_{D_{\varepsilon}}  e^{- s \delta |\boldsymbol{x}|} \rmd \boldsymbol{x}\notag\\
	& \leq C(M_1,C_1) \frac{\varepsilon^{\alpha_3} + \varepsilon^{1+\alpha_1}}{s^3} e^{- \varepsilon^m s \delta / 4}.
\end{align}
Moreover, according to the CGO solutions \eqref{eq:cgo}, the definition of $\Omega_{\varepsilon}$ in \eqref{eq:Gamma123}, and applying the first mean value theorem, for any point $( x_1^{\prime},\varepsilon^m, x_3^{\prime}) \in \Omega_{\varepsilon}$, we have
\begin{align}\label{eq:Gamma3}
	\|u_0\|_{L^1(\Omega_{\varepsilon})} = \int_{\Omega_{\varepsilon}} e^{s(d_1 x_1^{\prime} + d_3 x_3^{\prime})} e^{s d_2 \varepsilon^m}\rmd \sigma &= \varepsilon^2 e^{s d_2 \varepsilon^m} \int_{\Omega^0} e^{s \varepsilon (d_1 \tilde{x}_{1}+ d_3 \widetilde{x}_{3} )}\rmd \tilde{\sigma} \notag\\
	&= \varepsilon^2 e^{s d_2 \varepsilon^m} e^{s \varepsilon (d_1 \tilde{x}_{1,\xi_1} + d_3 \tilde{x}_{3,\xi_1} )} |\Omega^0|,
\end{align}
where the point $(\tilde{x}_{1,\,\xi_1}, \tilde{x}_{3,\,\xi_1}) \in \Omega^0 $  is free of $\varepsilon$. Next, considering that $\Gamma_{\varepsilon}$ represents the lateral boundary defined by \eqref{eq:lat}, the cross-section $\Omega_\varepsilon$ is a bounded, simply-connected Lipschitz domain in $\mathbb{R}^2$, and $\partial \Omega_\varepsilon$ is piecewise smooth, we may start by parameterizing $\Gamma_{\varepsilon}$ and $\partial \Omega_\varepsilon$. The parameterization of \(\Gamma_{\varepsilon}\) is given as follows:
\begin{equation*}
r(z,x_2)=\Big(x_1(z),\  x_2,\  x_3(z)+[\gamma(x_2+a-\varepsilon^m)-\gamma(a)]\Big),
\end{equation*}
where \(\big(x_1(z),x_3(z)\big)\) is the parameterized form of \(\partial \Omega_\varepsilon\), \(x_2 \in \big(\varepsilon^m,\,\varepsilon^m + \varepsilon^{\ell}\big)\), and \(\gamma\) is defined in \eqref{eq:gamm}. Here $z\in(0,z_1)$, $z_1$ represents the perimeter of $\partial \Omega_{\varepsilon}$ and denote $z_1:=M \varepsilon, M>0$. Then, we derive that
\begin{align}\label{eq:Gamma3 3}
	\|u_0\|_{L^1(\Gamma_{\varepsilon})} &= \int_{0}^{M\varepsilon} \int_{\varepsilon^m}^{\varepsilon^m+\varepsilon^{\ell}} e^{s d_1 x_1(z)} e^{s d_2 x_2} e^{s d_3 (x_3(z) +  [\gamma(x_2+a-\varepsilon^m)-\gamma(a)])} \notag\\
    & \quad\sqrt{x_3^{\prime}(z)^2+ x_1^{\prime}(z)^2(1+ \gamma^{\prime}(x_2+a-\varepsilon^m)^2)} \,\rmd x_2 \rmd z \notag\\ 
	& \leq C e^{s d_3 [\gamma(x_{2,\xi_1}+a-\varepsilon^m)-\gamma(a)]} e^{s d_2 \varepsilon^m} \frac{|e^{s d_2 \varepsilon^{\ell}}-1|}{s} \int_{0}^{M \varepsilon}  e^{s d_1 x_1(z)}  e^{s d_3 x_3(z)}\rmd z\notag\\
	& =  C(M) \varepsilon  e^{s d_3 [\gamma(x_{2,\xi_1}+a-\varepsilon^m)-\gamma(a)]} e^{s d_2 \varepsilon^m} \frac{|e^{s d_2 \varepsilon^{\ell}}-1|}{s}  e^{s  d_1 x_1(\xi)} e^{s  d_3 x_3(\xi)},
\end{align}
where $x_{2,\,\xi_1} \in \left(\varepsilon^m,\varepsilon^m +\varepsilon^{\ell}\right)$ and $\xi \in (0,M\varepsilon)$. The point $(x_1(\xi),x_3(\xi))$ lies on $ \partial \Omega_{\varepsilon}$. Here, $x_{1}(\xi)$ and $ x_{3}(\xi)$ are dependent on $\varepsilon$, namely, $x_1(\xi),x_3(\xi)=\Oh(\varepsilon)$. This property will be used in \eqref{eq:int 33F}. Thus, by combining \eqref{eq:w1}, \eqref{eq:lateral boundary}, \eqref{eq:Gamma3}, and \eqref{eq:Gamma3 3}, we obtain
\begin{align*}
	J_b  & \leq \|\tilde{w}\|_{L^{\infty}(\partial D_{\varepsilon})} \|\nabla u_0\|_{L^{1}(\partial D_{\varepsilon})} + \|\nabla \tilde{w}\|_{L^{\infty}(\partial D_{\varepsilon})} \| u_0\|_{L^{1}(\partial D_{\varepsilon})}  \notag\\[1mm]
	& \leq C(M_1,C_1)\big[(\varepsilon^{\alpha_3} +\varepsilon^{1+\alpha_1}) s + \varepsilon^{\alpha_1} \big] \| u_0\|_{L^{1}(\Omega_{\varepsilon} \cup \Omega_{\varepsilon}^{\prime})} + C(M_1,C_1)\big(\varepsilon^{\alpha_3} s + \varepsilon^{\alpha_3} \big) \| u_0\|_{L^{1}(\Gamma_{\varepsilon})} \notag\\[1mm]
	& \leq C(M_1,C_1) \bigg\{ \Big[(\varepsilon^{2+\alpha_3} +\varepsilon^{3+\alpha_1}) s + \varepsilon^{2+\alpha_1} \Big]  e^{ s d_2 \varepsilon^m} \Big[e^{s \varepsilon (d_1 \tilde{x}_{1,\xi_1} + d_3 \tilde{x}_{3,\xi_1} )} + e^{ s d_2 \varepsilon^{\ell} } e^{s \varepsilon (d_1 \tilde{x}_{1,\xi_2} + d_3 \tilde{x}_{3,\xi_2})}\Big] \notag\\
	& \, + \big(\varepsilon^{1+\alpha_3}  + \frac{\varepsilon^{1+\alpha_3}}{s} \big) |e^{s d_2 \varepsilon^{\ell}}-1| e^{ s d_2 \varepsilon^m}  e^{s d_1 x_{1}(\xi)} e^{s d_3 x_{3}(\xi)} e^{s d_3 [\gamma(x_{2,\xi_1}+a-\varepsilon^m)-\gamma(a)]}\bigg\},
\end{align*}
where $\Omega_{\varepsilon}^{\prime}:= \left\{(\varepsilon \tilde{x}_1,\varepsilon^{m+\ell},\varepsilon \tilde{x}_3)\mid (\tilde{x}_1,\tilde{x}_3) \in \Omega^{\prime} \right\}$ with $\Omega^{\prime}$ being a bounded Lipschitz domain independent of 
$\varepsilon$. The point $\big(\widetilde{x}_{1,\,\xi_2}, \widetilde{x}_{3,\,\xi_2}\big)\in \Omega^{\prime} $  is free of $\varepsilon$. Then, based on Lemma \ref{lem:integ}, the integral equation \eqref{eq:integral fg} can be rewritten as
\begin{equation}\label{eq:int F3}
	|J|=\big|-J_F-J_g -J_{\tilde{v}} - J_{\tilde{w}} + J_b\big|.
\end{equation}

Next, by incorporating \eqref{eq:gamm}, \eqref{eq:I5FF 3}, \eqref{eq:IF 3}--\eqref{eq:int F3}, and taking $s=\varepsilon^{\beta}$ with $\beta \in (-1,-\frac{2}{3}) $ such that
\begin{equation}\label{eq:beta 3F}
2+(1-\zeta \alpha_1)\ell <-3\beta,\ \  \beta + \ell > 0,\ \ \mbox{and} \ \ \beta + m > 0,
\end{equation}
where $\alpha_1$ and $\zeta$  are defined in \eqref{eq:ap1} and \eqref{eq:fg1}, respectively, and the parameters $\ell$ and $ m$ related to $D_{\varepsilon}$ are specified in \eqref{eq:tube}. It can be readily obtained that
\begin{align}\label{eq:int 33F}
	&\Big| F (\boldsymbol{x}_0,u(\boldsymbol{x}_0),v(\boldsymbol{x}_0))\Big| \varepsilon^{\ell+2} \notag\\
	\leq & \frac{ e^{s \gamma(a) d_3 } e^{-s \varepsilon d_1  \tilde{x}_{1,\xi} } e^{-s \varepsilon d_3 \tilde{x}_{3,\xi}}  }{M_2 \big|\Omega^0 \big| \big| \cos(s\varepsilon d_2 \tilde{x}_{1,\xi}) \big|}  \Big|-J_F-J_g -J_{\tilde{v}} - J_{\tilde{w}} + J_b\Big| + \varepsilon^2 \Big| F (\boldsymbol{x}_0,u(\boldsymbol{x}_0),v(\boldsymbol{x}_0))\Big| \int_{\varepsilon^m}^{\varepsilon^m + \varepsilon^{\ell}}  |B|\rmd x_2 \notag\\[1ex]
	\leq & \Big| F (\boldsymbol{x}_0,u(\boldsymbol{x}_0),v(\boldsymbol{x}_0))\Big| \varepsilon^{\ell+2}  s \left(\varepsilon^m+\varepsilon^{\ell}+ \Oh(\varepsilon)\right) + \frac{ e^{s \gamma(a) d_3 } e^{-s \varepsilon d_1  \tilde{x}_{1,\xi} } e^{-s \varepsilon d_3 \tilde{x}_{3,\xi}}  }{M_2 \big|\Omega^0 \big| \big| \cos(s\varepsilon d_2 \tilde{x}_{1,\xi}) \big|}\Bigg\{\notag\\ & \ \quad C(M_1,C_1,C_2)\bigg[ \frac{\varepsilon^{\zeta \ell}}{s^3}  \big(1+ \Oh(\varepsilon^{2-2\ell})\big)^{\frac{\zeta}{2}} + \frac{\varepsilon^{\zeta \alpha_1 \ell}}{s^3}  \big(1+ \Oh(\varepsilon^{2-2\ell})\big)^{\frac{\zeta \alpha_1}{2}} + \frac{ \varepsilon^{\alpha_2}}{s^3} + \frac{ \varepsilon^{1+\alpha_1}}{s^3}\bigg] e^{- \varepsilon^m s \delta / 4}\notag\\
		& + C(M_1,C_1) \bigg\{ \Big[(\varepsilon^{2+\alpha_3} +\varepsilon^{3+\alpha_1}) s + \varepsilon^{2+\alpha_1} \Big]  e^{ s d_2 \varepsilon^m} \Big(e^{s \varepsilon (d_1 \tilde{x}_{1,\xi_1} + d_3 \tilde{x}_{3,\xi_1} )} + e^{ s d_2 \varepsilon^{\ell} } e^{s \varepsilon (d_1 \tilde{x}_{1,\xi_2} + d_3 \tilde{x}_{3,\xi_2})}\Big)  \notag\\
	&  + \big(s \varepsilon^{1+\alpha_3+\ell}  + \varepsilon^{1+\alpha_3+\ell} \big)  e^{ s d_2 \varepsilon^m}  e^{s d_1 x_{1}(\xi)} e^{s d_3 x_{3}(\xi)} e^{s d_3 [\gamma(x_{2,\xi_1}+a-\varepsilon^m)-\gamma(a)]}\bigg\} \Bigg\}\notag\\
	\leq &\, C(\varepsilon_0,M_1,M_2,C_1,C_2) e^{\varepsilon^{\beta} \gamma(a) d_3 } e^{- \varepsilon^{\beta +1} (d_1  \widetilde{x}_{1,\xi} + d_3 \widetilde{x}_{3,\xi}) } \bigg\{
	\varepsilon^{\zeta \alpha_1 \ell-3 \beta} \big[1+ \Oh(\varepsilon^{2-2\ell})\big]^{\frac{\zeta \alpha_1}{2}} e^{-\varepsilon^{\beta+m}  \delta / 4}\notag\\
	& + \varepsilon^{\alpha_2-3\beta} e^{-\varepsilon^{\beta+m}  \delta / 4} +  \varepsilon^{2+\alpha_1}  e^{ \varepsilon^{\beta+m}  d_2}\Big[ e^{ \varepsilon^{\beta+1} (d_1 \widetilde{x}_{1,\xi} + d_3 \widetilde{x}_{3,\xi})}+  e^{ \varepsilon^{\beta+\ell} d_2} e^{ \varepsilon^{\beta+1} (d_1 \widetilde{x}_{1,\xi_1} + d_3 \widetilde{x}_{3,\xi_1})} \Big]  \bigg\}\notag\\
	& +  \varepsilon^{\ell+2} \varepsilon^{\beta} \Big(\varepsilon^m+\varepsilon^{\ell}+ \mathcal{O}(\varepsilon)\Big).
\end{align} 
By further combining \eqref{eq:lm3}, \eqref{eq:beta 3F}, and the condition $\alpha_1 \leq \alpha_2$ in Theorem \ref{thm:1}, we can deduce that
\begin{equation}\label{eq:b1F}
	 (\zeta \alpha_1-1)\ell - 3 \beta -2>0,\quad \alpha_2 - 3 \beta -\ell -1 >0,
\end{equation}
and
\begin{equation}\label{eq:l1F}
	\alpha_1-\ell\in(0,1),\quad  \beta + \ell \in(0,1), \quad \beta +1 \in (0,1).
\end{equation}
Therefore, the expression in \eqref{eq:int 33F} can be transformed into the following form
\begin{align}\notag
	\Big|F (\boldsymbol{x}_0,u(\boldsymbol{x}_0),v(\boldsymbol{x}_0))\Big| \leq C(\varepsilon_0,M_1,M_2,C_1,C_2)\Big(   \varepsilon^{(\zeta \alpha_1 -1)\ell-3 \beta -2} + \varepsilon^{\alpha_1-\ell}\Big) + \varepsilon^{\beta}\Big(\varepsilon^m + \varepsilon^\ell + \Oh(\varepsilon)\Big).
\end{align}
Moreover, taking $\beta=-\frac{3}{4},\,m=\ell=\frac{7}{9}$, we define \begin{equation}\label{eq:tau3} \tau=\min\Big\{\frac{7}{9}\zeta \alpha_1 - \frac{19}{36},\,\,\alpha_1-\frac{7}{9},\,\,\frac{1}{36},\,\,\frac{1}{4}\Big\}.\end{equation} 
By virtue of \eqref{eq:b1F} and \eqref{eq:l1F}, we have $\tau\in(0,1)$. For instance, setting $\zeta=\frac{5}{7}$ and $\alpha_1=\frac{39}{40}$ satisfies \eqref{eq:lm3} and \eqref{eq:beta 3F}, yielding $\tau=\frac{1}{72}$.

\medskip  \noindent {\bf Step II:} Prove the result established in the remaining portion of  $D_\varepsilon$.

 Analogous to the 2D case, we can establish the result for the remaining portion of  $\mathcal{N}_{\varepsilon}$ by employing translations of the region $\mathcal{N}_{\varepsilon}$. The details are omitted. Finally, we finish the proof.
\end{proof}

\subsection{Narrow ends}\label{T3N}
 Recall that the cross-section $\Omega_{\varepsilon}$  is a square with a side length of  $\varepsilon$, as mentioned in Section \ref{sub:main results}. Thus, the boundary of 
$\widetilde{D}_{\varepsilon}\subset\mathcal{N}_{\varepsilon}$ is given by 
$$\partial \widetilde{D}_{\varepsilon} =  \widetilde{\Gamma}_{\varepsilon}^3\cup \widetilde{\Gamma}_{\varepsilon}^4 \cup \widetilde{\Omega}_{\varepsilon} \cup \widetilde{\Omega}^{\prime}_{\varepsilon} \cup \widetilde{\Gamma}_f \cup \widetilde{\Gamma}_b,$$
where
\begin{align}
	\widetilde{\Omega}_{\varepsilon} &=\Big\{(x_1,x_2,x_3)\mid  x_1 \in \big(b,b+\varepsilon\big), \ x_2=a, \ x_3 \in \big(\gamma(a)-\varepsilon,\gamma(a)\big)\Big\},\notag\\
	\widetilde{\Omega}^{\prime}_{\varepsilon} &=\bigg\{(x_1,x_2,x_3)\mid   x_1 \in \big(b,b+\varepsilon\big),\ x_2=a+\varepsilon^{\ell},\ x_3 \in \Big(\gamma(a+\varepsilon^{\ell})-\varepsilon,\gamma(a+\varepsilon^{\ell})\Big)\bigg\}, 
	\notag\\
	\widetilde{\Gamma}_f&=\bigg\{(x_1,x_2,x_3) \mid x_1= b+\varepsilon, \ x_2\in (a,a+\varepsilon^{\ell},) \ x_3 \in \Big(\gamma(x_2)-\varepsilon,\gamma(x_2)\Big) \bigg\},\notag\\
	\widetilde{\Gamma}_b&=\Big\{(x_1,x_2,x_3) \mid x_1= b, \ x_2\in (a,a+\varepsilon^{\ell}), \ x_3 \in (\gamma(x_2)-\varepsilon,\gamma(x_2)) \Big\},\notag\\
	\widetilde{\Gamma}_{\varepsilon}^{3} &=\Big\{(x_1,x_2,x_3)\mid  x_1 \in (b,b+\varepsilon),\ x_2\in (a,a+\varepsilon^{\ell}), \ x_3= \gamma(x_2)\Big\},\notag\\
	\widetilde{\Gamma}_{\varepsilon}^{4} &=\Big\{(x_1,x_2,x_3)\mid  x_1 \in (b,b+\varepsilon), \ x_2\in (a,a+\varepsilon^{\ell}),\ x_3 = \gamma(x_2)-\varepsilon\Big\},\notag
\end{align}
with $a$ and $b$ as constants.

\vspace{+10pt}
\begin{figure}[htbp ]
	\centering
	\includegraphics[scale=0.8]{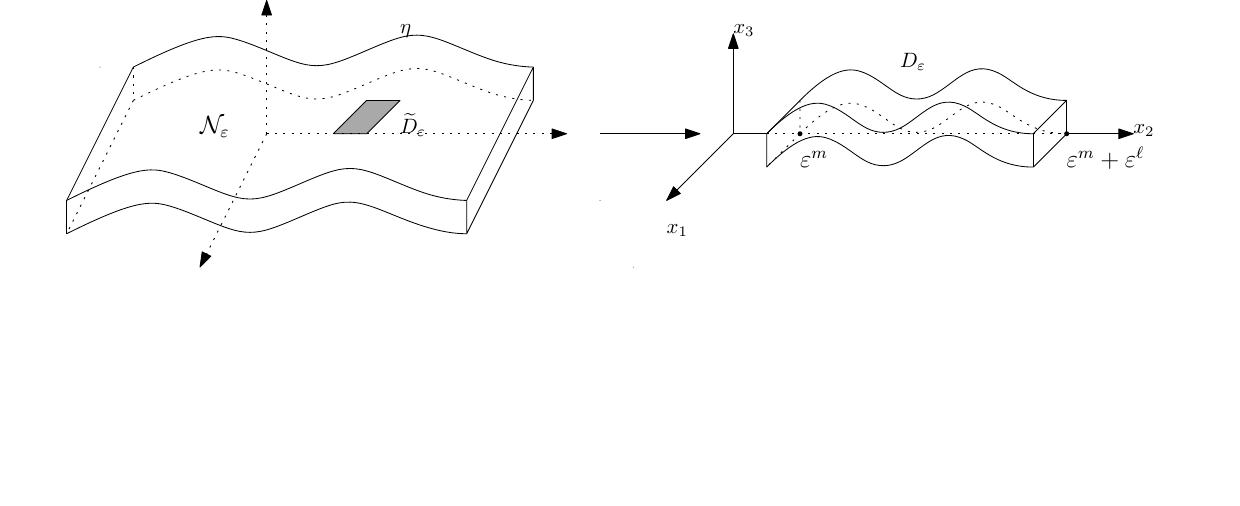}
	\vspace{-3cm}
	\caption{Schematic illustration of the narrow end in 3D}
	\label{fig:3}
\end{figure}

Assume that $a>0$ and $ b<0$. Then we further shift the entire region $\mathcal{N_{\varepsilon}}$ to the left by $a-\varepsilon^m$ units and forward by $-b$ units (translating by $-b$ units in the positive direction of the $x_1$-axis) as our approach focuses on localized analysis. Thus, the above region $\widetilde{D}_{\varepsilon}$ can be expressed as
\begin{align*}
	D_{\varepsilon}& :=\Omega_\varepsilon \times \eta(x_2), \quad x_2 \in ( \varepsilon^m, \varepsilon^m +  \varepsilon^\ell)\subset (-L-a+\varepsilon^m,L-a+ \varepsilon^m).
\end{align*}
The parameters $m$ and $\ell$ also satisfy the condition \eqref{eq:lm3}. And the boundary of  $D_{\varepsilon}$ is defined as
$$\partial D_{\varepsilon} =   \Gamma_{\varepsilon}^3\cup \Gamma_{\varepsilon}^4 \cup \Omega_{\varepsilon} \cup \Omega_{\varepsilon}^{\prime} \cup \Gamma_f \cup \Gamma_b,$$
where
\begin{align}
	\Omega_{\varepsilon} &:=\bigg\{(x_1,x_2,x_3)\Big|  x_1 \in (0,\varepsilon),\ x_2= \varepsilon^m, \ x_3 \in \Big(\gamma(a)-\varepsilon,\,\gamma(a)\Big)\bigg\},\notag\\
	\Omega_{\varepsilon}^{\prime} &:=\bigg\{(x_1,x_2,x_3)\Big| x_1 \in (0,\varepsilon), \ x_2=\varepsilon^m + \varepsilon^{\ell}, \ x_3 \in \Big(\gamma( a+\varepsilon^{\ell})- \varepsilon, \gamma(a + \varepsilon^{\ell})\Big)\bigg\},\notag\\
		\Gamma_f&=\bigg\{(x_1,x_2,x_3)\Big| x_1=\varepsilon, \ x_2\in (\varepsilon^m,\varepsilon^m+\varepsilon^{\ell}), \ x_3 \in \Big( \gamma(x_2+ a-\varepsilon^m)-\varepsilon, \gamma(x_2+ a-\varepsilon^m)\Big) \bigg\},\notag\\
        \Gamma_b&=\bigg\{(x_1,x_2,x_3)\Big| x_1= 0, \ x_2\in (\varepsilon^m,\varepsilon^m+\varepsilon^{\ell}), \ x_3 \in \Big(\gamma(x_2+ a-\varepsilon^m)-\varepsilon, \gamma(x_2+ a-\varepsilon^m)\Big) \bigg\},\notag\\
        \Gamma_{\varepsilon}^{3} &=\bigg\{(x_1,x_2,x_3)\Big| x_1 \in (0,\varepsilon),\ x_2\in \big(\varepsilon^m,\varepsilon^m +\varepsilon^{\ell}\big), \ x_3= \gamma(x_2+a-\varepsilon^m)\bigg\},\notag\\
        \Gamma_{\varepsilon}^{4} &=\bigg\{(x_1,x_2,x_3)\Big|  x_1 \in (0, \varepsilon), \ x_2\in \big(\varepsilon^m,\varepsilon^m +\varepsilon^{\ell}\big),\ x_3 = \gamma(x_2+a-\varepsilon^m)-\varepsilon\bigg\},\label{eq:GammaF}
\end{align}
with $\gamma$ satisfying \eqref{eq:gamm}; see Figure \ref{fig:3} for a schematic illustration of $\widetilde{D}_{\varepsilon}$ and $D_{\varepsilon}$.

From the construction of the aforementioned narrow ends, the narrow end can be regarded, to some extent, as a region formed by the parallel translation of thin ends along a specific normal-sized edge. Choosing any one ``thin end", there exist several (at least two) artificial boundaries that differ from those of the thin ends discussed in the preceding subsection. Therefore, we will focus exclusively on the artificial boundaries of the narrow end. To other proofs, such as Theorem \ref{thm:1} associated with thin ends in Subsection \ref{T3D}, we may apply them directly.

\begin{proof}[\textbf{ Proof of Theorem \ref{thm:1} in 3D with narrow ends}]
In comparison to the thin ends in 3D in Subsection \ref{T3D}, there is one term that requires special attention. 
	 \begin{equation*}
		 J_b:= \int_{\Gamma_f \cup \Gamma_b} \big(\tilde{w} \partial_{\nu} u_0 - u_0 \partial_{\nu} \tilde{w}\big) \rmd \sigma + \int_{\Omega_{\varepsilon}\cup \Omega_{\varepsilon}^{\prime}\cup \Gamma_{\varepsilon}^3 \cup \Gamma_{\varepsilon}^4} \big(\tilde{w} \partial_{\nu} u_0 - u_0 \partial_{\nu} \tilde{w}\big) \rmd \sigma. 
	\end{equation*}  
Given that $\Gamma_{\varepsilon}^3$ and $\Gamma_f$ are defined in \eqref{eq:GammaF}, we conclude that
\begin{align}\label{eq:Gamma 33}
	\|u_0\|_{L^1(\Gamma_{\varepsilon}^3)} &= \int_{0}^{\varepsilon} \int_{\varepsilon^m}^{\varepsilon^m+\varepsilon^{\ell}} e^{s d_1 x_1} e^{s d_2 x_2} e^{s d_3 \gamma(x_2+a-\varepsilon^m)} \sqrt{1+ \gamma^{\prime}(x_2+a-\varepsilon^m)^2} \rmd x_2 \rmd x_1 \notag\\ & \leq C e^{s d_3 \gamma(x_{2}^1+a-\varepsilon^m)} e^{s d_2 \varepsilon^m} \frac{|(e^{s d_2 \varepsilon^{\ell}}-1) (e^{s d_1 \varepsilon}-1)|}{s^2|d_1d_2|} ,\notag\\
	\|u_0\|_{L^1(\Gamma_{f})} & = e^{s d_1 \varepsilon} \int_{\varepsilon^m}^{\varepsilon^m+\varepsilon^{\ell}} \int_{\gamma(x_2+a-\varepsilon^m)-\varepsilon}^{\gamma(x_2+a-\varepsilon^m)}  e^{s d_2 x_2} e^{s d_3 x_3}  \rmd x_3 \rmd x_2 \notag\\
	&=  e^{s d_1 \varepsilon} \frac{|1-e^{-s\varepsilon d_3}|}{s|d_3|}\int_{\varepsilon^m}^{\varepsilon^m+\varepsilon^{\ell}}  e^{s  d_2 x_2} e^{s d_3 \gamma(x_2+a-\varepsilon^m)}  \rmd x_2 \notag\\
		&= e^{s d_1 \varepsilon}e^{s d_2 \varepsilon^m} e^{s d_3 \gamma(x_{2,\xi_1}+a-\varepsilon^m)} \frac{|(1-e^{-s d_3 \varepsilon})(e^{s d_2 \varepsilon^{\ell}}-1)|}{s^2|d_2d_3|},
\end{align}
where $x_{2}^{1},x_{2,\xi_1} \in \left(\varepsilon^m, \varepsilon^m + \varepsilon^{\ell}\right)$. Consequently, our primary focus will be on the first term. Combining \eqref{eq:3d} and \eqref{eq:GammaF}, we derive
\begin{align}\label{eq:na b1}
&\int_{\Gamma_f \cup \Gamma_b} \big(\tilde{w} \partial_{\nu} u_0 - u_0 \partial_{\nu} \tilde{w}\big) \rmd \sigma\notag\\
 =& \int_{\Gamma_f \cup \Gamma_b} \tilde{w} \partial_{\nu} u_0 \rmd \sigma - \int_{\Gamma_f } \nu_f \cdot \nabla \tilde{w} u_0 \rmd \sigma - \int_{\Gamma_b } \nu_b \cdot \nabla \tilde{w} u_0 \rmd \sigma \notag\\
	 = &\int_{\Gamma_f \cup \Gamma_b} \tilde{w} \partial_{\nu} u_0 \rmd \sigma - e^{s (d_1 - \bsi d_2) \varepsilon } \int_{\Gamma_f} e^{s (d_2 + \bsi d_1) x_2 } e^{s d_3 x_3  } \partial_{x_1} \tilde{w}(\varepsilon,x_2,x_3) \rmd \sigma\notag\\
	&+ \int_{\Gamma_b} e^{s (d_2 + \bsi d_1) x_2 } e^{s d_3 x_3  } \partial_{x_1} w(0,x_2,x_3) \rmd \sigma\notag\\
	= & \int_{\Gamma_f \cup \Gamma_b} \tilde{w} \partial_{\nu} u_0 \rmd \sigma, 
	\end{align}
where $\nu_f = (1, 0, 0)^{\top}$ and $\nu_b = (-1, 0, 0)^{\top}$. Thus, by using \eqref{eq:w1}, \eqref{eq:Gamma3}, \eqref{eq:Gamma 33}, and \eqref{eq:na b1}, we obtain
\begin{align}\label{eq:na b2}
	|J_b|\leq & C(M_1,C_1)\bigg\{  \big[(\varepsilon^{2+\alpha_3} +\varepsilon^{3+\alpha_1}) s + \varepsilon^{2+\alpha_1} \big]  e^{ s d_2 \varepsilon^m} \big(e^{s \varepsilon (d_1 \tilde{x}_{1,\xi_1} + d_3 \tilde{x}_{3,\xi_1} )} + e^{ s d_2 \varepsilon^{\ell} } e^{s \varepsilon (d_1 \tilde{x}_{1,\xi_2} + d_3 \tilde{x}_{3,\xi_2})}\big)  \notag\\
	&  + \big(\varepsilon^{\alpha_3} s  + \varepsilon^{\alpha_3} \big)  e^{s d_2 \varepsilon^m} \frac{|(e^{s d_2 \varepsilon^{\ell}}-1) (e^{s d_1 \varepsilon}-1)|}{s^2}\big( e^{s d_3 \gamma(x_{2}^1+a-\varepsilon^m)}+ e^{s d_3 (\gamma(x_{2}^2+a-\varepsilon^m)-\varepsilon)}\big) \notag\\
	&+ (\varepsilon^{\alpha_3}+\varepsilon^{1+\alpha_1}) e^{s d_2 \varepsilon^m}
    \frac{|(1-e^{-s d_3 \varepsilon})(e^{s d_2 \varepsilon^{\ell}}-1) |}{s}
    \Big[ e^{s d_1 \varepsilon} e^{s d_3  \gamma(x_{2,\xi_1}+a-\varepsilon^m)}+  e^{s d_3  \gamma(x_{2,\xi_2}+a-\varepsilon^m)} \Big]\bigg\},
\end{align}
where $(\tilde{x}_{1,\,\xi_2}, \tilde{x}_{3,\,\xi_2}) \in \Omega^0$   is independent of $\varepsilon$. And $x_2^2, x_{2,\xi_2} \in (\varepsilon^m, \varepsilon^m + \varepsilon^{\ell})$.

Similar to the proof of Theorem \ref{thm:1} in Subsection \ref{T3D}, we can derive the estimate in Theorem \ref{thm:1} by combining the estimates related to the terms $J_F$, $J_g$, $J_{\tilde{v}}$, and $J_{\tilde{w}}$ listed in Section \ref{T3D}, along with $\widetilde{J}$ and $J_b$ presented in \eqref{eq:I56 3F} and \eqref{eq:na b2}, along with some intricate calculations. To avoid redundancy, we have omitted the details.
\end{proof}

\section{Proofs of Theorems \ref{th:ibvp}-\ref{thm:iden} and Corollaries \ref{Cor:1} -\ref{Cor:2} }\label{sec:6}
This section provides detailed proofs of Theorems \ref{th:ibvp}--\ref{thm:iden} and Corollaries \ref{Cor:1}--\ref{Cor:2} in Section \ref{sec:applications}. We will first prove Corollaries \ref{Cor:1} and \ref{Cor:2}. As discussed in Section \ref{sec:applications}, the interior transmission eigenvalue problems \eqref{eq:s1} and \eqref{eq:h2} form a subclass of the general coupled PDE system \eqref{eq:sym1}, where the right-hand side functions $f$ and $g$ exhibit linearity with respect to $u$ and $v$, respectively. Consequently, we can employ a methodology analogous to that used in Theorem \ref{thm:1} to establish Corollaries \ref{Cor:1} and \ref{Cor:2}, with the necessary modifications. In the following, we will demonstrate Corollary \ref{Cor:2} while omitting the proof of Corollary \ref{Cor:1}. The proof of Corollary \ref{Cor:2} will be conducted by examining two distinct cases.

\medskip \noindent  {\bf 2D Case:} 
Before the proof, we present some important estimates that play a vital role in the subsequent analysis of the 2D scenarios. Additionally, the following lemma offers supplementary perspectives and analysis related to \eqref{eq:wu0} and \eqref{eq:integral}.

\begin{lem}\label{lem:I2}
	Define
	\begin{align*}
		I_1:&= v (\boldsymbol{x}_0) \int_{D_{\varepsilon}} ( \varphi(\boldsymbol{x}) - \varphi (\boldsymbol{x}_0)) e^{\rho \cdot \boldsymbol{x}} \rmd \boldsymbol{x}, \hspace{1.79cm}  I_2:= \varphi (\boldsymbol{x}_0)  \int_{D_{\varepsilon}} (v(\boldsymbol{x})-v (\boldsymbol{x}_0))  e^{\rho \cdot \boldsymbol{x}} \rmd \boldsymbol{x}, \notag \\ 
		I_3:&=  \int_{D_{\varepsilon}} (\varphi(\boldsymbol{x}) - \varphi(\boldsymbol{x}_0) ) (v(\boldsymbol{x})-v (\boldsymbol{x}_0))  e^{\rho \cdot \boldsymbol{x}} \rmd \boldsymbol{x}, \quad I_4:= \int_{D_{\varepsilon}} q w e^{\rho \cdot \boldsymbol{x}} \rmd \boldsymbol{x}, \\
		I_6 :&= \int_{\Gamma_{\varepsilon}^1\cup\Gamma_{\varepsilon}^{2}} \big(w \partial_{\nu} u_0 - u_0 \partial_{\nu} w\big) \rmd \sigma.
	\end{align*}
	Then the following estimates hold
	\begin{align}
		|I_1|& \leq C(V_0)  \frac{\varepsilon^{\alpha \ell}}{s^2}  \Big[1+\Oh(\varepsilon^{2-2\ell})\Big]^{\frac{\alpha}{2}}
		e^{- \varepsilon^m s \delta / 4}, \notag \\
		|I_2|& \leq C(V_0,C_1)  \frac{\varepsilon^{\alpha_1 \ell}}{s^2}  \Big[1+\Oh(\varepsilon^{2-2\ell})\Big]^{\frac{\alpha_1}{2}}
		e^{- \varepsilon^m s \delta / 4},\notag \\
		|I_3|& \leq  C(V_0,C_1)  \frac{\varepsilon^{(\alpha + \alpha_1)\ell }}{s^2}  \Big[1+\Oh(\varepsilon^{2-2\ell})\Big]^{\frac{\alpha+\alpha_1}{2}}
		e^{-\varepsilon^m s \delta / 4}, \notag \\
		|I_4| &\leq C(V_0,C_1) \frac{\varepsilon^{1+\alpha_1}}{s^2 }  \Big[1+ \varepsilon^{\alpha \ell }\big(1+\Oh(\varepsilon^{2-2\ell})\big)^{\frac{\alpha}{2}}\Big]  e^{-\varepsilon^m s \delta / 4},\notag\\
		|I_6|& \leq C(C_1)\Big(\varepsilon^{2+\alpha_1} s + \varepsilon^{1+\alpha_1} \Big) e^{ s d_1 \varepsilon^m} \Big(e^{s \varepsilon d_2 \widetilde{x}_{2,\xi_1}} + e^{ s d_1 \varepsilon^{\ell} } e^{s \varepsilon d_2 \widetilde{x}_{2,\xi_2}}\Big),\label{eq:I6}
	\end{align}
	where $C_1, \alpha_1$, and $V_0, \alpha$ are defined in \eqref{eq:ap1} and \eqref{eq:ap2}, respectively. The parameters $\ell$ and $\delta$ are defined in \eqref{eq:mini} and \eqref{eq:condition}, respectively. Here, $\widetilde{x}_{2,\,\xi_1} \in \left(\frac{\gamma(a)}{\varepsilon}-1, \frac{\gamma(a)}{\varepsilon}\right)$, $\widetilde{x}_{2,\,\xi_2} \in \left(\frac{\gamma(a+\varepsilon^{\ell})}{\varepsilon}-1, \frac{\gamma(a+\varepsilon^{\ell})}{\varepsilon}\right)$ are fixed and independent of $\varepsilon$.
\end{lem}
\begin{proof}
	Since $q\in C^{0,\alpha}(\overline{\mathcal{N}}_{\varepsilon})$, we have $\varphi:=q-1\in C^{0,\alpha}(\overline{D}_{\varepsilon})$. Therefore, it follows that
	\begin{equation}\label{eq:varphi}
		\varphi(\boldsymbol{x})= \varphi(\boldsymbol{x}_0)+(\varphi(\boldsymbol{x})-\varphi(\boldsymbol{x}_0)), \quad|\varphi(\boldsymbol{x})-\varphi(\boldsymbol{x}_0)|\leq \|\varphi(\boldsymbol{x})\|_{C^{0,\alpha}(\overline{D}_{\varepsilon})}|\boldsymbol{x}-\boldsymbol{x}_0|^{\alpha},
	\end{equation}
where the point $\boldsymbol{x}_0 = \left(x_{01}, x_{02}\right) \in D_{\varepsilon}$. It is noted that $\varphi(\boldsymbol{x}_0) \neq 0$ with respect to \eqref{eq:ap3}. By combining \eqref{eq:cgo}, \eqref{eq:condition}, \eqref{eq:x0}, and \eqref{eq:varphi}, we obtain
	\begin{align*}
		|I_1|& \leq | v (\boldsymbol{x}_0)|\|\varphi(\boldsymbol{x})\|_{C^{0,\alpha}(\overline{D}_{\varepsilon})} \varepsilon^{\alpha \ell} \Big[1+ \Oh(\varepsilon^{2-2\ell})\Big]^{\frac{\alpha}{2}} \int_{D_{\varepsilon}}  e^{- s \delta |\boldsymbol{x}|} \rmd \boldsymbol{x} \notag \\
		& \leq C( V_0) \frac{\varepsilon^{\alpha \ell}}{s^2}  \Big[1+ \Oh(\varepsilon^{2-2\ell})\Big]^{\frac{\alpha}{2}} 
		e^{- \varepsilon^m s \delta / 4}.
	\end{align*}		
Given that $v \in C^{0,\alpha_1}(\overline{D}_{\varepsilon})$ as defined in \eqref{eq:ap1}, it can be expressed as
	\begin{equation}\label{eq:vx}
		v(\boldsymbol{x}) = 	v(\boldsymbol{x}_0) + (	v(\boldsymbol{x})-	v(\boldsymbol{x}_0)),\quad  |v(\boldsymbol{x})-v(\boldsymbol{x}_0)|\leq C_1|\boldsymbol{x}-\boldsymbol{x}_0|^{\alpha_1}.
	\end{equation}
Thus, the estimates for $|I_2|$ and $|I_3|$ can be proven similarly to $I_1$ by combining them with \eqref{eq:vx}. Furthermore, by utilizing \eqref{eq:w} and a similar approach, we obtain
	\begin{align*}
		|I_4|  &\leq \Big|\varphi(\boldsymbol{x}_0)\int_{D_{\varepsilon}} w e^{\rho\cdot \boldsymbol{x}}\rmd \boldsymbol{x}\Big|+ \left| \int_{D_{\varepsilon}} (\varphi(\boldsymbol{x})-\varphi(\boldsymbol{x}_0))w e^{\rho\cdot \boldsymbol{x}}\rmd \boldsymbol{x}\right|+  \Big|\int_{D_{\varepsilon}} w e^{\rho\cdot \boldsymbol{x}}\rmd \boldsymbol{x}\Big|\notag\\
		& \leq C(V_0,C_1) \varepsilon^{1+\alpha_1} \left(\Big|\int_{D_{\varepsilon}}  e^{\rho\cdot \boldsymbol{x}}\rmd \boldsymbol{x}\Big|+ \Big|\int_{D_{\varepsilon}} |\boldsymbol{x}- \boldsymbol{x}_0|^{\alpha}  e^{\rho\cdot \boldsymbol{x}}\rmd \boldsymbol{x}\Big| \right) \notag\\
		& \leq C(V_0,C_1) \frac{\varepsilon^{1+\alpha_1}}{s^2 }  \Big[1+ \varepsilon^{\alpha \ell }\big(1+\Oh(\varepsilon^{2-2\ell})\big)^{\frac{\alpha}{2}}\Big]  e^{-\varepsilon^m s \delta / 4}.
	\end{align*}
Next, by combining \eqref{eq:w} and \eqref{eq:Gamma 1}, we find
	\begin{align*}\notag
		I_6  & \leq \|w\|_{L^{\infty}(\Gamma_{\varepsilon}^{1}  \cup \Gamma_{\varepsilon}^{2})} \|\nabla u_0\|_{L^{1}(\Gamma_{\varepsilon}^1 \cup \Gamma_{\varepsilon}^2)} + \|\nabla w\|_{L^{\infty}(\Gamma_{\varepsilon}^1  \cup \Gamma_{\varepsilon}^{2})} \| u_0\|_{L^{1}(\Gamma_{\varepsilon}^{1}  \cup \Gamma_{\varepsilon}^{2})}  \notag\\[1mm]
		& \leq C(C_1)\big(\varepsilon^{1+\alpha_1} s + \varepsilon^{\alpha_1} \big) \| u_0\|_{L^{1}(\Gamma_{\varepsilon}^{1} \cup \Gamma_{\varepsilon}^{2})} \notag\\[1mm]
		& \leq C(C_1)\Big(\varepsilon^{2+\alpha_1} s + \varepsilon^{1+\alpha_1} \Big) e^{ s d_1 \varepsilon^m} \Big(e^{s \varepsilon d_2 \widetilde{x}_{2,\xi_1}} + e^{ s d_1 \varepsilon^{\ell} } e^{s \varepsilon d_2 \widetilde{x}_{2,\xi_2}}\Big), 
	\end{align*}
	where $\widetilde{x}_{2,\,\xi_1} \in \left(\frac{\gamma(a)}{\varepsilon}-1,\frac{\gamma(a)}{\varepsilon}\right)$ and $\widetilde{x}_{2,\,\xi_2} \in \left(\frac{\gamma(a+\varepsilon^{\ell})}{\varepsilon}-1,\frac{\gamma(a+\varepsilon^{\ell})}{\varepsilon}\right)$ are constants independent of $\varepsilon$.

	The proof is complete.
\end{proof}
\begin{proof}[ \textbf{Proof of Corollary \ref{Cor:2} in 2D}]

Define
\begin{equation}\label{eq:I56}
I_5 := \varphi(\boldsymbol{x}_0) v(\boldsymbol{x}_0) \int_{D_{\varepsilon}} e^{\rho \cdot \boldsymbol{x}} \rmd \boldsymbol{x}.
\end{equation}
Then the integral equation \eqref{eq:integral} can be rewritten as
	\begin{equation}\label{eq:int 1}
		|I_5|=\big|k^{-2} I_6-(I_1 + I_2 +I_3+I_4)\big|.
	\end{equation}
Let $s=\varepsilon^{\beta}$,  where $\beta \in (-1, -\frac{1}{2})$ satisfies that
\begin{equation}\label{eq:beta1}
	\max\{1+(1-\alpha)\ell,1+(1-\alpha_1)\ell\}<-2\beta,\ \  \beta + \ell > 0\ \ \mbox{and} \ \ \beta + m > 0,
\end{equation}
where $\alpha_1$ and $\alpha$ are defined in \eqref{eq:ap1} and \eqref{eq:ap2}, respectively. The parameters $\ell$ and $m$ related to $D_{\varepsilon}$ are defined in \eqref{eq:mini}.
Set $F\left(\boldsymbol{x}_0, u(\boldsymbol{x}_0), v(\boldsymbol{x}_0)\right) = \varphi(\boldsymbol{x}_0) v(\boldsymbol{x}_0)$ and $h_1(\boldsymbol{x})=1$ in \eqref{eq:I5 F}. Subsequently, substituting \eqref{eq:gamm},  \eqref{eq:I5 F}, \eqref{eq:1A}, \eqref{eq:I6}, and \eqref{eq:beta1} into \eqref{eq:int 1}, it can be readily obtained that
	\begin{align}\label{eq:int 3}
		| v(\boldsymbol{x}_0)| \varepsilon^{\ell+1} \leq & \frac{ e^{s d_2 \gamma(a)} e^{-s \varepsilon d_2 \widetilde{x}_{2,\xi}}  }{\Big| \cos (s\varepsilon d_1 \widetilde{x}_{2,\xi})\varphi (\boldsymbol{x}_0)   \big| }\left \{k^{-2}|I_6| + \sum\limits_{n=1}^4 |I_n|\right\} + \varepsilon  \big|v (\boldsymbol{x}_0) \big| \int_{\varepsilon^m}^{\varepsilon^m+\varepsilon^\ell} \big|A\big| \rmd x_1 \notag\\
		 \leq & \frac{ e^{s d_2 \gamma(a)} e^{-s \varepsilon d_2 \widetilde{x}_{2,\xi}}  }{\Big| \cos (s\varepsilon d_1 \widetilde{x}_{2,\xi})\varphi (\boldsymbol{x}_0)   \big| } \bigg\{ k^{-2} \big(\varepsilon^{2+\alpha_1} s + \varepsilon^{1+\alpha_1} \big) e^{ s d_1 \varepsilon^m} \big(e^{s \varepsilon d_2 \widetilde{x}_{2,\xi_1}} + e^{ s d_1 \varepsilon^{\ell} } e^{s \varepsilon d_2 \widetilde{x}_{2,\xi_2}}\big)\notag \\
		& + C(V_0)  \frac{\varepsilon^{\alpha \ell}}{s^2}  \big[1+\Oh(\varepsilon^{2-2\ell})\big]^{\frac{\alpha}{2}}
		e^{- \varepsilon^m s \delta / 4} +  C(V_0,C_1)  \frac{\varepsilon^{\alpha_1 \ell}}{s^2}  \big[1+\Oh(\varepsilon^{2-2\ell})\big]^{\frac{\alpha_1}{2}}
		e^{- \varepsilon^m s \delta / 4}\notag \\
		& +  C(V_0,C_1)  \frac{\varepsilon^{(\alpha + \alpha_1)\ell }}{s^2}  \big[1+\Oh(\varepsilon^{2-2\ell})\big]^{\frac{\alpha+\alpha_1}{2}}
		e^{-\varepsilon^m s \delta / 4} \notag \\
		& + C(V_0,C_1) \frac{\varepsilon^{1+\alpha_1}}{s^2 }  \big[1+ \varepsilon^{\alpha \ell }\big(1+\Oh(\varepsilon^{2-2\ell})\big)^{\frac{\alpha}{2}}\big]  e^{-\varepsilon^m s \delta / 4}\bigg\} + C_1 \varepsilon^{\ell+1} s (\varepsilon^m+\varepsilon^{\ell}+ \Oh(\varepsilon))\notag\\
		 \leq & C(\varepsilon_0,\epsilon_0,k,V_0,C_1)  e^{\varepsilon^{\beta} d_2 \gamma(a)} e^{-s \varepsilon^{\beta+1} d_2 \widetilde{x}_{2,\xi}} \bigg\{
		   \varepsilon^{\alpha \ell - 2 \beta } \big[1+\Oh(\varepsilon^{2-2\ell})\big]^{\frac{\alpha}{2}}
		 e^{- \varepsilon^{\beta+m} \delta / 4} \notag\\
		 & + k^{-2} \big(\varepsilon^{\beta +2+\alpha_1}  + \varepsilon^{1+\alpha_1} \big) e^{\varepsilon^{\beta+m} d_1} \Big(e^{ \varepsilon^{\beta+1} d_2 \widetilde{x}_{2,\xi_1}} + e^{ \varepsilon^{\beta+\ell} d_1} e^{ \varepsilon^{\beta+1} d_2 \widetilde{x}_{2,\xi_2}}\Big)\notag\\
		  & +  \varepsilon^{\alpha_1 \ell - 2 \beta } \big[1+\Oh(\varepsilon^{2-2\ell})\big]^{\frac{\alpha_1}{2}}
		e^{- \varepsilon^{\beta+m} \delta / 4} +    \varepsilon^{(\alpha + \alpha_1)\ell-2 \beta }  \big[1+\Oh(\varepsilon^{2-2\ell})\big]^{\frac{\alpha+\alpha_1}{2}}
		e^{-\varepsilon^{\beta+m}\delta / 4} \notag\\
		& +  \varepsilon^{1+\alpha_1-2 \beta}  \big[1+ \varepsilon^{\alpha \ell }\big(1+\Oh(\varepsilon^{2-2\ell})\big)^{\frac{\alpha}{2}}\big]  e^{-\varepsilon^{\beta+m} \delta / 4}\bigg\} + \varepsilon^{\ell+1} \varepsilon^{\beta} \Big(\varepsilon^m +\varepsilon^{\ell}+ \Oh(\varepsilon)\Big). 
	\end{align} 

Then, by combining \eqref{eq:mini}, and \eqref{eq:beta1}, we can deduce that
\begin{equation}\label{eq:b}
	 (\alpha-1)\ell - 2 \beta -1>0, \ \  (\alpha_1-1)\ell - 2 \beta -1>0,\ \ \beta -\ell +1 +\alpha_1 >0,\ \ \alpha_1-2\beta -\ell>0,
\end{equation}
and
\begin{equation}\label{eq:l}
 \alpha_1 -\ell\in(0,1),\quad \beta + \ell \in(0,1), \quad \beta +1 \in (0,1).
\end{equation}
Therefore, the above expression \eqref{eq:int 3} can be transformed into the following form:
\begin{align*}
	|v(\boldsymbol{x}_0)| \leq C(\varepsilon_0,\epsilon_0,k,V_0,C_1)\Big(& \varepsilon^{(\alpha -1)\ell-2\beta -1} + \varepsilon^{\beta -\ell +1+\alpha_1 } + \varepsilon^{\alpha_1-\ell}+ \varepsilon^{(\alpha_1 -1)\ell-2\beta -1} + \varepsilon^{\alpha_1-2\beta-\ell}\Big) \notag \\
	&+ \varepsilon^{\beta}\Big(\varepsilon^\ell + \Oh(\varepsilon)\Big).
\end{align*}
Moreover, taking $\beta=-\frac{2}{3},\,m=\ell=\frac{7}{9}$, we define
\begin{equation}\notag \tau=\min\Big\{\frac{7}{9} \alpha - \frac{4}{9},\,\,\alpha_1-\frac{4}{9},\,\,\alpha_1-\frac{7}{9},\,\,\frac{7}{9} \alpha_1 - \frac{4}{9},\,\,\frac{1}{9},\,\,\frac{1}{3}\Big\}.\end{equation}
From \eqref{eq:b} and \eqref{eq:l}, it follows that $\tau \in (0, 1)$. As a concrete illustration, the choice $\alpha=\alpha_1=\frac{5}{6}$ simultaneously satisfies the conditions specified in \eqref{eq:mini} and \eqref{eq:beta1}, yielding $\tau=\frac{1}{18}$.

In the subsequent step, we consider the remaining portion of $\mathcal{N}_{\varepsilon}$ in a manner analogous to that employed in Theorem \ref{thm:1}. In the analysis of $u$, we can incorporate Proposition \ref{prop:1} and apply the methodologies used in the analysis of $v$ to obtain analogous results. Consequently, we can logically derive Corollary \ref{Cor:1} in the 2D case.
\end{proof}

\medskip \noindent   {\bf 3D Case:} Before commencing the proof, we present several significant estimates that are essential to prove Corollary \ref{Cor:2} in 3D.

\begin{lem} 
	Define
	\begin{align}
		\widetilde{I_1}:&= v (\boldsymbol{x}_0) \int_{D_{\varepsilon}} ( \varphi(\boldsymbol{x}) - \varphi (\boldsymbol{x}_0)) \,e^{\rho \cdot \boldsymbol{x}} \rmd \boldsymbol{x}, \hspace{1.83cm}  \widetilde{I_2}:= \varphi (\boldsymbol{x}_0)  \int_{D_{\varepsilon}} (v(\boldsymbol{x})-v (\boldsymbol{x}_0))  \,e^{\rho \cdot \boldsymbol{x}} \rmd \boldsymbol{x}, \notag \\ 
		\widetilde{I_3}:&=  \int_{D_{\varepsilon}} (\varphi(\boldsymbol{x}) - \varphi(\boldsymbol{x}_0) ) (v(\boldsymbol{x})-v (\boldsymbol{x}_0))\,  e^{\rho \cdot \boldsymbol{x}} \rmd \boldsymbol{x}, \quad \widetilde{I_4}:= \int_{D_{\varepsilon}} \,q\, w \,e^{\rho \cdot \boldsymbol{x}} \rmd \boldsymbol{x}, \label{eq:part3}\\
		\widetilde{I_6} :&= \int_{\Omega_{\varepsilon}\cup \Omega^{\prime}_{\varepsilon}} \big(w \partial_{\nu} u_0 - u_0 \partial_{\nu} w\big) \rmd \sigma. \label{eq:bF}
	\end{align}
	Then, the following estimates hold
	\begin{align}
		|\widetilde{I_1}|& \leq C( V_0)
		\frac{\varepsilon^{\alpha \ell}}{s^3} \big[1+ \Oh(\varepsilon^{2-2\ell})\big]^{\frac{\alpha}{2}} e^{-\varepsilon^m s \delta / 4}, \notag\\
		|\widetilde{I_2}| &\leq C(V_0,C_1)  
		\frac{\varepsilon^{\alpha_1 \ell}}{s^3} \big[1+ \Oh(\varepsilon^{2-2\ell})\big]^{\frac{\alpha_1}{2}} e^{-\varepsilon^m s \delta / 4},\notag \\
		|\widetilde{I_3}|&  \leq C( V_0)
		\frac{\varepsilon^{(\alpha+\alpha_1) \ell}}{s^3} \big[1+ \Oh(\varepsilon^{2-2\ell})\big]^{\frac{\alpha+\alpha_1}{2}} e^{-\varepsilon^m s \delta / 4} , \notag\\
		|\widetilde{I_4}|&\leq C(V_0,C_1) \frac{\varepsilon^{1+\alpha_1}}{s^3}  \big[1+ \varepsilon^{\alpha \ell }\big(1+\Oh(\varepsilon^{2-2\ell})\big)^{\frac{\alpha}{2}}\big]  e^{-\varepsilon^m s \delta / 4},\notag\\
		|\widetilde{I_6}|& \leq C(C_1)\big(\varepsilon^{3+\alpha_1} s + \varepsilon^{2+\alpha_1} \big)  e^{s d_2 \varepsilon^m}\big[ e^{s \varepsilon (d_1 \widetilde{x}_{1,\xi} + d_3 \widetilde{x}_{3,\xi})}+  e^{s d_2 \varepsilon^{\ell}} e^{s \varepsilon (d_1 \widetilde{x}_{1,\xi_1} + d_3 \widetilde{x}_{3,\xi_1})} \big],\label{eq:bF1}
	\end{align}
	where $C_1, \alpha_1$, and $V_0, \alpha$  are defined in \eqref{eq:ap1} and \eqref{eq:ap2}, respectively. The parameters $\ell$ and $\delta$ are specified in \eqref{eq:tube} and \eqref{eq:condition3}, respectively. The points $\left(\widetilde{x}_{1,\xi}, \widetilde{x}_{3,\xi}\right) \in \Omega^0$ and $(\widetilde{x}_{1,\xi_1}, \widetilde{x}_{3,\xi_1}) \in \Omega^{\prime}$ are both free of $\varepsilon$. Here   $\Omega_{\varepsilon}:= \left\{(\varepsilon \widetilde{x}_1,\varepsilon^m,\varepsilon \widetilde{x}_3) \mid (\widetilde{x}_1,\widetilde{x}_3) \in \Omega^0\right\}$ and $\Omega_{\varepsilon}^{\prime}:= \left\{(\varepsilon \widetilde{x}_1,\varepsilon^{m+\ell},\varepsilon \widetilde{x}_3)\mid (\widetilde{x}_1,\widetilde{x}_3) \in \Omega^{\prime}\right\}$.
\end{lem}

\begin{proof} The proof proceeds similarly to that of Lemma \ref{lem:I2}. For any point $\boldsymbol{x} = (x_1, x_2, x_3) \in D_{\varepsilon}$, by combining \eqref{eq:tube} and \eqref{eq:lm3}, we find that
	\begin{equation}\label{eq:x0 3}
		|\boldsymbol{x} -\boldsymbol{x}_0|^{\alpha} \leq		[\varepsilon^{2\ell}+\Oh(\varepsilon^2)]^{\frac{\alpha}{2}} = \varepsilon^{ \alpha \ell}[1+\Oh(\varepsilon^{2-2\ell})]^{\frac{\alpha}{2}}.
	\end{equation}
By combining \eqref{eq:cgo}, \eqref{eq:condition3}, \eqref{eq:part3}, and \eqref{eq:x0 3}, we obtain
	\begin{align*}
		|\widetilde{I_1}|& \leq | v (\boldsymbol{x}_0)|\|\varphi(\boldsymbol{x})\|_{C^{0,\alpha}(\overline{D}_{\varepsilon})} \varepsilon^{\alpha \ell} \big[1+ \Oh(\varepsilon^{2-2\ell})\big]^{\frac{\alpha}{2}} \int_{D_{\varepsilon}}  e^{- s \delta |\boldsymbol{x}|} \rmd \boldsymbol{x} \notag \\
		&\leq  | v (\boldsymbol{x}_0)|\|\varphi(\boldsymbol{x})\|_{C^{0,\alpha}(\overline{D}_{\varepsilon})} \varepsilon^{\alpha \ell} \big[1+ \Oh(\varepsilon^{2-2\ell})\big]^{\frac{\alpha}{2}} \int_{f_1(\theta,\varphi)}^{f_2(\theta,\varphi)} \int_{\varepsilon^m/2}^{\infty} r^2  e^{- s \delta r} \rmd r \rmd (\theta,\varphi) \notag \\
		&\leq   C(V_0) \varepsilon^{\alpha \ell} \big[1+ \Oh(\varepsilon^{2-2\ell})\big]^{\frac{\alpha}{2}} \int_{ \varepsilon^m s \delta/2}^{\infty} (\frac{t}{s \delta})^{2} e^{-t} \frac{1}{s \delta} \rmd t \notag \\
		&\leq  C(V_0) \frac{\varepsilon^{\alpha \ell}}{s^3} \big[1+ \Oh(\varepsilon^{2-2\ell})\big]^{\frac{\alpha}{2}} \Gamma(3,\varepsilon^m s \delta/2) \notag \\
		& \leq C( V_0)
		\frac{\varepsilon^{\alpha \ell}}{s^3} \big[1+ \Oh(\varepsilon^{2-2\ell})\big]^{\frac{\alpha}{2}} e^{-\varepsilon^m s \delta / 4}.
	\end{align*}
Note that we have utilized \eqref{eq:gamma0} in the last inequality in $|\widetilde{I_1}|$. Thus, by applying \eqref{eq:w} and using a similar approach as above, we derive
	\begin{align*}
		|\widetilde{I_4}|  &\leq \left|\varphi(\boldsymbol{x}_0)\int_{D_{\varepsilon}} w e^{\rho\cdot \boldsymbol{x}}\rmd \boldsymbol{x}\right|+ \left| \int_{D_{\varepsilon}} (\varphi(\boldsymbol{x})-\varphi(\boldsymbol{x}_0))w e^{\rho\cdot \boldsymbol{x}}\rmd \boldsymbol{x}\right|+  \left|\int_{D_{\varepsilon}} w e^{\rho\cdot \boldsymbol{x}}\right|\rmd \boldsymbol{x}\notag\\
		& \leq C(V_0,C_1) \varepsilon^{1+\alpha_1} \bigg(\left|\int_{D_{\varepsilon}}  e^{\rho\cdot \boldsymbol{x}}\rmd \boldsymbol{x}\right|+ \left|\int_{D_{\varepsilon}} |\boldsymbol{x}- \boldsymbol{x}_0|^{\alpha}  e^{\rho\cdot \boldsymbol{x}}\rmd \boldsymbol{x}\right| \bigg) \notag\\
		& \leq C(V_0,C_1) \frac{\varepsilon^{1+\alpha_1}}{s^3}  \Big[1+ \varepsilon^{\alpha \ell }\big(1+\Oh(\varepsilon^{2-2\ell})\big)^{\frac{\alpha}{2}}\Big]  e^{-\varepsilon^m s \delta / 4}.
	\end{align*}
After that, based on \eqref{eq:varphi}, \eqref{eq:vx}, and \eqref{eq:x0 3}, we obtain that
	\begin{align*}
		|\widetilde{I_3}|& \leq \| v (\boldsymbol{x})\|_{C^{0,\alpha_1}(\overline{D}_{\varepsilon})} \|\varphi(\boldsymbol{x})\|_{C^{0,\alpha}(\overline{D}_{\varepsilon})} \varepsilon^{(\alpha+\alpha_1) \ell} \Big[1+ \Oh(\varepsilon^{2-2\ell})\Big]^{\frac{\alpha+\alpha_1}{2}} \int_{D_{\varepsilon}}  e^{- s \delta |\boldsymbol{x}|} \rmd \boldsymbol{x} \notag \\
		& \leq C( V_0)
		\frac{\varepsilon^{(\alpha+\alpha_1) \ell}}{s^3} \Big[1+ \Oh(\varepsilon^{2-2\ell})\Big]^{\frac{\alpha+\alpha_1}{2}} e^{-\varepsilon^m s \delta / 4}.
	\end{align*} 
The estimate for $\widetilde{I_2}$ can be determined through a similar approach. Then, by combining \eqref{eq:w}, \eqref{eq:Gamma3}, and \eqref{eq:bF}, we find
	\begin{align}
		|\widetilde{I_6}|  & \leq \|w\|_{L^{\infty}(\Omega_{\varepsilon} \cup \Omega_{\varepsilon}^{\prime})} \|\nabla u_0\|_{L^{1}(\Omega_{\varepsilon} \cup \Omega_{\varepsilon}^{\prime})} + \|\nabla w\|_{L^{\infty}(\Omega_{\varepsilon} \cup \Omega_{\varepsilon}^{\prime})} \| u_0\|_{L^{1}(\Omega_{\varepsilon} \cup \Omega_{\varepsilon}^{\prime})}  \notag\\[1mm]
		& \leq C(C_1)\big(\varepsilon^{1+\alpha_1} s + \varepsilon^{\alpha_1} \big) \| u_0\|_{L^{1}(\Omega_{\varepsilon} \cup \Omega_{\varepsilon}^{\prime})} \notag\\[1mm]
		& \leq C(C_1)\big(\varepsilon^{3+\alpha_1} s + \varepsilon^{2+\alpha_1} \big)  e^{s d_2 \varepsilon^m}\Big[ e^{s \varepsilon (d_1 \widetilde{x}_{1,\xi} + d_3 \widetilde{x}_{3,\xi})}+  e^{s d_2 \varepsilon^{\ell}} e^{s \varepsilon (d_1 \widetilde{x}_{1,\xi_1} + d_3 \widetilde{x}_{3,\xi_1})} \Big]. \notag
	\end{align}
	Here, $\Omega_{\varepsilon}^{\prime} := \{(\varepsilon \widetilde{x}_1, \varepsilon^{m+\ell}, \varepsilon \widetilde{x}_3) \mid (\widetilde{x}_1, \widetilde{x}_3)\in \Omega^{\prime}\}$, with $\Omega^{\prime}$ being a bounded Lipschitz domain independent of $\varepsilon$. The point $(\widetilde{x}_{1,\xi_1}, \widetilde{x}_{3,\xi_1}) \in \Omega^{\prime}$ is free of $\varepsilon$.

	The proof is now complete.
\end{proof}

\begin{proof}[ \textbf{Proof of Corollary \ref{Cor:2} in 3D}] 
We will prove the conclusion in Corollary \ref{Cor:2} for the scenarios of thin ends and narrow ends. The proof for the case where $\mathcal{N}_\varepsilon$ represents a narrow end is similar to that of Theorem \ref{thm:1} in 3D. Here, we will focus solely on the proof for the case where $\mathcal{N}_\varepsilon$ represents a thin end. Analogous to \eqref{eq:I56}, we define
\begin{equation*}
			\widetilde{I}_5:= \varphi (\boldsymbol{x}_0) v (\boldsymbol{x}_0) \int_{D_{\varepsilon}} e^{\rho \cdot \boldsymbol{x}} \rmd \boldsymbol{x}.
		\end{equation*}
Then, based on Lemma \ref{lem:integ}, the integral equation \eqref{eq:integral} can be rewritten as
	\begin{equation}\label{eq:int 13}
		|\widetilde{I_5}|=\left|k^{-2} \widetilde{I_6}-(\widetilde{I_1} + \widetilde{I_2} + \widetilde{I_3} + \widetilde{I_4})\right|.
	\end{equation}
Let $s=\varepsilon^{\beta}$ with $\beta \in (-1,-\frac{2}{3})$, satisfying that
\begin{equation}\label{eq:beta 3}
\max\{2+(1-\alpha)\ell,2+(1-\alpha_1)\ell\}<-3\beta,\ \ \beta + \ell > 0\ \ \mbox{and} \ \ \beta + m > 0,
\end{equation}
where $\alpha_1$ and $\alpha_1$ are defined in \eqref{eq:ap1} and \eqref{eq:ap2}, respectively. Additionally, the parameters $\ell$ and $m$ related to $D_{\varepsilon}$ are established in \eqref{eq:tube}. Set $F \left(\boldsymbol{x}_0,u(\boldsymbol{x}_0),v(\boldsymbol{x}_0)\right) = \varphi(\boldsymbol{x}_0) v(\boldsymbol{x}_0)$ and $h_1(\boldsymbol{x})=1$ in \eqref{eq:I5FF 3} for clarity. Subsequently, by substituting  \eqref{eq:gamm}, \eqref{eq:I5FF 3}, and \eqref{eq:bF1} into \eqref{eq:int 13}, we can obtain that
	 \begin{align}\label{eq:int 33}
		& | v(\boldsymbol{x}_0)| \varepsilon^{\ell+2} \leq  \frac{ e^{s \gamma(a) d_3 } e^{-s \varepsilon d_1  \widetilde{x}_{1,\xi} } e^{-s \varepsilon d_3 \widetilde{x}_{3,\xi}}  }{\big|\Omega \big| \big| \cos(s\varepsilon d_2 x_{1,\xi}) \varphi(\boldsymbol{x}_0)\big|} \left\{ \sum\limits_{n=1}^4 |\widetilde{I_n}| +  k^{-2}|\widetilde{I_6}| \right\} + \varepsilon^2 | v(\boldsymbol{x}_0)| \int_{\varepsilon^m}^{\varepsilon^m + \varepsilon^{\ell}}  |B|\rmd x_2 \notag\\
		\leq & \frac{ e^{s \gamma(a) d_3 } e^{-s \varepsilon d_1  \widetilde{x}_{1,\xi} } e^{-s \varepsilon d_3 \widetilde{x}_{3,\xi}}  }{\big|\Omega \big| \big| \cos(s\varepsilon d_2 x_{1,\xi}) \varphi(\boldsymbol{x}_0)\big|}\bigg\{ C(V_0) 
		\frac{\varepsilon^{\alpha \ell}}{s^3} \Big[1+ \Oh(\varepsilon^{2-2\ell})\Big]^{\frac{\alpha}{2}} e^{-\varepsilon^m s \delta / 4}\notag\\
		& + C(V_0,C_1)  
		\frac{\varepsilon^{\alpha_1 \ell}}{s^3} \Big[1+ \Oh(\varepsilon^{2-2\ell})\Big]^{\frac{\alpha_1}{2}} e^{-\varepsilon^m s \delta / 4} + C( V_0,C_1)
		\frac{\varepsilon^{(\alpha+\alpha_1) \ell}}{s^3} \Big[1+ \Oh(\varepsilon^{2-2\ell})\Big]^{\frac{\alpha+\alpha_1}{2}} e^{-\varepsilon^m s \delta / 4} \notag\\
	& + k^{-2} \big(\varepsilon^{3+\alpha_1} s + \varepsilon^{2+\alpha_1} \big)  e^{s d_2 \varepsilon^m}\Big[ e^{s \varepsilon (d_1 \widetilde{x}_{1,\xi} + d_3 \widetilde{x}_{3,\xi})}+  e^{s d_2 \varepsilon^{\ell}} e^{s \varepsilon (d_1 \widetilde{x}_{1,\xi_1} + d_3 \widetilde{x}_{3,\xi_1})} \Big]  \notag\\
	& + C(V_0,C_1) \frac{\varepsilon^{1+\alpha_1}}{s^3}  \Big[1+ \varepsilon^{\alpha \ell }\big(1+\Oh(\varepsilon^{2-2\ell})\big)^{\frac{\alpha}{2}}\Big]  e^{-\varepsilon^m s \delta / 4} \bigg\}
		 + C_1 \varepsilon^{\ell+2}  s (\varepsilon^m+\varepsilon^{\ell}+ \Oh(\varepsilon)) \notag\\
		\leq & C(\varepsilon_0,\epsilon_0,k,V_0,C_1) e^{\varepsilon^{\beta} \gamma(a) d_3 } e^{- \varepsilon^{\beta +1} (d_1  \widetilde{x}_{1,\xi} + d_3 \widetilde{x}_{3,\xi}) } \bigg\{
		 \varepsilon^{\alpha \ell-3 \beta} \big[1+ \Oh(\varepsilon^{2-2\ell})\big]^{\frac{\alpha}{2}} e^{-\varepsilon^{\beta+m}  \delta / 4}\notag\\
		 & +   
		 \varepsilon^{\alpha_1 \ell-3 \beta} \big[1+ \Oh(\varepsilon^{2-2\ell})\big]^{\frac{\alpha_1}{2}} e^{-\varepsilon^{\beta+m} s \delta / 4} + \varepsilon^{(\alpha+\alpha_1) \ell-3 \beta} \big[1+ \Oh(\varepsilon^{2-2\ell})\big]^{\frac{\alpha+\alpha_1}{2}} e^{-\varepsilon^{\beta+m} s \delta / 4} \notag\\
		 & + k^{-2} \big(\varepsilon^{\beta+3+\alpha_1} + \varepsilon^{2+\alpha_1} \big)  e^{ \varepsilon^{\beta+m}  d_2}\Big[ e^{ \varepsilon^{\beta+1} (d_1 \widetilde{x}_{1,\xi} + d_3 \widetilde{x}_{3,\xi})}+  e^{ \varepsilon^{\beta+\ell} d_2} e^{ \varepsilon^{\beta+1} (d_1 \widetilde{x}_{1,\xi_1} + d_3 \widetilde{x}_{3,\xi_1})} \Big]  \notag\\
		 & + \varepsilon^{1+\alpha_1-3 \beta} \Big[1+ \varepsilon^{\alpha \ell }\big(1+\Oh(\varepsilon^{2-2\ell})\big)^{\frac{\alpha}{2}}\Big]  e^{-\varepsilon^{\beta+m} \delta / 4} \bigg\}
		+ \varepsilon^{\ell+2} 	 \varepsilon^{\beta}\Big(\varepsilon^m+\varepsilon^\ell + \Oh(\varepsilon)\Big). 
	\end{align} 
Then, by combining \eqref{eq:lm3}, and \eqref{eq:beta 3}, one can deduce that
\begin{equation}\label{eq:b1}
	(\alpha-1)\ell - 3 \beta -2>0, \ \  (\alpha_1-1)\ell - 3 \beta -2>0,\ \  \beta -\ell +1+\alpha_1 >0,\ \ \alpha_1-3\beta -\ell-1>0,
\end{equation}
and
\begin{equation}\label{eq:l1}
	\alpha_1-\ell\in(0,1),\quad \beta + \ell \in(0,1), \quad \beta +1 \in (0,1).
\end{equation}
Furthermore, the expression in \eqref{eq:int 33} can be transformed into the following form
	\begin{align*}
		|v(\boldsymbol{x}_0)| \leq C(\varepsilon_0,\epsilon_0,k,V_0,C_1)\Big(&  \varepsilon^{(\alpha-1)\ell - 3 \beta -2}+ \varepsilon^{(\alpha_1-1)\ell - 3 \beta -2} + \varepsilon^{\beta -\ell +1+\alpha_1} +  \varepsilon^{\alpha_1-3\beta -\ell-1} +  \varepsilon^{\alpha_1-\ell } \Big) \notag \\
		&+ \varepsilon^{\beta}\left(\varepsilon^\ell + \Oh(\varepsilon)\right).
\end{align*}
Moreover, taking $\beta=-\frac{3}{4},\,m=\ell=\frac{7}{9}$, we define 
\begin{equation}\notag
\tau=\min\Big\{\frac{7}{9} \alpha - \frac{19}{36},\,\,\frac{7}{9} \alpha_1 - \frac{19}{36},\,\,\alpha_1-\frac{19}{36},\,\, \alpha_1-\frac{7}{9},\,\,\frac{1}{36},\,\,\frac{1}{4}\Big\}.\end{equation}  By virtue of \eqref{eq:b1} and \eqref{eq:l1}, we have $\tau\in(0,1)$. For instance, setting $\alpha=\alpha_1=\frac{4}{5}$ satisfies \eqref{eq:lm3} and \eqref{eq:beta 3}, yielding $\tau=\frac{1}{45}$.

Analogous to the 2D case, we can establish the result for the remaining portion of 
$\mathcal{N_{\varepsilon}}$
  by employing translations of the region 
$\mathcal{N_{\varepsilon}}$. The details are omitted; thus, we complete the proof. 
\end{proof}

Next, we will prove Theorems \ref{th:ibvp}-\ref{thm:iden} using Theorem \ref{thm:1} from Section \ref{sub:main results} and the aforementioned corollaries.

\begin{proof}[ \textbf{Proof of Theorem \ref{th:ibvp}}]
Since the identity
$
\left(\psi,\Lambda_{f_1}\psi\right)\big|_{\hat{\Gamma}} = \left(\psi,\Lambda_{f_2}\psi\right)\big|_{\hat{\Gamma}}
$
holds on \(\hat{\Gamma} \subset \partial \Omega\), and given that both \(f_1(\boldsymbol{x},u_1)\) and \(f_2(\boldsymbol{x},u_2)\) satisfy the regularity condition \eqref{eq:fg1}, it follows that \(u_1\) and \(u_2\) are solutions to the coupled PDE system with TBCs. More specifically,
\begin{equation*}
\begin{cases}
-\nabla \cdot (h\, \nabla u_1)=f_1(\boldsymbol{x},u_1)& \mbox{in}\ \ \Omega,\\
-\nabla \cdot (h\, \nabla u_2)=f_2(\boldsymbol{x},u_2)& \mbox{in}\ \ \Omega,\\
 \hspace{0.2cm}u_1=u_2,\,\,\partial_{\nu}u_1=\partial_{\nu}u_2& \mbox{on}\ \ \hat{\Gamma}.
\end{cases}
\end{equation*}
With the help of Theorem \ref{thm:1}, we can readily obtain the following estimate
 $$\big|f_1(\boldsymbol{x},u_1(\boldsymbol{x}))-f_2(\boldsymbol{x},u_2(\boldsymbol{x})) \big|\leq C(\varepsilon_0,M_1,M_2,C_1,C_2)  \varepsilon^{\tau}, \ \mbox{with} \ \varepsilon\ll1,
 $$
 where $\tau \in (0,1)$ and $C(\varepsilon_0,M_1,M_2,C_1,C_2)$ is a generic positive constant depending only on $ \varepsilon_0,M_1$,\\$M_2,C_1,C_2$. Thus, we derive the error associated with the right-hand side list in \eqref{eq:err}.
\end{proof}

\begin{proof}[ \textbf{Proof of Theorem \ref{th:operator}}]
From Theorem~\ref{th:ibvp}, we know that the realization \( f_1(\boldsymbol{x},u_1) \) of \(\mathscr{F}_1 \) can approximate the realization \( f_2(\boldsymbol{x},u_2) \) of \( \mathscr{F}_2 \) with an error of \( \mathcal{O}(\varepsilon^{\tau}) \). By substituting both \( f_1(\boldsymbol{x},u_1(\boldsymbol{x})) \) and \( f_2(\boldsymbol{x},u_2(\boldsymbol{x})) \) into \eqref{eq:ibvp} and combining the boundary conditions
\(
\left(\psi,\Lambda_{f_1}(\psi)\right)\big|_{\hat{\Gamma}} = \left(\psi,\Lambda_{f_2}(\psi)\right)\big|_{\hat{\Gamma}},
\) 
and elliptic regularity theory, it follows that \( u_1(\boldsymbol{x}) \) approximates \( u_2(\boldsymbol{x}) \) with an error of \( \mathcal{O}(\varepsilon^{\tau}) \). Since \( f_2 \in \mathcal{M} \), it follows that \( f_2(\boldsymbol{x},u_2(\boldsymbol{x})) \) is \( C^{0,\zeta} \)-continuous in \( (\boldsymbol{x},u_2(\boldsymbol{x})) \). Hence, we can derive
\begin{align*}
\big|\mathscr{F}_1(u_1)(\boldsymbol{x}) - \mathscr{F}_2(u_1)(\boldsymbol{x})\big| &= \big|f_1(\boldsymbol{x},u_1(\boldsymbol{x})) - f_2(\boldsymbol{x},u_1(\boldsymbol{x}))\big|\\ &\leq \big|f_1(\boldsymbol{x},u_1(\boldsymbol{x})) - f_2(\boldsymbol{x},u_2(\boldsymbol{x}))\big| +\big|f_2(\boldsymbol{x},u_2(\boldsymbol{x})) - f_2(\boldsymbol{x},u_1(\boldsymbol{x}))\big| \\ &\leq  \widetilde{C}_2(\varepsilon_0,M_1,M_2,C_1,C_2)(\varepsilon^{\tau} + \varepsilon^{\tau\zeta}).
\end{align*}
This implies that for a single sample $u_1$, the realization of \(\mathscr{F}_1\) approximates that of \(\mathscr{F}_2\) up to an \( \mathcal{O}(\varepsilon^{\tau\zeta}) \) error. The estimate for \( u_2 \) follows analogously.

The proof is complete.
\end{proof}

\begin{proof}[\textbf{Proof of Theorem \ref{thm:visible}}] 
Since the proofs of the scattering problem \eqref{eq:ibvp} for quantum and acoustic wave scattering are analogous, we will focus exclusively on the quantum scattering associated with the physical configuration $V$ that satisfies \eqref{eq:ap3}. To this end, we consider a plane incident wave 
 $w^i:=e^{\mathrm{i}k\boldsymbol{x}\cdot\boldsymbol{d}}$, where $\boldsymbol{d}$ is a unit vector denoting the direction of incidence, along with an inhomogeneous medium $\Omega$. The field $w^i$ represents a solution to the homogeneous Helmholtz equation, specifically,
\begin{equation*}
\Delta w^i+k^2 w^i=0\quad \mbox{in}\quad \mathbb{R}^n.
\end{equation*}

In fact, for the acoustic wave scattering problem, the incident wave can be chosen as plane waves, Herglotz waves, or point-source incidences. Specifically, for the incidence of a point source $w^i$, it should not be too far from the scatterer $\Omega$. This ensures that when $w^i$ is restricted to the interior of the scatterer, the modulus of  $w^i$ has a positive lower bound.

In this proof, we establish the theorem by contradiction. Let  $w\in H^1_{\text{loc}}(\mathbb{R}^n)$ be a solution to \eqref{eym:Sch}. Suppose that  $\Omega$ has one or more thin or narrow ends $\mathcal{N_{\varepsilon}}$ and does not scatter; thus, $\Omega$ is considered invisible, indicating that $w^s=0$. Therefore, $(w,w^i) \in H^1(\Omega) \times H^1(\Omega)$ is the solution to \eqref{eym:Sch}. By Rellich's lemma, the unique continuation principle, and the fact that $\mathcal{N}_\varepsilon\subset \Omega$, we obtain 
\begin{equation}\notag
	\begin{cases}
		- \Delta w -Vw =\lambda w & \text { in } \mathcal{N_{\varepsilon}}, \\ 
		\hspace{0.96cm} - \Delta w^i  =\lambda w^i & \text { in } \mathcal{N_{\varepsilon}}, \\ 
		w=w^i, \,\,\partial_{\nu} w= \partial_{\nu} w^i& \text { on }  \hat{\Gamma},
	 \end{cases}
\end{equation}
where $\hat{\Gamma}$ denotes the lateral boundary of $\mathcal{N_{\varepsilon}}$, and $\hat{\Gamma}$ is of class $C^2$ based on the geometric settings described previously. Then, using Corollary \ref{Cor:1}, we can conclude that
	\begin{equation*}
		|w^i(\boldsymbol{x})|\leq C( \varepsilon_0, \epsilon_0, k, V_0,C_1) \varepsilon^{\tau_1} \quad \mbox{with}\ \ \varepsilon\ll 1.
	\end{equation*}
This contradicts the fact that the incident wave 
 $w^i$ cannot be made too small when restricted to $\mathcal{N_{\varepsilon}}$.
\end{proof}

   \begin{proof}[\textbf{Proof of Theorem \ref{thm:iden}}]
	Let $\mathbf{G}$ denote an unbounded connected component of $\mathbb{R}^n \backslash(\Omega \cup \widetilde{\Omega})$. By contradiction, if \eqref{eq:far} holds, then $\Omega \Delta \widetilde{\Omega}$ contains one or more thin or narrow ends.  Furthermore, for every point $\boldsymbol {x}'$ on the lateral boundary of these  thin or narrow ends, there exists an unbounded path $\Upsilon \subset \mathbb{R}^n\setminus(\Omega\cup \widetilde{\Omega})$ connecting $\boldsymbol {x}'$ to infinity. Since the partial differential operator $\Delta$ is invariant under rigid motion, without loss of generality, we assume that there exists a thin end $\mathcal{N_{\varepsilon}} \subset \mathbf{G} \backslash \Omega$,  as denoted in \eqref{eq:n1}, with $\partial \mathcal{N_{\varepsilon}} \subset \partial \widetilde{\Omega}$.
	
Let $w(\boldsymbol{x} )$ and $\tilde{w}(\boldsymbol{x} )$ denote the total wave fields, and let $w^s(\boldsymbol{x} )$ and $\tilde{w}^s(\boldsymbol{x} )$ denote the scattering wave fields for \eqref{eym:Sch} associated with $\Omega$ and $\widetilde{\Omega}$, respectively, with respect to the same incident wave $w^i(\boldsymbol{x})$. These correspond to the far-field patterns $w_{\infty}(\hat{\boldsymbol{x}}) $ and $\tilde{w}_{\infty}(\hat{\boldsymbol{x}})$ of \(\Omega\) and \(\widetilde{\Omega}\). By \eqref{eq:far} and Rellich's Lemma, we know that
	\begin{equation}\label{eq:t2}
		w^s(\boldsymbol{x})=\tilde{w}^s(\boldsymbol{x}), \ \ \boldsymbol{x}\in \mathbf{G}.
	\end{equation}
By virtue of \eqref{eq:t2} and the transmission conditions in \eqref{eym:Sch}, we have
	\begin{align*}
w(\boldsymbol{x})&=w^i(\boldsymbol{x})+w^s(\boldsymbol{x})=w^i(\boldsymbol{x})+\tilde{w}^s(\boldsymbol{x})= \tilde{w}(\boldsymbol{x})\hspace{2.2cm}  \mbox{on} \ \hat{\Gamma},\\
\partial_{\nu}w(\boldsymbol{x})&=\partial_{\nu}w^i(\boldsymbol{x})+\partial_{\nu}w^s(\boldsymbol{x})=\partial_{\nu}w^i(\boldsymbol{x})+\partial_{\nu}\tilde{w}^s(\boldsymbol{x})= \partial_{\nu}\tilde{w}(\boldsymbol{x}) \ \ \mbox{on} \ \hat{\Gamma},
	\end{align*}
	where $\hat{\Gamma}$ denotes the lateral boundary of  $\mathcal{N_{\varepsilon}}$ as given by \eqref{eq:lateral}. Furthermore, $\tilde{w}$ and $w$ satisfy that 
	\begin{equation*}
		\Delta \tilde{w}+k^2 q \tilde{w}  =0 \quad\mbox{and}\quad \ \Delta w+k^2 w  =0  \quad \mbox{in}\quad\mathcal{N}_{\varepsilon}.
	\end{equation*}
Then, by using Corollary \ref{Cor:2}, we deduce that
	\begin{equation*}
		|w(\boldsymbol{x})|\leq C( \varepsilon_0, \epsilon_0, k, V_0,C_1) \varepsilon^{\tau_2} \leq \Oh(\varepsilon_{\max}^{\tilde{\tau}}), \quad \forall \ \boldsymbol{x} \in \overline{\mathcal{N}}_{\varepsilon} \ \mbox{with}\ \   \varepsilon_{\max}\ll 1.
	\end{equation*}
This leads to a contradiction with \eqref{eq:max} in Definition \ref{def:ad}.

The proof is complete.
\end{proof}

\appendix

\section{}\label{sec:apA}

In this appendix, we verify that  \textbf{Assumption G}  holds when the cross-section $   \Omega_{\varepsilon} \subset \mathcal{N}_{\varepsilon}   $ is a Lipschitz domain for $   n=2   $ or a convex domain with a $   C^2   $ boundary for $   n=3   $. Thus, Assumption G is a generic condition. For instance, Assumption G is also satisfied when $   \Omega_{\varepsilon} \subset \mathcal{N}_{\varepsilon}   $ is a convex Lipschitz domain for $   n=3   $. As an example, when $   \Omega_{\varepsilon}   $ is a rectangle in $   \mathbb{R}^2   $, the generated domain $   \mathcal{N}_{\varepsilon}   $ readily satisfies  \textbf{Assumption G}.

\begin{theorem}\label{th:cover} 
Let $\mathcal{N}_\varepsilon \subset \mathbb{R}^n$ ($n = 2, 3$) be a domain constructed by the parallel translation of the cross-section $\Omega_{\varepsilon}$ along a simple $C^2$-curve $\eta(t)$, where $\eta: I \rightarrow \mathbb{R}^n$ is a graph parametrized over the interval $I = (-L, L)$ with $L \in \mathbb{R}_+$, as defined in Section~\ref{sub:main results}. 
Suppose that \( \Omega_\varepsilon \) is a bounded, simply-connected domain in \( \mathbb{R}^{n-1} \). Additionally, cross-section \( \Omega_\varepsilon\subset\mathcal{N}_{\varepsilon} \) are assumed to be a Lipschitz domain for \( n=2 \) and a convex domain with a \( C^2 \) boundary for \( n=3 \).  The boundary $\partial\mathcal{N}_\varepsilon$ comprises the lateral boundary $\hat{\Gamma}$ given by \eqref{eq:lateral} and the two cross-sectional boundaries $\Omega_{\varepsilon}|_{t=-L}$ and $\Omega_{\varepsilon}|_{t=L}$.  Then, for every point $\boldsymbol{x} \in \mathcal{N}_{\varepsilon}$, there exists a corresponding point $\boldsymbol{x}_0 \in \hat{\Gamma}$ such that  $\boldsymbol{x}_0 - \boldsymbol{x}$ is parallel to the outward unit normal vector $\boldsymbol{\nu}_{\boldsymbol{x}_0}$ at $\boldsymbol{x}_0$.
\end{theorem}

\begin{proof}
We prove the theorem for the cases \( n = 2 \) and \( n = 3 \) separately.

\medskip
\noindent \textbf{Case I:}  
For \( n = 2 \), the lateral boundary \( \hat{\Gamma} \) comprises two components, denoted by \( \hat{\Gamma} = \Gamma_1 \cup \Gamma_2 \), where \( \Gamma_1 \) and \( \Gamma_2 \) represent the upper and lower boundaries, respectively. Following the coordinate representations in Subsection~\ref{sub:parameterized}, these components are parametrized as  
\[
\Gamma_1 = \bigl(t, \gamma(t)\bigr) \quad \text{and} \quad \Gamma_2 = \bigl(t, \gamma(t) - \varepsilon\bigr).
\]
To establish our result, we introduce the distance function
\begin{equation}\label{eq:dis}
    F(\boldsymbol{y}) = \|\boldsymbol{y} - \boldsymbol{x}\|_2^2, \quad \boldsymbol{y} \in \Gamma_1,
\end{equation}
where \( \|\cdot\|_2 \) denotes the Euclidean norm. Suppose \( \Gamma_1 \) is closed, then it follows that \( \Gamma_1 \) is compact. Combining this with the continuity of \( F \), we can conclude, by the extreme value theorem, that there exists a minimizer \( \boldsymbol{x}_0 \in \Gamma_1 \) satisfying
\begin{equation*}
    F(\boldsymbol{x}_0) = \min_{\boldsymbol{y} \in \Gamma_1} F(\boldsymbol{y}).
\end{equation*}

Next, we shall demonstrate that the vector \( \boldsymbol{x} - \boldsymbol{x}_0 \) is collinear with the outward unit normal \( \boldsymbol{\nu}(\boldsymbol{x}_0) \) to \( \Gamma_1 \) at \( \boldsymbol{x}_0 \). Given the \( C^2 \)-regularity of \( \Gamma_1 \), we utilize its parametric representation:
\begin{equation*}
    \boldsymbol{y}(t) = (t, \gamma(t)), \quad t \in [-L, L],
\end{equation*}
where \( \boldsymbol{y}'(t) = (1, \gamma'(t)) \) represents the tangent vector, and \( \boldsymbol{y}''(t) \) exists continuously. For any fixed point \( \boldsymbol{x} = (x_1, x_2) \in \mathcal{N}_{\varepsilon} \),  the distance function $F$ given by \eqref{eq:dis} is expressed as
\begin{equation*}
    F(\boldsymbol{y}) = F(t) = (t - x_1)^2 + (\gamma(t) - x_2)^2.
\end{equation*}

To find the critical points, we set the derivative of \( F \) with respect to \( t \) to zero:
\begin{equation}\label{eq:derF}
    F'(t) = 2(\boldsymbol{y}(t) - \boldsymbol{x}) \cdot \boldsymbol{y}'(t) = 0.
\end{equation}
Let \( t_0 \) correspond to the minimizer \( \boldsymbol{x}_0 = \boldsymbol{y}(t_0) \). We analyze two scenarios:

\noindent \textbf{Case 1 (Interior point):} If \( t_0 \in (-L, L) \), equation \eqref{eq:derF} yields

\[
(\boldsymbol{x}_0 - \boldsymbol{x}) \cdot \boldsymbol{y}'(t_0) = 0,
\]
implying orthogonality between \( \boldsymbol{x}_0 - \boldsymbol{x} \) and the tangent vector. In \( \mathbb{R}^2 \), this necessitates

\[
\frac{\boldsymbol{x}_0 - \boldsymbol{x}}{\|\boldsymbol{x}_0 - \boldsymbol{x}\|} = \boldsymbol{\nu}(\boldsymbol{x}_0),
\]
where \( \boldsymbol{\nu}(\boldsymbol{x}_0) \) is the outward unit normal.

\noindent \textbf{Case 2 (Endpoint):} If \( t_0 = \pm L \), we consider the one-sided derivatives. For \( t_0 = -L \) (left endpoint), we examine the right limit:

\[
\boldsymbol{y}'(t_0^+) = \lim_{t \to t_0^+} \boldsymbol{y}'(t),
\]
and proceed analogously to Case 1. The analysis for \( t_0 = L \) (right endpoint) follows in a similar manner.

\medskip
\noindent \textbf{Case II:} For \( n = 3 \), following the coordinate representation in Subsection \ref{sub:parameterized}, we have
\begin{equation}\label{eq:p eta}
    \boldsymbol{\eta}(t) = \big(0,\: t,\: x_3(t)\big), \quad t \in (-L, L), \quad L \in \mathbb{R}_+.
\end{equation}
The cross-section \(\Omega_{\varepsilon}\) is aligned with the \(x_1x_3\)-plane, with the unit normal vector \(\boldsymbol{n}_0 = (0, -1, 0)^{\top}\). The local basis vectors are defined as \(\boldsymbol{e}_1 = (1, 0, 0)^{\top}\) and \(\boldsymbol{e}_3 = (0, 0, 1)^{\top}\). For \(t_0 \in (-L, L)\), we define \(\Omega_{\varepsilon}^{t_0} \) as the cross-section parallel to the cross-section \(\Omega_{\varepsilon} \)  at \(t_0\), and its parametric representation is
\begin{equation}\label{eq:C -L}
\Omega_{\varepsilon}^{t_0} = \left\{ \big( z_1,\: t_0,\: x_3(t_0) + z_3 \big) \mid (z_1, z_3) \in \Omega_{\varepsilon} \right\}.
\end{equation}
Given that \(\partial \Omega_{\varepsilon}^{t_0}\) is a \(C^2\) closed curve, it possesses the parametric representation
\[
\partial \Omega_{\varepsilon}^{t_0} = \left\{ \boldsymbol{w}(\theta) = \big( w_1(\theta),\: w_3(\theta) \big) \mid \theta \in [0, 2\pi) \right\}.
\]
 According to \eqref{eq:p eta}, the domain \(\mathcal{N}_{\varepsilon}\) and the lateral boundary \(\hat{\Gamma}\) are defined as
\begin{align}
\mathcal{N}_{\varepsilon} &= \left\{ \boldsymbol{\eta}(t) + z_1 \boldsymbol{e}_1 + z_3 \boldsymbol{e}_3 \mid t \in (-L, L),\: (z_1, z_3) \in \Omega_{\varepsilon} \right\}, \notag\\
\hat{\Gamma} &= \left\{ \boldsymbol{y}(s, \theta) \mid s \in (-L, L),\: \theta \in [0, 2\pi) \right\}, \label{eq:C Gam}
\end{align}
with
\begin{equation*}
\boldsymbol{y}(s, \theta) = \boldsymbol{\eta}(s) + w_1(\theta) \boldsymbol{e}_1 + w_3(\theta) \boldsymbol{e}_3 = \big( w_1(\theta),\: s,\: x_3(s) + w_3(\theta) \big).
\end{equation*}
The tangent vectors are expressed as
\begin{equation*}
\frac{\partial \boldsymbol{y}}{\partial s} = \boldsymbol{\eta}'(s) = \big( 0,\: 1,\: x_3'(s) \big), \quad
\frac{\partial \boldsymbol{y}}{\partial \theta} = \big( w_1'(\theta),\: 0,\: w_3'(\theta) \big),
\end{equation*}
and the unit normal vector \(\boldsymbol{\nu}_{\hat{\Gamma}}\) associated with \(\hat{\Gamma}\) is
\begin{equation}\label{eq:N Gam}
\boldsymbol{\nu}_{\hat{\Gamma}} = \frac{
    \frac{\partial \boldsymbol{y}}{\partial s} \times \frac{\partial \boldsymbol{y}}{\partial \theta}
}{
    \left\| \frac{\partial \boldsymbol{y}}{\partial s} \times \frac{\partial \boldsymbol{y}}{\partial \theta} \right\|
} = \frac{
    \big( w_3'(\theta),\: -x_3'(s) w_1'(\theta),\: w_1'(\theta) \big)
}{
    \sqrt{ \big(w_3'(\theta)\big)^2 + \big(x_3'(s) w_1'(\theta)\big)^2 + \big(w_1'(\theta)\big)^2 }
}.
\end{equation}
For any point \(\boldsymbol{x} = \big( z_1,\: t,\: x_3(t) + z_3 \big)^{\top} \in \mathcal{N}_{\varepsilon}\), define the distance function
\begin{equation}\label{eq:C F}
F(s, \theta) = \| \boldsymbol{y} - \boldsymbol{x} \|_2^2 = \big( w_1(\theta) - z_1 \big)^2 + (s - t)^2 + \big( x_3(s) - x_3(t) + w_3(\theta) - z_3 \big)^2, \quad \boldsymbol{y} \in \hat{\Gamma}.
\end{equation}

We consider two cases: either \(\boldsymbol{x}_0\) is an interior point of \(\hat{\Gamma}\), or \(\boldsymbol{x}_0 \in \hat{\Gamma} \cap \Omega_{\varepsilon}^{\pm L}\). The interior case follows similarly to the analysis in two dimensions and is therefore omitted. We focus on the case where \(\boldsymbol{x}_0 \in \hat{\Gamma} \cup \Omega_{\varepsilon}^{-L}\) (the case of \(\Omega_{\varepsilon}^{+L}\) is analogous).  

Based on \eqref{eq:C -L} and \eqref{eq:C Gam}, we obtain
\begin{equation*}
\hat{\Gamma} \cap \Omega_{\varepsilon}^{-L} = \left\{ \boldsymbol{y}(-L, \theta) = \big( w_1(\theta),\: -L,\: x_3(-L) + w_3(\theta) \big) \mid \theta \in [0, 2\pi) \right\}.
\end{equation*}
The normal vector associated with \(\Omega_{\varepsilon}^{-L}\) is \(\boldsymbol{\nu}_{-L} = (0, -1, 0)^{\top}\). According to \eqref{eq:N Gam}, \(\boldsymbol{\nu}_{-L} \cdot \boldsymbol{\nu}_{\hat{\Gamma}} \neq 0\), confirming that the surfaces intersect transversely. The set of sub-normal vectors at \(\boldsymbol{x}_0\) can be expressed as
\begin{equation*}
N_{\partial \mathcal{N}_{\varepsilon}}(\boldsymbol{x}_0) = \left\{ \lambda \boldsymbol{\nu}_A + \mu \boldsymbol{\nu}_{\hat{\Gamma}} \mid \lambda, \mu \geq 0,\: \| \lambda \boldsymbol{\nu}_A + \mu \boldsymbol{\nu}_{\hat{\Gamma}} \| = 1 \right\}.
\end{equation*}

The distance function given in \eqref{eq:C F} at \(s = -L\) simplifies to
\begin{equation*}
F_{-L}(\theta) = \big( w_1(\theta) - z_1 \big)^2 + (-L - t)^2 + \big( x_3(-L) - x_3(t) + w_3(\theta) - z_3 \big)^2.
\end{equation*}
If \(\boldsymbol{x}_0 = \big( w_1(\theta_0),\: -L,\: x_3(-L) + w_3(\theta_0) \big)^{\top}\) is an minimum point, then
\begin{equation*}
\frac{\partial F_{-L}}{\partial \theta} \bigg|_{\theta = \theta_0} = 2\big( w_1(\theta_0) - z_1 \big) w_1'(\theta_0) + 2\big( x_3(-L) - x_3(t) + w_3(\theta_0) - z_3 \big) w_3'(\theta_0) = 0.
\end{equation*}
The vector \(\boldsymbol{x}_0 - \boldsymbol{x} = \big( w_1(\theta_0) - z_1,\: -L - t,\: x_3(-L) - x_3(t) + w_3(\theta_0) - z_3 \big)\) satisfies \((\boldsymbol{x}_0 - \boldsymbol{x}) \cdot \frac{\partial \boldsymbol{y}}{\partial \theta}(-L, \theta_0) = 0\). If additionally \((\boldsymbol{x}_0 - \boldsymbol{x}) \cdot \boldsymbol{\eta}'(-L) = 0\), then \(\boldsymbol{x}_0 - \boldsymbol{x}\) is parallel to \(\boldsymbol{\nu}_{\hat{\Gamma}}\), and we set \(\boldsymbol{\nu}_{\boldsymbol{x}_0} = \boldsymbol{\nu}_{\hat{\Gamma}}\). Otherwise, we select the next nearest point \(\tilde{\boldsymbol{x}}_0 \in \hat{\Gamma}\).

The proof is complete. 
\end{proof}

\section{}\label{sec:apB}

In Subsection \ref{sub:visi} and Remark \ref{rem:CD}, we demonstrate that assumption \eqref{eq:uv2} holds for the PDE systems \eqref{eq:s1} and \eqref{eq:h2}, which model quantum and acoustic medium scattering in \eqref{eym:Sch}. For these systems, the solution is given by $ u = w|_{\Omega} $ and $ v = w^i|_{\Omega} $, where $ w = w^i + w^s $ represents the total wave field in \eqref{eym:Sch}, with $ w^i $ and $ w^s $ denoting the incident and scattered waves, respectively. Assumption \eqref{eq:uv2} concerns the regularity of $ u $ and $ v $. For illustration, we consider acoustic medium scattering, noting that the quantum scattering case follows analogously. Assuming a plane wave incident field, $ v = w^i = e^{\mathrm{i} k \boldsymbol{x} \cdot \boldsymbol{d}} $, we have $ v \in H^1_{\mathrm{loc}}(\mathcal{N}_\varepsilon) \cap C^{1,\alpha_1}(\overline{\mathcal{N}_\varepsilon}) $, with the estimate $ \|v\|_{C^{1,\alpha_1}(\overline{\mathcal{N}_\varepsilon})} \leq C_1 $, where $ \alpha_1 \in (0,1) $ and $ C_1 $ is independent of the geometric parameter $ \varepsilon $. Thus, the component of \eqref{eq:uv2} related to $ v $ is satisfied.


In Theorem \ref{thm:re}, we use system \eqref{eq:h2} as an illustration to rigorously prove that the component of \eqref{eq:uv2} related to \( u \) is satisfied. Specifically, we select \( \alpha_1 = \frac{49}{50} \) for \( n = 2, 3 \) and demonstrate that \( u \) exhibits \( C^{1,\frac{49}{50}} \)-regularity on the extended domain \( B_{3R/2} \supseteq \Omega \), with a uniform upper bound \( \|u\|_{C^{1,\frac{49}{50}}(B_{3R/2})} \leq C_1 \), independent of \( \varepsilon \). Here, \( B_{3R/2} \subset \mathbb{R}^n \) denotes the ball centered at the origin with radius \( \frac{3R}{2} \), where \( R \in \mathbb{R}_+ \). To establish the \( C^{1,\frac{49}{50}}(B_{3R/2}) \)-regularity of \( u \), we adopt a proof strategy analogous to that in Lemma 5.9 of \cite{BL}.

\begin{theorem}\label{thm:re}
Let \(\Omega \subset B_R \subset \mathbb{R}^n \,(n=2,3)\) be a bounded Lipschitz domain with thin or narrow ends \(\mathcal{N}_{\varepsilon}\). Here, \(B_R\) denotes a disk or ball centered at the origin with radius \(R\) in \(\mathbb{R}^n\) for some \(R > 0\). Let \(q \in L^{\infty}(\mathbb{R}^n)\) with \(q - 1\) compactly supported in \(\Omega\), and let \(w \in H^2_{loc}(\mathbb{R}^n)\) be a solution to the acoustic medium scattering problem \eqref{eym:Sch}. Then \(w \in C^{1, \frac{49}{50}}(\bar{B}_{3R/2})\), and there exists a constant \(C(R, k, n, \|q\|_{L^\infty}, \|w^i\|_{L^2}) > 0\) that is independent of \(\varepsilon\) such that
\[
\|w\|_{C^{1, \frac{49}{50}}(\bar{B}_{3R/2})} \leq C(R, k, n, \|q\|_{L^\infty}, \|w^i\|_{L^2}).
\]
\end{theorem}

\begin{proof}
Since \( w \) is the unique solution to \eqref{eym:Sch}, and due to the well-posedness of the acoustic scattering problem \eqref{eym:Sch}, we conclude that \( w \in H^2_{loc}(\mathbb{R}^n) \). Specifically, we observe that \( w \in H^2(B_{2R}) \) satisfies 
\[
\left(\Delta + k^2 q\right) w = 0 \quad \text{in} \quad B_{2R},
\]
where \( B_{2R} \) denotes the central ball at the origin with a radius of \( 2R \in \mathbb{R}_+ \). Next, the proof is primarily organized into two steps.

\medskip  \noindent {\bf Step I:} Prove that \(\left\lVert w(\boldsymbol{x})\right\rVert_{L^2(B_{2R})} \leq C(R, k, n, \|q\|_{L^\infty}, \|w^i\|_{L^2})\), where \( C(R, k, n, \|q\|_{L^\infty}, \|w^i\|_{L^2}) > 0 \) is independent of \(\varepsilon\).

According to \cite{CK2017}, the acoustic scattering system \eqref{eym:Sch} can be reformulated as a Lippmann-Schwinger integral equation
\begin{equation}\label{eq:LS}
    w(\boldsymbol{x})=w^i(\boldsymbol{x})-k^2\int_{B_{2R}} \Phi(\boldsymbol{x},\boldsymbol{y})(1-q)(\boldsymbol{y})w(\boldsymbol{y}) \rmd \boldsymbol{y},\ \ \boldsymbol{x}\in \mathbb{R}^n,\ n=2,3,
\end{equation}
where \(\Phi(\boldsymbol{x},\boldsymbol{y})\) denotes the fundamental solution of the Helmholtz equation. Specifically,
\[
\Phi(\boldsymbol{x}, \boldsymbol{y}) = \begin{cases}
    \dfrac{\mathrm{i}}{4} H_0^{(1)}(k|\boldsymbol{x}-\boldsymbol{y}|), & \boldsymbol{x} \neq \boldsymbol{y}, \quad n=2, \\[10pt]
    \dfrac{1}{4\pi} \dfrac{e^{\mathrm{i}k|\boldsymbol{x}-\boldsymbol{y}|}}{|\boldsymbol{x}-\boldsymbol{y}|}, & \boldsymbol{x} \neq \boldsymbol{y}, \quad n=3.
\end{cases}
\]
For any fixed \(\boldsymbol{y} \in \mathbb{R}^n\), this solution satisfies \((\Delta + k^2)\Phi(\cdot, \boldsymbol{y}) = 0\) in \(\mathbb{R}^n \setminus \{\boldsymbol{y}\}\).

It is noted that \eqref{eq:LS} can be reformulated as follows: \begin{align}\notag
w(\boldsymbol{x})= w^i(\boldsymbol{x}) - K_{B_{2R}}[w](\boldsymbol{x}), \end{align}
where the operator $K_{B_{2R}}:L^2( B_{2R})\to H^2( B_{2R})$ (see \cite{CK2017}), is defined by 
\begin{align} \notag
K_{B_{2R}}[w](\boldsymbol{x})&:=k^2\int_{B_{2R}} \Phi(\boldsymbol{x},\boldsymbol{y})(1-q)(\boldsymbol{y})w(\boldsymbol{y}) \rmd \boldsymbol{y}.
\end{align}
By Rellich-Kondrachov theorem, the embedding \( H^2(B_{2R}) \hookrightarrow L^2(B_{2R}) \) is compact. This implies that the operator \( K_{B_{2R}}: L^2(B_{2R}) \to L^2(B_{2R}) \) is compact.  
As established in \cite{CK2017}, when
\(\Re q>0\), and\( \Im q \geq 0\), the inverse operator \((I + K_{B_{2R}})^{-1}\) exists and is bounded. By the Schur test, it yields that the $L^2$-operator norm of $K_{B_{2R}}$ is bounded by an upper bound, which only depends on the infinity norm of $q$, $k$, $n$ and $R$.  Therefore, we obtain the estimate:
\begin{equation}\label{eq:w2RL}
\|w\|_{L^2(B_{2R})} \leq \|(I + K_{B_{2R}})^{-1}\| \|w^i\|_{L^2(B_{2R})} \leq C(R, k, n, \|q\|_{L^\infty}) \|w^i\|_{L^2(B_{2R})},
\end{equation}
where the positive number $C(R, k, n, \|q\|_{L^\infty}) $ is independent of $\varepsilon$. 

\medskip  \noindent {\bf Step II:} Prove that $\left\lVert w(\boldsymbol{x})\right\rVert_{C^{1,\frac{49}{50}}(B_{3R/2})} \leq C(R, k, n, \|q\|_{L^\infty}, \|w^i\|_{L^2})$, with 
\( C(R, k, n, \|q\|_{L^\infty}, \|w^i\|_{L^2}) \) \(> 0 \) is independent of $\varepsilon$.

Interior elliptic regularity in \( B_{7R/4} \setminus B_{5R/4} \), where \( q - 1 \equiv 0 \) (see \cite[Theorem 8.10]{Gilbarg}), implies that \( w \in H^s(B_{7R/4} \setminus \overline{B}_{5R/4}) \) for any \( s \geq 0 \), with corresponding norm estimates. Specifically, we take \( s = 3 \) for \( n = 2 \) and \( s = 4 \) for \( n = 3 \).

By the Sobolev embedding theorem \cite{Evans2010}, we obtain
\begin{equation}\label{eq:holder}
    \|w\|_{C^{1, \frac{49}{50}}(B_{7R/4} \setminus B_{5R/4})} \leq C \|w\|_{H^{s}(B_{7R/4} \setminus B_{5R/4})} \leq C \|w\|_{H^2(B_{2R} \setminus B_R)},
\end{equation}
where the constant \( C = C(R, k, n) \) may vary between inequalities. This implies that \( w \) has well-defined boundary values in \( C^{1,\frac{49}{50}}(\partial B_{3R/2}) \). More precisely, there exists an extension \( \varphi \in C^{1,\frac{49}{50}}(\mathbb{R}^n) \), supported in \( B_{7R/4} \setminus B_{5R/4} \), such that \( w|_{\partial B_{3R/2}} = \varphi|_{\partial B_{3R/2}} \).

Consider the Dirichlet problem for $\widetilde{v}$:
\begin{equation}\label{eq:vw}
\Delta \widetilde{v} = -k^2 q w \quad \text{in } B_{3R/2}, \quad\widetilde{v} = \varphi \quad \text{on } \partial B_{3R/2}.
\end{equation}
Since $-k^2 q w \in L^\infty(B_{3R/2})$ and $\varphi \in C^{1,\frac{49}{50}}(\partial B_{3R/2})$, Theorem 8.34 in \cite{Gilbarg} yields the unique solvability of \eqref{eq:vw} in $C^{1,\frac{49}{50}}(\overline{B}_{3R/2})$.
To establish $w = \widetilde{v}$ (and consequently $w \in C^{1,\frac{49}{50}}(\overline{B}_{3R/2})$), we first note that both $v$ and $w$ satisfy the weak formulation of \eqref{eq:vw} in $H^1(B_{3R/2})$. The difference $\widetilde{v} - w$ solves the homogeneous boundary value problem:
\[
\Delta(\widetilde{v} - w) = 0 \quad \text{in } B_{3R/2}, \quad\widetilde{v} - w = 0 \quad \text{on } \partial B_{3R/2},
\]
where the boundary equality holds in the trace sense. Applying the $H^1$-maximum principle for harmonic functions yields $\widetilde{v} \equiv w$ in $H^1(B_{3R/2})$. We thus conclude $w \in C^{1,\frac{49}{50}}(\overline{B}_{3R/2})$.

To complete the argument, we apply Theorem 8.33 in \cite{Gilbarg}, which provides an a priori estimate for $\|\widetilde{v}\|$ in $C^{1,\frac{49}{50}}\left(\bar{B}_{3 R / 2}\right)$ in terms of boundary and source term. Combining this with the Sobolev embedding of $H^2 \hookrightarrow L^{\infty}$ in two and three dimensions, and  H\"older estimate \eqref{eq:holder}, we have
\[
\|w\|_{C^{1,\frac{49}{50}}(\bar{B}_{3R/2})}\leq C\left(\|w\|_{H^1(B_{2R})} + \|k^2(1+q)w\|_{L^\infty(B_{2R})}\right) \leq C \|w\|_{H^2(B_{2R})},
\]
where $C = C(R, k, n, \|q\|_{L^\infty})$. Combining this with \eqref{eq:w2RL} yields
\[
\|w\|_{C^{1,\frac{49}{50}}(\bar{B}_{3R/2})} \leq C\|w\|_{L^2(B_{2R})} \leq C(R, k, n, \|q\|_{L^\infty}, \|w^i\|_{L^2}).
\]

The proof is complete.
\end{proof}

\noindent\textbf{Acknowledgment.} 
The work of H. Diao is supported by National Natural Science Foundation of China  (No. 12371422) and the Fundamental Research Funds for the Central Universities, JLU. The work of H. Liu is supported by the Hong Kong RGC General Research Funds (projects 11311122, 11300821, and 11303125), the NSFC/RGC Joint Research Fund (project  N\_CityU101/21), the France-Hong Kong ANR/RGC Joint Research Grant, A-CityU203/19. The work of Q. Meng is fully supported by a fellowship award from the Research Grants Council of the Hong Kong Special Administrative Region, China (Project No. CityU PDFS2324-1S09).

\end{document}